\documentclass[smallextended]{svjour3}       
\usepackage{enumerate}
\usepackage[OT1]{fontenc}
\usepackage[usenames]{color}
\usepackage[colorlinks,
linkcolor=red,
anchorcolor=blue,
citecolor=blue
]{hyperref}
\usepackage{mathrsfs}
\usepackage[square,numbers]{natbib}
\usepackage{hyperref}
\usepackage[protrusion=true, expansion=true]{microtype}
\usepackage{float}
\usepackage{subfigure}
\usepackage{amsfonts,amsmath,amssymb,url,xspace}
\usepackage{tikz}

\usepackage{verbatim}
\usetikzlibrary{arrows,shapes}
\usepackage{mathtools}
\usepackage{authblk}
\usepackage[bottom]{footmisc}
\usepackage{enumitem}
\usepackage{arydshln,leftidx,mathtools}
\usepackage{lscape}
\usepackage{multirow}

\usepackage{algorithm}
\usepackage{algorithmicx,algpseudocode}

\usepackage{chngcntr}
\counterwithout{equation}{section}

\allowdisplaybreaks[1]

\newcommand{\black}{\color{black}}

\newcommand{\mR}{\mathbb{R}}

\newcommand{\mE}{\mathbb{E}}
\newcommand{\mP}{\mathcal{P}}

\newcommand{\mK}{\mathcal{K}}
\newcommand{\mL}{\mathcal{L}}

\newcommand{\mX}{\mathcal{X}}
\newcommand{\mF}{\mathcal{F}}
\newcommand{\mA}{\mathcal{A}}
\newcommand{\mB}{\mathcal{B}}

\def\P{{\mathcal P}}

\def\T{{\mathcal T}}
\def\0{{\boldsymbol 0}}

\def\N{{\mathcal N}}

\def\bv{{\boldsymbol{v}}}
\def\by{{\boldsymbol{y}}}
\def\bz{{\boldsymbol{z}}}
\def\b1{{\boldsymbol{1}}}
\def\bx{{\boldsymbol{x}}}
\def\bu{{\boldsymbol{u}}}

\def\bz{{\boldsymbol{z}}}
\def\blambda{{\boldsymbol{\lambda}}}
\def\tlambda{{\boldsymbol{\lambda}}^\star}

\def\tx{{\boldsymbol{x}}^\star}

\def\tDelta{{\tilde{\Delta}}}
\def\hDelta{{\hat{\Delta}}}

\def\barg{{\bar{g}}}
\def\barf{{\bar{f}}}
\def\barH{{\bar{H}}}
\def\bnabla{{\bar{\nabla}}}
\def\barM{{\bar{M}}}
\def\barT{{\bar{T}}}
\def\barDelta{{\bar{\Delta}}}
\def\tDelta{{\tilde{\Delta}}}
\def\barmu{{\bar{\mu}}}
\def\Mid{{\bigg\vert}}
\def\bartau{{\bar{\tau}}}
\def\barepsilon{{\bar{\epsilon}}}

\def\baralpha{{\bar{\alpha}}}
\def\barL{{\bar{\mL}}}

\def\barK{{\bar{K}}}
\def\tmu{{\tilde{\mu}}}
\def\tdelta{{\tilde{\delta}}}

\mathtoolsset{showonlyrefs=true}

\usepackage{xargs}
\usepackage[colorinlistoftodos,prependcaption,textsize=tiny]{todonotes}
\newcommandx{\unsure}[2][1=]{\todo[linecolor=red,backgroundcolor=red!25,bordercolor=red,#1]{#2}}
\newcommandx{\change}[2][1=]{\todo[linecolor=blue,backgroundcolor=blue!25,bordercolor=blue,#1]{#2}}
\newcommandx{\info}[2][1=]{\todo[linecolor=OliveGreen,backgroundcolor=OliveGreen!25,bordercolor=OliveGreen,#1]{#2}}
\newcommandx{\improvement}[2][1=]{\todo[linecolor=Plum,backgroundcolor=Plum!25,bordercolor=Plum,#1]{#2}}

\allowdisplaybreaks[1]

\newcommand{\rbr}[1]{\left(#1\right)}
\newcommand{\sbr}[1]{\left[#1\right]}
\newcommand{\cbr}[1]{\left\{#1\right\}}
\newcommand{\nbr}[1]{\left\|#1\right\|}
\newcommand{\abr}[1]{\left|#1\right|}

\ifx\QED\undefined
\def\QED{~\rule[-1pt]{5pt}{5pt}\par\medskip}
\fi

\smartqed  
\usepackage{graphicx}
\begin{document}

\title{An Adaptive Stochastic Sequential Quadratic Programming with Differentiable Exact Augmented Lagrangians}

\titlerunning{An Adaptive Stochastic SQP with Differentiable Exact Augmented Lagrangians}

\author{Sen Na \and Mihai Anitescu \and Mladen Kolar}

\authorrunning{Na et al.} 

\institute{Sen Na \at
Department of Statistics, University of California, Berkeley\\
International Computer Science Institute\\
\email{senna@berkeley.edu}
\and
Mihai Anitescu \at
Mathematics and Computer Science Division, Argonne National Laboratory\\
\email{anitescu@mcs.anl.gov}
\and
Mladen Kolar \at
Booth School of Business, University of Chicago\\
\email{mladen.kolar@chicagobooth.edu}
}

\date{Received: date / Accepted: date}

\maketitle

\begin{abstract}

We consider solving nonlinear optimization problems with stochastic objective and deterministic equality constraints. We assume for the objective that its evaluation, gradient, and Hessian are inaccessible, while one can~compute their stochastic estimates by, for example, subsampling. We propose a stochastic algorithm based on sequential quadratic programming (SQP) that uses a differentiable exact augmented Lagrangian as the merit function. To motivate our algorithm design, we first revisit and simplify an old SQP method~\citep{Lucidi1990Recursive} developed for solving deterministic problems, which serves as the skeleton of~our stochastic algorithm. Based on the simplified deterministic algorithm, we then propose a non-adaptive SQP for dealing with stochastic objective, where the gradient and Hessian are replaced by stochastic estimates but the stepsizes are deterministic and prespecified. Finally, we incorporate a recent stochastic line search procedure \citep{Paquette2020Stochastic} into the non-adaptive stochastic SQP to adaptively~select the random stepsizes, which leads to an adaptive~stochastic SQP. The~global ``almost sure" convergence for both non-adaptive and adaptive SQP methods is established. Numerical experiments on nonlinear problems in CUTEst test set demonstrate the superiority of the adaptive algorithm.

\keywords{Stochastic optimization \and Sequential quadratic programming \and Augmented Lagrangian \and Exact penalty}
\end{abstract}

\section{Introduction}\label{sec:1}

We consider the nonlinear equality-constrained stochastic optimization problem
\begin{equation}\label{pro:1}
\begin{aligned}
\min_{\bx \in \mR^{d}}\;\; & f(\bx) = \mE[f(\bx; \xi)],\\
\text{s.t. } & c(\bx) = \0,
\end{aligned}
\end{equation}
where $f: \mR^d\rightarrow \mR$ is the objective, $c:\mR^d\rightarrow\mR^m$ are the constraints, and $\xi\sim \P$ is a random variable following the distribution $\P$. In stochastic optimization, the function $f(\bx)$, gradient $\nabla f(\bx)$, and Hessian $\nabla^2 f(\bx)$ either cannot be~evaluated due to the expectation over $\xi$, or are too expensive to compute, for example,~in big data applications in machine learning. Instead, one can generate either a single sample $\xi\sim \P$ and let $\barf(\bx) = f(\bx; \xi)$, $\bnabla f(\bx) = \nabla f(\bx; \xi)$, $\bnabla^2f(\bx) = \nabla^2f(\bx;\xi)$; or a mini-batch $\xi_{i_1}, \ldots \xi_{i_B} \stackrel{i.i.d}{\sim}\P$~and let $\barf(\bx) = \frac{1}{B}\sum_{j=1}^{B}f(\bx; \xi_{i_j})$, $\bnabla f(\bx) = \frac{1}{B}\sum_{j=1}^{B}\nabla f(\bx; \xi_{i_j})$, $\bnabla^2f(\bx) = \frac{1}{B}\sum_{j=1}^{B}\nabla^2 f(\bx; \xi_{i_j})$ be the~corresponding stochastic estimates at $\bx$.

Problem \eqref{pro:1} prominently appears in machine learning where the objective takes the form
\begin{equation*}
f(\bx) = \int f(\bx; \by, \bz)\; d\P(\by, \bz)
\end{equation*}
with $(\by, \bz)$ representing random input-output pairs with the joint distribution $\P(\by, \bz)$. The loss function $f$ is parameterized by $\bx$, and our goal is~to~obtain the optimal parameter $\bx^\star$ by minimizing the loss function under nonlinear constraints. In practice, the distribution $\P(\by, \bz)$ is unknown and only $n$ data points $\{\xi_i = (\by_i, \bz_i)\}_{i=1}^n$ from $\P(\by, \bz)$ are available. In that case, $\P(\by, \bz)$ is approximated by the empirical distribution and the objective can be written as
\begin{equation}\label{pro:finite}
f(\bx) = \frac{1}{n}\sum_{i=1}^{n}f(\bx; \xi_i) = \mE\sbr{f(\bx; \xi)},
\end{equation}
where $\xi\sim \text{Uniform}\rbr{\cbr{\xi_i}_{i=1}^n}$ follows uniform distribution over $n$ discrete data points. The finite-sum loss function above is called the empirical risk. For~example, recent studies on constrained deep neural networks \cite{Nandwani2019Primal, Ravi2019Explicitly} are in the form of \eqref{pro:finite}.

Problem \eqref{pro:1} also appears widely in statistics as a way to estimate model parameters that best fit the data. The constraint set encodes prior information or physical constraints on the parameters. In applications of maximum likelihood estimation (MLE) under constraints, $f(\bx, \xi)$ is a negative log-likelihood function evaluated at one data point. For example, \cite{Nagaraj1991Estimation} studied least square problems under nonlinear equality constraints; \cite{Dupacova1988Asymptotic} studied constrained MLE; and \cite{Shapiro2000asymptotics} studied constrained $M$-estimation. These references established consistency and asymptotic normality for the optimizer $\bx^\star$ of \eqref{pro:1}. See \cite{Prekopa1973Contributions, Wets1983Stochastic} for more statistical background on the problem.

While there is a rich history on statistical properties of constrained estimators, numerical procedures for solving \eqref{pro:1} are less well developed. When constraints are characterized by an abstract set, projected stochastic first- and second-order methods have been developed \citep{Nemirovski2009Robust, Bertsekas1982Projected}. However, these projection-based methods are not applicable for \eqref{pro:1}, since the projection to the null~space of nonlinear equations may not be easily computed. To our knowledge, only recently \cite{Berahas2021Sequential} proposed the very first practical algorithm for solving \eqref{pro:1}, which is a stochastic sequential quadratic programming (StoSQP) with $\ell_1$ penalized merit function. The method of \cite{Berahas2021Sequential} possesses the following three main components:
\begin{enumerate}[label=(\alph*),topsep=0pt]	\setlength\itemsep{0.0em}

\item \cite{Berahas2021Sequential} designed a novel stepsize selection scheme, where the stepsize $\bar{\alpha}_k$ is set by projecting a quantity to an interval. The quantity and the interval boundaries depend on the magnitude of the search direction, the Lipschitz constant of the merit function, and a prespecified deterministic sequence $\{\beta_k\}_k$. Such a stepsize selection scheme is designed to introduce certain adaptivity into the algorithm without involving a line search. The sequence $\{\beta_k\}_k$, which is either constant or decaying with~a~proper~rate,~determines~the iteration convergence behavior, however, still heavily affects the selection~of~$\baralpha_k$.~As shown numerically in \cite{Berahas2021Sequential}, the line search procedure is still preferable for most of problems.

\item The penalty parameter of the merit function in each iteration is adaptively selected in \cite{Berahas2021Sequential}, which can be viewed as a random walk. To~establish~convergence, it is important to show two properties for the underlying random~walk. (i) It stabilizes after a number of iterations, so that one can accumulate the~descent of the merit function in each iteration. (ii) The stabilized penalty value is less (or greater depending on the definition) than a \textit{deterministic} threshold, so that SQP iterates indeed converge to a KKT point instead of a stationary point of the merit function. \cite{Berahas2021Sequential} showed (i) under a boundedness condition (cf. Proposition 3.18), which is a standard condition even in deterministic setting (cf. \cite[Chapter 4]{Bertsekas1982Constrained}). However, (ii) is more subtle for StoSQP, where the stabilized value itself is random and different in each run. To this end, \cite{Berahas2021Sequential} showed (ii) under additional assumptions on the noise distribution (e.g Gaussian). See Proposition 3.16 and Example 3.17 therein.

\item \cite{Berahas2021Sequential} employed an $\ell_1$ penalized merit function. As noticed for deterministic problems, differentiable merit functions are alternative choices that may enjoy some local benefits, such as overcoming the Maratos effect \citep{Maratos1978Exact}.~Although local analysis of StoSQP is not a goal of the paper, and it~is~unclear whether the Maratos effect is a serious issue on stochastic problems (given the randomness of objective estimation), it still stimulates interest in designing a StoSQP with an alternative, differentiable merit function.

\end{enumerate}

\vskip2pt
\noindent This paper provides different resolutions on the above three components to~derive an adaptive StoSQP. For (a), we study a different setup compared to~\cite{Berahas2021Sequential}. Under the setup, we are allowed to generate mini batches to have a more~precise objective approximation, which further allows~us to adaptively select the stepsize via stochastic line~search. To this end, we generalize the stochastic~line search procedure in \cite{Paquette2020Stochastic} to equality-constrained problems under constraint qualifications. The generalization is non-trivial since, due to the constraints,~we have to use the merit function with varying stochastic penalty parameters~in~the Armijo condition, while \cite{Paquette2020Stochastic} used the fixed objective function. With the help of the line search, we do not have to prespecify a sequence to control the stepsize. For (c), we employ a differentiable exact augmented Lagrangian merit~function. Comparing with a standard design of using the $\ell_1$ merit function in the line search, our experiments provide some evidence for supporting the usage~of~the~augmented Lagrangian. Although our experiments only investigate the global~convergence so that the improvement of the augmented Lagrangian is not so significant,~and we do not attempt to claim benefits of the augmented Lagrangian over the $\ell_1$ merit function under the present global analysis, our paper indeed serves as~the first step towards understanding the local convergence of differentiable merit functions for solving constrained stochastic problems. Finally, for (b)(i), we consider a similar setup to \cite{Berahas2021Sequential}, and show that the penalty parameter is stabilized when the estimation errors are bounded in each iteration (cf. Assumption \ref{ass:A:2}). However, for (b)(ii), we adopt a different approach. We modify the SQP scheme when specifying the penalty parameter. In particular, we accept the penalty parameter only if a more stringent condition is satisfied, \textit{where the feasibility error is bounded by the gradient magnitude of the augmented Lagrangian}. By such modification, we ensure the convergence to a KKT point without relying on the noise distribution. With all above differences, we establish the global convergence of the proposed StoSQP by showing that $\liminf_{k\rightarrow \infty}\|\nabla\mL(\bx_k, \blambda_k)\| = 0$~\textit{almost surely}, where $\mL(\bx, \blambda) = f(\bx)+\blambda^Tc(\bx)$ is the Lagrange function with $\blambda\in \mR^m$ being the dual variables associated with the constraints; and $(\bx_k, \blambda_k)$ is the $k$-th primal-dual iterate. This type of result is also different from \cite{Berahas2021Sequential}, which shows $\liminf_{k\rightarrow \infty}\mE[\|\nabla\mL(\bx_k, \blambda_k)\|] =0$.

\subsection{Literature review}

Problem \eqref{pro:1} is closely related to two classes of problems: constrained deterministic optimization and unconstrained stochastic optimization, both of which have been extensively investigated. We review some related literature.

There exist numerous methods for solving constrained deterministic problems, including exact penalty methods, augmented Lagrangian methods, and~sequential quadratic programming (SQP) methods \citep{Nocedal2006Numerical}. This paper utilizes SQP schemes to solve stochastic problems. SQP schemes apply Newton's method on KKT equations of \eqref{pro:1}. Within a SQP scheme, an exact penalty function is used in the line search to monitor the iteration progress towards a KKT point.~One of the advantages of SQP is that Newton's system preserves the structure of the objective, which contrasts with exact penalty method and augmented Lagrangian method. The search direction of the latter two methods is derived by minimizing a~penalized objective, and the penalty term may affect the~problem structure. See \cite{Gill2005SNOPT, Gill2011Sequential} and references therein for a brief~review~of~SQP.

The present paper contributes to the existing literature on SQP by studying its global convergence for solving stochastic optimization problems. \cite{Berahas2021Sequential} designed a StoSQP using the exact penalty function $f(\bx) + \mu\|c(\bx)\|_1$. We instead use $\mL(\bx, \blambda) + \mu/2\cdot \|\nabla_{\blambda}\mL(\bx, \blambda)\|_2^2 + \nu/2\cdot \|\nabla^Tc(\bx)\nabla_{\bx}\mL(\bx, \blambda)\|_2^2$ with $\mu, \nu>0$ being the penalty parameters. This function is defined on the primal-dual~pair $(\bx, \blambda)$ and called exact augmented Lagrangian \citep{Pillo1994Exact}. The exactness of such~an augmented Lagrangian has been studied in \cite{Pillo1979New,	Pillo1980method}. \cite{Lucidi1990Recursive} employed it in a~SQP framework for deterministic problems, which we will revisit and modify in Section \ref{sec:2}. The modified algorithm serves as the skeleton of our StoSQP designed in Sections \ref{sec:3} and \ref{sec:4}.

Without constraints, numerous first- and second-order methods have also been reported for solving stochastic problems. We refer to \cite{Bottou2018Optimization, Berahas2020investigation} for a recent review. One practical concern for stochastic methods is the adaptivity of the stepsize. A growing body of literature has generalized the natural line search for gradient descent to stochastic line search by dynamically controlling the sample size for computing the gradient estimates \citep{Friedlander2012Hybrid, Byrd2012Sample, Krejic2013Line, De2017Automated, Bollapragada2018Adaptive}. However, most of works study convex problems, do not provide convergence guarantees for the line search, and use sample size selection strategies that depend on unknown quantities such as the true gradient. Practical stochastic stepsize selection methods with convergence guarantees have been developed only recently. \cite{Bandeira2014Convergence, Chen2017Stochastic, Gratton2017Complexity, Curtis2019Stochastic, Blanchet2019Convergence, Curtis2020fully} studied stochastic trust-region methods, where stepsizes are absorbed in local quadratic models. \cite{Cartis2017Global} established the expected iteration complexity of stochastic line search for nonconvex objectives, while required function evaluations to be computed exactly. \cite{Serafino2020LSOS} allowed for stochastic evaluations, but only proved convergence for strongly convex objectives. Finally, \cite{Paquette2020Stochastic} generalized the analysis in \cite{Cartis2017Global} and \cite{Serafino2020LSOS} by applying the martingale framework in \cite{Blanchet2019Convergence}, and obtained convergence guarantees for stochastic line search with both convex and nonconvex objectives.

Our paper generalizes \cite{Paquette2020Stochastic} to equality-constrained problems under constraint qualifications. As discussed earlier, the generalization is not immediate since the merit function has varying stochastic penalty parameters. In addition, in unconstrained problems, if one observes a sufficient decrease of the objective in each iteration, then one can claim that the iterates converge to a stationary point (or a minimizer if the objective is convex). Differently, in constrained problems, if one observes a sufficient decrease of the merit function in each iteration, then one can \textit{only} claim that the iterates converge to a stationary point of the merit function. Only if the penalty parameter of the exact penalty merit function is less than a critical threshold can one claim the iterates indeed converge to a KKT point. In stochastic optimization, the problem is even more challenging: one has to show that the stabilized penalty parameter, a random quantity, is less than the deterministic threshold to ensure convergence. \cite[Section 3.2.2]{Berahas2021Sequential} discussed this subtlety and resolved it for some noise distributions (e.g. Gaussian), while, instead, we achieve convergence to~a~KKT~point~by~modifying a SQP scheme and imposing a more stringent condition when selecting the penalty parameter. Our analysis does not rely on the symmetry of the noise~distributions (cf. \cite[Example 3.17]{Berahas2021Sequential}).

\vskip5pt
\noindent\textbf{Structure of the paper:} We revisit and simplify an old deterministic SQP scheme \cite{Lucidi1990Recursive} in Section \ref{sec:2}.~Based on the simplified SQP scheme, a non-adaptive StoSQP is designed in Section \ref{sec:3}, and an adaptive StoSQP that utilizes a~stochastic line search is designed in Section \ref{sec:4}. Numerical experiments and conclusions are presented in Sections \ref{sec:5} and \ref{sec:6}. Due to space limit, some technical proofs are deferred to the appendix.

\vskip5pt

\noindent\textbf{Notation:} We use $\|\cdot\|$ to denote the $\ell_2$ norm for vectors and the spectral norm for matrices. $I$ denotes the identity matrix and $\0$ denotes the zero matrix or vector, whose dimensions are clear from the context. For scalars $a$ and $b$, $a\vee b = \max(a, b)$ and $a\wedge b = \min(a, b)$. For random variables $\xi_1$, $\xi_2$, we denote $\mE_{\xi_1}[g(\xi_1, \xi_2)] = \mE[g(\xi_1, \xi_2)\mid \xi_2]$ to be the conditional expectation, that is, the expectation taken over randomness in $\xi_1$ only.

\section{SQP with Differentiable Exact Augmented Lagrangian}\label{sec:2}

In this section, we assume $f$ is deterministic. We revisit and simplify the~SQP scheme proposed by \cite{Lucidi1990Recursive} to facilitate the later design of StoSQP. Let $\mL(\bx, \blambda) = f(\bx) + \blambda^Tc(\bx)$ be the Lagrangian function of \eqref{pro:1}. {We aim at finding} a KKT point $(\tx, \tlambda)$ that satisfies
\begin{equation}\label{equ:KKT}
\begin{pmatrix}
\nabla_{\bx}\mL(\tx, \tlambda)\\
\nabla_{\blambda}\mL(\tx, \tlambda)
\end{pmatrix} = \begin{pmatrix}
\nabla f(\tx) + G^T(\tx)\tlambda\\
c(\tx)
\end{pmatrix} = \begin{pmatrix}
\0\\
\0
\end{pmatrix},
\end{equation}
where $G(\bx) = \nabla^T c(\bx) = (\nabla c_1(\bx), \ldots, \nabla c_m(\bx))^T \in \mR^{m\times d}$ is the Jacobian~matrix. The differentiable exact augmented Lagrangian considered in this~paper takes the form
\begin{equation}\label{equ:augmented:L}
\mL_{\mu, \nu}(\bx, \blambda) = \mL(\bx, \blambda) + \frac{\mu}{2}\|c(\bx)\|^2 + \frac{\nu}{2}\|G(\bx)\nabla_{\bx}\mL(\bx, \blambda)\|^2.
\end{equation}
The first penalty characterizes the feasibility error, which is a quadratic penalty on $\nabla_{\blambda}\mL(\bx, \blambda)$. The second penalty characterizes the optimality error, where $\nabla_{\bx}\mL(\bx, \blambda)$ is transformed by the Jacobian $G(\bx)$. We notice that, even without transformation $G(\bx)$, the function \eqref{equ:augmented:L} is still an exact augmented Lagrangian, provided $\mu$ is sufficiently large and $\nu$ is sufficiently small \citep[Proposition 4.15]{Bertsekas1982Constrained}. On the other hand, $G(\bx)$ can be replaced by other suitable weight matrices that satisfy certain conditions. See \cite{Pillo1979New} for other choices and their comparisons in experiments. Our analysis can be easily extended to other forms of augmented Lagrangian. We use $G(\bx)$ for transformation as it is simpler than~other~choices, especially when computing the gradient $\nabla\mL_{\mu, \nu}$; and, unlike without transformation, $\nu>0$ in our analysis can be any positive parameter and need not be adapted. The gradient of $\mL_{\mu, \nu}(\bx, \blambda)$ is
\begin{align}\label{equ:derivative:AL}
\begin{pmatrix}
\nabla_{\bx}\mL_{\mu, \nu}(\bx, \blambda)\\
\nabla_{\blambda}\mL_{\mu, \nu}(\bx, \blambda)
\end{pmatrix} = \begin{pmatrix}
I + \nu M(\bx, \blambda)G(\bx) & \mu G^T(\bx)\\
\nu G(\bx)G^T(\bx)G(\bx) & I
\end{pmatrix}\begin{pmatrix}
\nabla_{\bx}\mL(\bx, \blambda)\\
\nabla_{\blambda}\mL(\bx, \blambda)
\end{pmatrix}
\end{align}
\vskip-2pt
\noindent where
\begin{equation*}
\begin{aligned}
M(\bx, \blambda) = & \nabla_{\bx}(G(\bx)\nabla_{\bx}\mL(\bx, \blambda)) = \nabla_{\bx}^2\mL(\bx, \blambda)G^T(\bx) + T(\bx, \blambda) \in \mR^{d\times m},\\
T(\bx,\blambda) =& \rbr{\nabla^2c_1(\bx)\nabla_{\bx}\mL(\bx, \blambda), \ldots, \nabla^2c_m(\bx)\nabla_{\bx}\mL(\bx, \blambda)}.
\end{aligned}
\end{equation*}

We now present a SQP scheme. Given a pair $(\bx_k, \blambda_k)$ at the $k$-th iteration, we denote $f_k = f(\bx_k)$, $\nabla f_k = \nabla f(\bx_k)$ (similar for $G_k, M_k,\nabla_{\bx}\mL_k$ etc.) to~be the quantities evaluated at the $k$-th iterate. We then compute the search direction $(\Delta\bx_k, \Delta\blambda_k)$ by solving
\begin{equation}\label{equ:Newton:AL}
\begin{aligned}
\begin{pmatrix}
B_k & G_k^T\\
G_k & \0
\end{pmatrix} \begin{pmatrix}
\Delta\bx_k\\
\hat{\Delta}\blambda_k
\end{pmatrix}& =-\begin{pmatrix}
\nabla_{\bx}\mL_k\\
c_k
\end{pmatrix},\\
G_kG_k^T\Delta\blambda_k & = - (G_k\nabla_{\bx}\mL_k + M^T_k\Delta\bx_k),
\end{aligned}
\end{equation}
where $B_k$ is an approximation of the Lagrangian Hessian $\nabla_{\bx}^2\mL_k$. The next iterate is then
\begin{equation}\label{equ:update}
\begin{pmatrix}
\bx_{k+1}\\
\blambda_{k+1}
\end{pmatrix} = \begin{pmatrix}
\bx_k\\
\blambda_k
\end{pmatrix} + \alpha_k\begin{pmatrix}
\Delta\bx_k\\
\Delta\blambda_k
\end{pmatrix},
\end{equation}
where the stepsize $\alpha_k$ is chosen to make the Armijo condition hold
\begin{equation}\label{equ:Armijo}
\mL_{\mu, \nu}^{k+1} \leq \mL_{\mu, \nu}^k + \alpha_k\beta\begin{pmatrix}
\nabla_{\bx}\mL_{\mu, \nu}^k\\
\nabla_{\blambda}\mL_{\mu, \nu}^k
\end{pmatrix}^T \begin{pmatrix}
\Delta\bx_k\\
\Delta\blambda_k
\end{pmatrix}
\end{equation}
for a prespecified $\beta\in(0, 1)$. 

\begin{remark}
We should mention that the dual search direction in this scheme is $\Delta\blambda_k$, not $\hDelta\blambda_k$ like in most of SQP schemes. This difference is driven by the penalty term $\|G(\bx)\nabla_{\bx}\mL(\bx, \blambda)\|^2$ in \eqref{equ:augmented:L}. In fact, if $B_k \approx \nabla_{\bx}^2\mL_k$ and $(\bx_k,\blambda_k)$ is close to a KKT point, $\Delta\blambda_k\approx\hDelta\blambda_k$ by noting that $M_k\approx \nabla_{\bx}^2\mL_kG_k^T$. However, for an iterate that is far from a KKT point or $B_k \not\approx \nabla_{\bx}^2\mL_k$ (which is the case in this paper), $\Delta\blambda_k$ and $\hDelta\blambda_k$ are significantly different and such~a~modification~is essential. For a SQP scheme that uses the merit function $\mL_{\mu, \nu}$, we desire a dual direction that, together with $\Delta\bx_k$, is a descent direction of $\mL_{\mu, \nu}^k$ for any~$\nu>0$, as long as $\mu>0$ is sufficiently large. A sufficient (may not be~necessary)~condition for such a dual direction, proposed by \cite[(8)]{Lucidi1990Recursive}, is
\begin{equation*}
\lim\limits_{\alpha\rightarrow 0}G_kG_k^T\Delta\blambda_k(\alpha) = -(G_k\nabla_{\bx}\mL_k + M^T_k\Delta\bx_k),
\end{equation*}	
where $\Delta\blambda_k(\alpha)$ can depend on the stepsize $\alpha$ selected by the line search.~This paper removes the dependence on $\alpha$ and directly solves $\Delta\blambda_k$ from the above equation. We note, however, that $\hDelta\blambda_k$ does not satisfy this condition. For~a~general penalty term $\|W(\bx)\nabla_{\bx}\mL(\bx, \blambda)\|^2$ with a differentiable weight matrix~$W(\bx)\in \mR^{m\times d}$, we let $S(\bx, \blambda) = \nabla_{\bx}\rbr{W(\bx)\nabla_{\bx}\mL(\bx, \blambda)}\in \mR^{d\times m}$, and the above condition can be adapted as
\begin{equation*}
\lim\limits_{\alpha\rightarrow 0} W_kG_k^T\Delta\blambda_k(\alpha) = - (W_k\nabla_{\bx}\mL_k + S_k^T\Delta\bx_k).
\end{equation*}
By the exact penalty analysis (cf. \cite[Proposition 3]{Pillo1979New}), $W(\bx)G^T(\bx)$ is assumed to be invertible.

\end{remark}

The downside of the above scheme is that the penalty parameter $\mu$ is fixed. To prove convergence, $\mu$ needs to be large enough so that the square matrix in \eqref{equ:derivative:AL} is invertible. Then $\|\nabla\mL_{\mu, \nu}^k\|\rightarrow 0$ implies $\|\nabla\mL_k\|\rightarrow 0$. The threshold depends on multiple unknown quantities of the problem and is hard to tune manually. Thus, we propose an adaptive SQP method in Algorithm~\ref{alg:SQP} to resolve this issue. The adaptivity is achieved by the While loop in Line 4. In the While loop, we iteratively check if the inner product of the search direction $(\Delta\bx_k, \Delta\blambda_k)$ and the gradient $\nabla \mL_{\mu_k, \nu}^k$ is upper bounded by $-\delta_k\rbr{\|\Delta\bx_k\|^2 +  \|G_k\nabla_{\bx}\mL_k\|^2}$. If not, we increase $\mu_k$ and decrease $\delta_k$ until the condition is satisfied. This parameter selection scheme is also used in our later design of StoSQP, although there we require an additional condition to be satisfied for $\mu_k$.

\begin{algorithm}[!tp]
\caption{{\black SQP with Differentiable Exact Augmented Lagrangian}}	\label{alg:SQP}
\begin{algorithmic}[1]
\State \textbf{Input:} initial iterate $(\bx_0, \blambda_0)$, scalars $\nu, \mu_0, \delta_0>0$, $\rho>1$, $\beta\in(0, 1)$;
\For{$k = 0,1,2\ldots$}
\State Compute $(\Delta\bx_k, \Delta\blambda_k)$ by \eqref{equ:Newton:AL};
\While{$\begin{pmatrix}
\nabla_{\bx}\mL_{\mu_k, \nu}^k\\
\nabla_{\blambda}\mL_{\mu_k, \nu}^k
\end{pmatrix}^T \begin{pmatrix}
\Delta\bx_k\\
\Delta\blambda_k
\end{pmatrix} > -\delta_k \nbr{\begin{pmatrix}
\Delta \bx_k\\
G_k\nabla_{\bx}\mL_k
\end{pmatrix}}^2$}
\State Let $\mu_k = \rho \mu_k$, $\delta_k = \delta_k/\rho$;
\EndWhile
\State Select $\alpha_k$ to satisfy the Armijo condition in \eqref{equ:Armijo} with parameters $\mu_k$, $\nu$, and $\beta$;
\State Update the iterate by \eqref{equ:update};
\State Let $\mu_{k+1} = \mu_k$, $\delta_{k+1} = \delta_k$;
\EndFor
	\end{algorithmic}
\end{algorithm}

The global convergence of Algorithm \ref{alg:SQP} can be established in a similar way~to \cite{Lucidi1990Recursive}. We do not go over the analysis in this paper, but mainly focus on analyzing a stochastic version of it. However, we should mention that Algorithm \ref{alg:SQP} adopts two simplifications from \cite{Lucidi1990Recursive} as listed below.
\begin{enumerate}[label=(\alph*),topsep=0pt]	\setlength\itemsep{0.0em}
\item \cite{Lucidi1990Recursive} applied finite differences to approximate $M_k = \nabla_{\bx}\rbr{G_k\nabla_{\bx}\mL_k}$ in \eqref{equ:Newton:AL}, while we assume that second-order matrices are computable to simplify the presentation. This allows us to compute the Lagrangian Hessian $\nabla_{\bx}^2\mL$, the constraint Hessian $\nabla^2c_i$, $\forall i$, and hence $M_k$.

\item \cite{Lucidi1990Recursive} used a more complex While loop to select $\mu_k$. That update of $\mu_k$ depends on the iterate (cf. \cite[(13)]{Lucidi1990Recursive}). In the end, the complex While loop condition implies that the inner product of the search direction and~the~gradient $\nabla\mL_{\mu_k, \nu}^k$ is bounded by $-\delta(\|\Delta\bx_k\|^2 + \|G_k\nabla_{\bx}\mL_k\|^2)$ for some $\delta>0$ (cf. \cite[Proposition 4.2]{Lucidi1990Recursive}), which is what the authors finally used in the analysis.~We simplify their While loop by directly enforcing the final condition, and increasing $\mu_k$ by a fixed factor $\rho>1$ each time (Line 5 in Algorithm \ref{alg:SQP}). 

\end{enumerate}

Algorithm \ref{alg:SQP} is a framework that shows how to design an adaptive SQP using differentiable exact augmented Lagrangian as the merit function. Using the constructed SQP as a skeleton, we design a non-adaptive and adaptive StoSQP scheme in Sections \ref{sec:3} and \ref{sec:4}, respectively.

\section{Warm-up: A Non-Adaptive Stochastic SQP}\label{sec:3}

Before designing an adaptive StoSQP in Section \ref{sec:4}, we design~a non-adaptive StoSQP in this section. It is a straightforward counterpart of Algorithm \ref{alg:SQP},~except that the deterministic gradient and Hessian of $f$ are replaced by their stochastic estimates, and the line search is replaced by a prespecified sequence of stepsizes $\{\alpha_k\}_k$. Since no line search is used, the merit function \eqref{equ:augmented:L} is not involved in the algorithm presentation, and is only used to analyze the convergence.

Suppose $f$ is stochastic, and its function value, gradient, and Hessian are~not accessible. Instead, we have access to their estimates $\barf$, $\barg$, and $\barH$. In what follows, we use $\bar{(\cdot)}$ to denote random quantities, except for the iterate $(\bx_k, \blambda_k)$. Let $\xi^k_g, \xi^k_H$ be two independent realizations of the random variable $\xi$, and denote
\begin{align}\label{def:MT}
\barg_k & =  \nabla f(\bx_k; \xi^k_g), \hskip0.4cm \bnabla_{\bx}\mL_k =  \bnabla_{\bx}\mL(\bx_k, \blambda_k; \xi^k_g) = \barg_k + G^T_k\blambda_k, \nonumber\\
\barH_k & = \nabla^2f(\bx_k; \xi^k_H), \hskip0.15cm \bnabla_{\bx}^2\mL_k =  \bnabla_{\bx}^2\mL(\bx_k, \blambda_k; \xi^k_H) = \barH_k + \sum_{j=1}^m(\blambda_k)_j\nabla^2c_j(\bx_k), \nonumber\\
\barT_k & = \rbr{
\nabla^2c_1(\bx_k)\bnabla_{\bx}\mL(\bx_k, \blambda_k; \xi^k_H), \cdots , \nabla^2c_m(\bx_k)\bnabla_{\bx}\mL(\bx_k, \blambda_k; \xi^k_H)}, \\
\barM_k & = \bnabla_{\bx}^2\mL_kG^T_k + \barT_k. \nonumber
\end{align}
In particular, $\xi_g^k$ is used for estimating the gradient $\bnabla_{\bx}\mL_k$, while $\xi_H^k$ is used~for estimating the second-order derivatives, including $\bnabla_{\bx}^2\mL_k$ and $\barM_k$. Since $\xi^k_g$ and $\xi^k_H$ are independent, $\barM_k$ and $\bnabla_{\bx}\mL_k$ are also independent. The subscript $k$ for both stochastic and deterministic quantities corresponds to the evaluation at the $k$-th iterate $(\bx_k, \blambda_k)$.

\begin{algorithm}[!tp]
\caption{A Non-Adaptive StoSQP}	\label{alg:NSto:SQP}
\begin{algorithmic}[1]
\State \textbf{Input:} initial iterate $(\bx_0, \blambda_0)$, stepsizes $\{\alpha_k\}_k$;
\For{$k = 0,1,2\ldots$}
\State Generate independent realizations $\xi_g^k$ and $\xi_H^k$, and compute $\bnabla_{\bx}\mL_k$ and $\barM_k$ as in \eqref{def:MT};
\State Compute $(\barDelta\bx_k, \barDelta\blambda_k)$ by \eqref{equ:ran:Newton};
\State Update the iterate by \eqref{equ:ran:update};
\EndFor
\end{algorithmic}
\end{algorithm}

Given the iterate $(\bx_k, \blambda_k)$, the non-adaptive StoSQP generates $\xi^k = (\xi_g^k, \xi_H^k)$ and computes $\bnabla_{\bx}\mL_k$ and $\barM_k$ as in \eqref{def:MT}. Then, the stochastic search direction $(\barDelta\bx_k, \barDelta\blambda_k)$ is given similarly to \eqref{equ:Newton:AL} by
\begin{equation}\label{equ:ran:Newton}
\begin{aligned}
\begin{pmatrix}
B_k & G^T_k\\
G_k & \0
\end{pmatrix}\begin{pmatrix}
\barDelta \bx_k\\
\tDelta\blambda_k
\end{pmatrix} & = -\begin{pmatrix}
\bnabla_{\bx}\mL_k\\
c_k
\end{pmatrix},\\
G_kG_k^T\barDelta\blambda_k & = -(G_k\bnabla_{\bx}\mL_k + \barM^T_k\barDelta\bx_k).
\end{aligned}
\end{equation}
Same as \cite{Berahas2021Sequential}, we suppose $B_k$ is deterministic given $(\bx_k, \blambda_k)$. In practice, $B_k = I$ is sufficient for global convergence. Then, for a prespecified stepsize $\alpha_k$, we let
\begin{equation}\label{equ:ran:update}
\begin{pmatrix}
\bx_{k+1}\\
\blambda_{k+1}
\end{pmatrix} = \begin{pmatrix}
\bx_k\\
\blambda_k
\end{pmatrix} + \alpha_k\begin{pmatrix}
\barDelta\bx_k\\
\barDelta\blambda_k
\end{pmatrix}.
\end{equation}
The procedure is summarized in Algorithm \ref{alg:NSto:SQP}.

\vskip 2pt

\noindent \textbf{Relationship to StoSQP in \cite{Berahas2021Sequential}:} The non-adaptive StoSQP is applicable~under the \textit{fully} stochastic setup as in \cite{Berahas2021Sequential}, where the derivatives of $f$ are estimated with constant variance. This setup contrasts with the line search setup in Section \ref{sec:4}, where we will require a more precise model approximation. \cite{Berahas2021Sequential} has~a smaller computational cost\footnote{The implicit cost for deriving Lipschitz constants sequence in \cite{Berahas2021Sequential} is not counted here, which, however, requires non-negligible evaluations for the objective and constraints.}: in each iteration, it generates one sample, and evaluates one gradient for the objective, and one function value and one Jacobian for~the~constraints. In contrast, Algorithm \ref{alg:NSto:SQP} generates two independent samples, and evaluates two gradients and one Hessian for the objective, one function value, one Jacobian, and one Hessian for the constraints. Thus, Algorithm \ref{alg:NSto:SQP} is at least twice as expensive~as \cite{Berahas2021Sequential}. The additional cost arises from the usage of~the augmented Lagrangian merit function; we note that Newton's system \eqref{equ:ran:Newton} contains~second-order~information that is not involved in \cite{Berahas2021Sequential}. Requiring the Hessian information is a drawback of using the augmented Lagrangian if only global convergence is of interest. The independence between estimating the gradient and Hessian is not ideal, but common in the literature~\citep{Bollapragada2018Exact}.

The basic design of Algorithm \ref{alg:NSto:SQP} is not as practical as \cite{Berahas2021Sequential}, since the stepsize sequence $\{\alpha_k\}_k$ is prespecified without any adjustment based on the iterate. \cite{Berahas2021Sequential} designed a novel scheme for selecting the stepsize, which introduces certain adaptivity into the StoSQP algorithm without relying on the line search. However, that scheme still requires another prespecified sequence $\{\beta_k\}_k$ to control the stepsize. As revealed in experiments in Section \ref{sec:5}, Algorithm \ref{alg:NSto:SQP} with a simple setup of $\alpha_k$ may outperform \cite{Berahas2021Sequential} in some cases, especially when~$\alpha_k, \beta_k$~are~decaying. This observation suggests that the prespecified sequence in both algorithms highly affects the performance.

This paper addresses the adaptivity under a different, more restrictive setup. We suppose that the estimation variance can be diminished by generating more samples; and this setup allows us to adopt (natural) stochastic line search~to select the stepsize. Before designing an adaptive StoSQP, we establish the~global ``almost sure" convergence of Algorithm \ref{alg:NSto:SQP}. The ``almost~sure" type of convergence guarantee is consistent with our later analysis of adaptive StoSQP, and differs from \cite{Berahas2021Sequential}, which established the convergence in expectation.

\subsection{Convergence Analysis}

We show the convergence of Algorithm \ref{alg:NSto:SQP} for two different stepsize sequences: a constant sequence and a decaying sequence. We need the following assumptions.

\begin{assumption}\label{ass:ran:1}
The iterates $(\bx_k, \blambda_k)$ and trial points $(\bx_k+\alpha_k\barDelta\bx_k, \blambda_k+\alpha_k\barDelta\blambda_k)$ are contained in a convex compact set $\mX\times \Lambda$. The functions $f$ and $c$ are thrice continuously differentiable over $\mX$; and the Jacobian $G(\bx) = \nabla^Tc(\bx)$ has full row rank over $\mX$. The sequence $\cbr{B_k}_k$ satisfies $\bx^TB_k\bx \geq \gamma_{RH} \|\bx\|^2$~for any $\bx\in \{\bx: G_k\bx = \0, \bx \neq \0\}$, and $\|B_k\|\leq \kappa_B$ for constants $0<\gamma_{RH}\leq 1\leq\kappa_B$.

\end{assumption}

\begin{assumption}\label{ass:ran:2}
We assume $\nabla f(\bx_k; \xi)$ and $\nabla^2 f(\bx_k; \xi)$ are unbiased estimators of $\nabla f_k$ and $\nabla^2f_k$ with bounded variance. That is, for a constant $\psi>0$,
\begin{equation}\label{Bound:var}
\begin{aligned}
\mE_{\xi}[\nabla f(\bx_k; \xi)] = &\nabla f_k, \quad \;\;\;\;\;\mE_{\xi}[\|\nabla f(\bx_k; \xi) - \nabla f_k\|^2] \leq \psi,\\
\mE_{\xi}[\nabla^2 f(\bx_k; \xi)] =& \nabla^2f_k, \quad
\mE_{\xi}[\|\nabla^2 f(\bx_k; \xi) - \nabla^2 f_k\|^2] \leq \psi.
\end{aligned}
\end{equation}
\end{assumption}

Assumption \ref{ass:ran:1} is standard in SQP literature \citep{Bertsekas1982Constrained, Nocedal2006Numerical}. In particular,~the convex compactness condition, as assumed, for example, in \cite[Theorem 18.3]{Nocedal2006Numerical} and \cite[Proposition 4.15]{Bertsekas1982Constrained}, ensures that all continuous functions such as $\mL_{\mu, \nu}$, $f$, $c$~are bounded over $\mX\times\Lambda$ (or over $\mX$); and a linear combination of two iterates~still lies in the set. We note that \cite{Berahas2021Sequential} assumed a convex open set, while the boundedness of $f$, $c$ was assumed additionally. We prefer to assume a~compact set, which is more common for studying exact augmented Lagrangian \cite{Pillo1980method}, \cite[Chapter 4.3]{Bertsekas1982Constrained}.

We also assume $f, c$ are thrice continuously differentiable. This is standard in the literature on augmented Lagrangian (cf. \cite[Proposition 4.16]{Bertsekas1982Constrained} and \cite[Section 3]{Pillo1980method}), since the third derivatives are needed for the existence of the~Hessian $\nabla^2\mL_{\mu, \nu}$. This condition ensures that the second-order Taylor expansion of $\mL_{\mu, \nu}$ is valid and $\nabla^2\mL_{\mu, \nu}$ is continuous and hence bounded in $\mX\times\Lambda$. It can be replaced by assuming Lipschitz continuity on $\nabla\mL_{\mu, \nu}$, equivalently, on $\nabla^2f$ and $\nabla^2c$, which is stronger than \cite{Berahas2021Sequential} where Lipschitz continuity on $\nabla f$ and $\nabla c$ was assumed. Fortunately, the third derivatives are required only for the analysis, and never computed in the implementation.

We also assume that the Jacobian $G(\bx)$ has full row rank, and $B_k$ is positive definite in the null space of $G_k$. These conditions are same as \cite{Berahas2021Sequential} and other SQP literature \cite{Lucidi1990Recursive}, and ensure that matrices in \eqref{equ:ran:Newton} are invertible (cf. \cite[Lemma 16.1]{Nocedal2006Numerical}). Thus, Algorithm \ref{alg:NSto:SQP} is well-posed. The range of positive constants $\gamma_{RH}$, $\kappa_B$ (and also the range of other positive constants defined later) is not crucial to the analysis; all results hold by replacing $\gamma_{RH}$ by $\gamma_{RH}\wedge 1$ and $\kappa_B$ by $\kappa_B\vee 1$ if we enforce $\gamma_{RH}, \kappa_B>0$ only.

Assumption \ref{ass:ran:2} assumes that the gradient and Hessian estimators have~bounded variance, which is similar to the fully stochastic setup in \cite{Berahas2021Sequential}, although, \cite{Berahas2021Sequential} only~required a gradient estimator. We need a Hessian estimator since \eqref{equ:ran:Newton} contains the second-order information.

An immediate implication of Assumption \ref{ass:ran:1} is that there exist multiple~positive constants $0<\kappa_{1,G}\leq 1\leq \kappa_{2,G}\wedge\kappa_{\nabla_{\bx}\mL}\wedge\kappa_{\nabla_{\bx}^2c}\wedge\kappa_\nu\wedge\kappa_{\mL_{\mu, \nu}}\wedge\kappa_M$ such~that
\begin{align}\label{Prop:1}
&\kappa_{1,G} I \preceq G_kG_k^T \preceq \kappa_{2,G} I, \quad \|\nabla_{\bx}\mL_k\|\leq \kappa_{\nabla_{\bx}\mL}, \quad (\sum_{i=1}^m\|\nabla^2c_i(\bx_k)\|^2)^{\frac{1}{2}}\leq \kappa_{\nabla_{\bx}^2c}, \nonumber\\
&\sup_{\mX\times\Lambda}\nbr{\nabla^2\mL_{\mu, \nu}} \leq \kappa_{\mL_{\mu, \nu}} \coloneqq \mu \cdot \sup_{\mX}\nbr{\nabla\rbr{ G^Tc}}+ \kappa_\nu, \quad\quad \|M_k\| \leq\kappa_M.
\end{align}
Here, $\kappa_{\mL_{\mu, \nu}}$ is a uniform bound of the Hessian $\nabla^2\mL_{\mu, \nu}$ over $\mX\times \Lambda$.

The following lemma characterizes the bias and variance of the stochastic search direction $(\barDelta\bx_k, \barDelta\blambda_k)$.

\begin{lemma}\label{lem:6}
Under Assumptions \ref{ass:ran:1} and \ref{ass:ran:2}, there exists a constant $\Upsilon_0>0$ such~that
\begin{align*}
&\mE_{\xi_g^k,\xi_H^k}\sbr{\begin{pmatrix}
\barDelta\bx_k\\
\barDelta\blambda_k
\end{pmatrix} } = \begin{pmatrix}
\Delta\bx_k\\
\Delta\blambda_k
\end{pmatrix},\\ &\mE_{\xi_g^k,\xi_H^k}\sbr{\nbr{\begin{pmatrix}
\barDelta\bx_k\\
\barDelta\blambda_k
\end{pmatrix} }^2} \leq \Upsilon_0\rbr{\nbr{\begin{pmatrix}
\Delta\bx_k\\
\Delta\blambda_k	
\end{pmatrix} }^2 + \psi}.
\end{align*}
\end{lemma}

\begin{proof}
See Appendix \ref{pf:lem:6}.\qed
\end{proof}

Using Lemma \ref{lem:6}, we are able to show the global convergence of Algorithm~\ref{alg:NSto:SQP}. We streamline the analysis procedure. For any penalty parameters $\mu, \nu>0$, we apply Taylor's expansion and have
\begin{equation*}
\mL_{\mu, \nu}^{k+1} \leq \mL_{\mu, \nu}^k + \alpha_k\left(\begin{smallmatrix}
\nabla_{\bx}\mL_{\mu, \nu}^k\\
\nabla_{\blambda}\mL_{\mu, \nu}^k
\end{smallmatrix}\right)^T\left(\begin{smallmatrix}
\barDelta\bx_k\\
\barDelta\blambda_k
\end{smallmatrix}\right) + \frac{\kappa_{\mL_{\mu, \nu}}\alpha_k^2}{2}\nbr{\left(\begin{smallmatrix}
\barDelta\bx_k\\
\barDelta\blambda_k
\end{smallmatrix}\right) }^2.
\end{equation*}
Taking conditional expectation over randomness in $\xi_g^k, \xi_H^k$ and applying Lemma~\ref{lem:6},
\begin{equation}\label{equ:taylor}
\mE_{\xi^k_g,\xi^k_H}[\mL_{\mu, \nu}^{k+1}] \leq \mL_{\mu, \nu}^k + \alpha_k\left(\begin{smallmatrix}
\nabla_{\bx}\mL_{\mu, \nu}^k\\
\nabla_{\blambda}\mL_{\mu, \nu}^k
\end{smallmatrix}\right)^T\left(\begin{smallmatrix}
\Delta\bx_k\\
\Delta\blambda_k
\end{smallmatrix}\right) + \frac{\Upsilon_0\kappa_{\mL_{\mu, \nu}}\alpha_k^2}{2}\big\{\nbr{\left(\begin{smallmatrix}
\Delta\bx_k\\
\Delta\blambda_k
\end{smallmatrix}\right) }^2+\psi\big\}.
\end{equation}
In principle, the middle term on the right hand side provides a sufficient descent, while the last term leads to a higher-order error if the stepsize is small~enough (since it depends on the stepsize quadratically). In particular,~by~direct~calculation (as rigorously proved in Appendix~\ref{pf:thm:4}),    
\begin{equation}\label{equ:inner}
\left(\begin{smallmatrix}
\nabla_{\bx}\mL_{\mu, \nu}^k\\
\nabla_{\blambda}\mL_{\mu, \nu}^k
\end{smallmatrix}\right)^T\left(\begin{smallmatrix}
\Delta\bx_k\\
\Delta\blambda_k
\end{smallmatrix}\right)=(\Delta\bx_k)^T\nabla_{\bx}\mL_k - \mu\|c_k\|^2 + c_k^T\Delta\blambda_k - \nu\|G_k\nabla_{\bx}\mL_k\|^2.
\end{equation}	
Then, by \cite[Proposition 4.2]{Lucidi1990Recursive}, there exist a threshold $\tilde{\mu}>0$ and a constant $\tilde{\delta}>0$, such that if $\mu \geq \tilde{\mu}$,
\begin{equation*}
\left(\begin{smallmatrix}
\nabla_{\bx}\mL_{\mu, \nu}^k\\
\nabla_{\blambda}\mL_{\mu, \nu}^k
\end{smallmatrix}\right)^T\left(\begin{smallmatrix}
\Delta\bx_k\\
\Delta\blambda_k
\end{smallmatrix}\right) \leq -\tilde{\delta}\nbr{\left(\begin{smallmatrix}
\Delta\bx_k\\
G_k\nabla_{\bx}\mL_k
\end{smallmatrix} \right)}^2.
\end{equation*}
Further, by the system \eqref{equ:Newton:AL}, we can show (as rigorously proved in Appendix~\ref{pf:thm:4})
\begin{equation*}
\nbr{\left(\begin{smallmatrix}
	\Delta\bx_k\\
	\Delta\blambda_k
	\end{smallmatrix}\right) }^2 \leq \frac{3\kappa_M^2}{\kappa_{1,G}^2}\nbr{\left(\begin{smallmatrix}
	\Delta\bx_k\\
	G_k\nabla_{\bx}\mL_k
	\end{smallmatrix} \right)}^2.
\end{equation*}
Plugging the above two displays into \eqref{equ:taylor}, we further obtain
\begin{align}\label{NN:7}
&\mE_{\xi^k_g,\xi^k_H}[\mL_{\mu, \nu}^{k+1}] 
\leq \mL_{\mu, \nu}^k - \alpha_k\big\{\tilde{\delta} - \frac{3\Upsilon_0\kappa_{\mL_{\mu, \nu}}\kappa_M^2}{2\kappa_{1,G}^2}\alpha_k\big\}\nbr{\left(\begin{smallmatrix}
\Delta\bx_k\\
G_k\nabla_{\bx}\mL_k
\end{smallmatrix} \right)}^2 + \frac{\Upsilon_0\kappa_{\mL_{\mu, \nu}}\psi}{2}\alpha_k^2 \nonumber\\
& \leq \mL_{\mu, \nu}^k - \frac{\alpha_k\tilde{\delta}}{2}\nbr{\left(\begin{smallmatrix}
\Delta\bx_k\\
G_k\nabla_{\bx}\mL_k
\end{smallmatrix} \right)}^2 + \frac{\Upsilon_0\kappa_{\mL_{\mu, \nu}}\psi}{2}\alpha_k^2.\quad\; \text{(if $\alpha_k \leq \frac{\tilde{\delta}\kappa_{1,G}^2}{3\Upsilon_0\kappa_{\mL_{\mu, \nu}}\kappa_M^2}$)}
\end{align}
With this recursion, we apply a supermartingale argument \cite[Theorem 4.2.12]{Durrett2019Probability} and have the following result. A similar supermartingale technique has been used in \cite{Wang2017Stochastic, Curtis2020fully} for analyzing different stochastic algorithms.

\begin{theorem}[Global convergence of Algorithm \ref{alg:NSto:SQP}]\label{thm:4}
Consider Algorithm~\ref{alg:NSto:SQP} under Assumptions \ref{ass:ran:1} and \ref{ass:ran:2}. There exist positive constants $\tilde{\mu}, \tilde{\delta}>0$, such that if $\mu \geq \tilde{\mu}$ and $\alpha_k \leq \frac{\tilde{\delta}\kappa_{1,G}^2}{3\Upsilon_0\kappa_{\mL_{\mu, \nu}}\kappa_M^2}$, $\forall k\geq 0$, where $\Upsilon_0>0$ is from Lemma \ref{lem:6}, then we have the following two cases. 

\noindent(a) If $\alpha_k = \alpha $ is a constant sequence, then
\begin{equation*}
\frac{1}{K+1}\sum_{k=0}^{K}\mE[\|\nabla\mL_k\|^2] \leq \frac{2(\kappa_B^2+\kappa_{2,G})}{\tilde{\delta} \kappa_{1,G}}\rbr{\frac{\mL_{\mu, \nu}^0 - \min_{\mX\times\Lambda}\mL_{\mu, \nu}}{(K+1)\alpha} + \frac{\Upsilon_0\kappa_{\mL_{\mu, \nu}}\psi}{2}\alpha}
\end{equation*}
\noindent(b) If $\alpha_k$ is a decaying sequence with $\sum_{k=0}^{\infty}\alpha_k = \infty$, and $\sum_{k=0}^{\infty}\alpha_k^2<\infty$, then
\begin{equation*}
\lim\limits_{K\rightarrow\infty}\frac{1}{\sum_{k=0}^{K}\alpha_k}\sum_{k=0}^{K}\alpha_k\|\nabla\mL_k\|^2 = 0 \quad \text{ almost surely.} 
\end{equation*}	
Furthermore, $\liminf_{k\rightarrow \infty}\|\nabla\mL_k\| = 0$ almost surely.
	
\end{theorem}

\begin{proof}
See Appendix \ref{pf:thm:4}. \qed
\end{proof}

Theorem \ref{thm:4}(b) differs from \cite[Corollary 3.14(b)]{Berahas2021Sequential}, where the authors~showed the ``liminf" convergence of the expected KKT residual $\mE[\|\nabla\mL_k\|]$. Our ``almost sure" convergence result suggests that, in \textbf{each} run of the algorithm, the~KKT residual sequence $\{\|\nabla\mL_k\|\}_k$ contains a convergent subsequence. Differently,~the convergence in expectation suggests that, the average of the~KKT residual sequence $\{\|\nabla\mL_k\|\}_k$ across \textbf{multiple} runs contains a convergent subsequence. Our ``almost sure" convergence is achieved by using a supermartingale~argument \cite[Theorem 4.2.12]{Durrett2019Probability}, which states that \textit{a positive supermartingale converges almost surely}. The convergence in expectation can also be established for Algorithm \ref{alg:NSto:SQP}, by analogy to the analysis in \cite{Berahas2021Sequential}.

As commented earlier, the design of Algorithm \ref{alg:NSto:SQP} is impractical due to~the lack of adaptivity of the stepsize: a whole stepsize sequence has to be specified~without any adjustment based on the iterate, with an unknown stepsize upper~bound~being tuned manually. \cite{Berahas2021Sequential} introduced a novel stepsize selection scheme to resolve this limitation and introduce certain adaptivity into the method. However, that scheme requires another prespecified sequence $\{\beta_k\}_k$ to control~the stepsize, which still highly affects the performance as revealed in our~experiments. This paper resolves the limitation from a different angle in~Section~\ref{sec:4}.~In particular, we incorporate a (stochastic) line search step into SQP, to~make the resulting algorithm more adaptive than Algorithm \ref{alg:NSto:SQP} and \cite{Berahas2021Sequential}; although such enhancement requires a more precise model approximation.

\section{An Adaptive Stochastic SQP}\label{sec:4}

We now develop an adaptive StoSQP scheme. We embed a stochastic~line~search procedure into our SQP framework. Unlike the scheme in Section \ref{sec:3},~the stepsize in this section is stochastic.

The stochastic line search procedure replaces inaccessible quantities in \eqref{equ:Armijo}, $\mL_{\mu, \nu}$ and $\nabla\mL_{\mu, \nu}$, by their stochastic estimates, $\barL_{\mu, \nu}$ and $\bnabla\mL_{\mu, \nu}$. This results~in three main challenges. First, the stochastic merit function $\barL_{\mu, \nu}$ is a random function. Its decrease in each iteration may not accumulate, as $\barL_{\mu, \nu}$ depends~on the particular realization of $\xi$ that varies with $k$. Second, a bad estimate of $\mL_{\mu, \nu}$ or $\nabla\mL_{\mu, \nu}$ may cause $\alpha_k$ to become arbitrarily small. As a result, $\alpha_k$ does~not have a uniform lower bound, which, however, is critical for global analysis in deterministic setting. Finally, even if the Armijo condition \eqref{equ:Armijo} is satisfied for a stochastic function, the objective in expectation may be arbitrarily large and the value of $\mL_{\mu, \nu}$ may actually increase. \cite{Paquette2020Stochastic} resolved these technical challenges for unconstrained problems by combining approaches of \cite{Blanchet2019Convergence} and~\cite{Cartis2017Global}.~In~principle, one requires a sufficiently accurate model in the line search step to ensure the validity of the selected random~stepsize.~In~other~words,~the~line~search~adaptivity is achieved by drawing more~samples in each iteration. While this might~not~seem ideal, it is standard in the design of other adaptive algorithms \citep{Friedlander2012Hybrid, Byrd2012Sample, Krejic2013Line, De2017Automated, Bollapragada2018Adaptive}.

We generalize the line search scheme of \cite{Paquette2020Stochastic} to equality-constrained problems under constraint qualifications. There are two additional challenges. First,~we have to adaptively select the penalty parameter in the algorithm. Without constraints, the Armijo condition uses the objective function, and there are~no~extra parameters. With constraints, the search direction is a descent direction~of $\mL_{\mu, \nu}$ only if $\mu$ is large enough. In stochastic case, the selected $\mu$ is a random quantity that induces a random walk. We need to ensure that the random~walk enjoys a similar property to that in deterministic case, that is, $\mu$ stabilizes~after a number of iterations. Second, for unconstrained~problems,~if~the~Armijo~condition~is satisfied for $f$ in each iteration, then the iterates converge to a stationary point of $f$, that is $\|\nabla f_k\|\rightarrow 0$. However, we utilize a merit function in the line search for constrained problems. The stationary point of $\mL_{\mu, \nu}$ is not necessarily a KKT point of the original problem \eqref{pro:1}, unless $\mu$ is greater than an \textit{unknown} deterministic threshold such that the square matrix in \eqref{equ:derivative:AL} is invertible. In stochastic case, the stabilized value of $\mu$ is random~and varies in~each~run. Thus, it~is difficult to enforce the stabilized $\mu$ to be always above the deterministic threshold.

To resolve the first challenge, we adopt a similar While loop to Algorithm~\ref{alg:SQP}, and iteratively increase $\mu$ until the projection of the search direction on the gradient of the merit function yields a sufficient descent on the merit function. To resolve the second challenge, we further adjust the While loop condition~to impose a different, but more stringent condition on $\mu$. Our adjustment is~inspired by the following critical observation. 

\begin{proposition}\label{prop:1}
For any point $(\bx, \blambda)$ and penalty parameters $(\mu, \nu)$ with $\nu \neq 0$, if $G(\bx)G^T(\bx)$ is non-singular, then $\|c(\bx)\| = \|\nabla\mL_{\mu, \nu}(\bx, \blambda)\| = 0$ implies $\|\nabla\mL(\bx, \blambda)\|=0$.

\end{proposition}
\begin{proof}
We suppress the evaluation point for simplicity. It suffices to show that $\|\nabla_{\bx}\mL\| = 0$. By the definition of $\nabla_{\blambda}\mL_{\mu, \nu}$ in \eqref{equ:derivative:AL}, since $\nabla_{\blambda}\mL_{\mu, \nu} = \0$ and $c(\bx) = \0$, we know that
\begin{equation}
\nu GG^TG\nabla_{\bx}\mL = \0,
\end{equation}
which implies $G\nabla_{\bx}\mL = \0$. Combining this result with $\nabla_{\bx}\mL_{\mu, \nu} = \0$, we obtain $\nabla_{\bx}\mL = \0$, and complete the proof. \qed
\end{proof}

By Proposition \ref{prop:1}, we notice that any stationary point of the merit function that is feasible is a KKT point of \eqref{pro:1}. The significance of Proposition \ref{prop:1} is that there is no requirement on the penalty parameter $\mu$. Thus, instead~of~investigating if $\mu$ is above the threshold or not like \cite[Proposition 3.16, Example~3.17]{Berahas2021Sequential}, we can be more selective to which stationary point we converge. In particular, we can converge to a stationary point that is also~feasible. We can achieve~this goal by enforcing $\|c_k\|\leq \|\nabla\mL_{\mu, \nu}^k\|$, so that the convergence of $\|\nabla\mL_{\mu, \nu}^k\|$ implies the convergence of $\|c_k\|$. In stochastic case, we instead enforce $\|c_k\|\leq \|\bnabla\mL_{\mu, \nu}^k\|$, where $\bnabla\mL_{\mu, \nu}^k$ is an estimate of $\nabla\mL_{\mu, \nu}^k$. As shown in Lemma \ref{lem:9}, this condition is achievable for sufficiently large $\mu$. We note that the condition $\|c_k\|\leq \|\nabla\mL_{\mu, \nu}^k\|$ has never been enforced in deterministic SQP,~which converges globally even without enforcing it (cf. \cite[Theorem~4.1]{Lucidi1990Recursive}).~For StoSQP, \cite{Berahas2021Sequential} also did not impose additional conditions on the feasibility error $\|c_k\|$. However, we realize that having this additional and more stringent condition in the penalty parameter selection in StoSQP is helpful (at least for our study), in the sense that the condition imposed on the noise distribution in \cite[Proposition 3.16]{Berahas2021Sequential}, in order to converge to a KKT point based on the merit function, is not required in our analysis.

\subsection{The StoSQP scheme}

We generalize the notation of $\xi_g^k$ to let it denote a set of independent realizations of $\xi$, and $|\xi_g^k|$ is its size. We recall that $\nu>0$ is any fixed penalty parameter. Given the $k$-th iterate $(\bx_k, \blambda_k, \baralpha_k, \barepsilon_k)$, the algorithm proceeds in four steps.

\vskip 2pt
\noindent\textbf{Step 1: Estimate the derivatives.} We generate samples $\xi_g^k$ and let
\begin{equation}\label{def:g}
\barg_k = \frac{1}{|\xi_g^k|}\sum_{\zeta\in \xi_g^k}\nabla f(\bx_k; \zeta), \quad \barH_k = \frac{1}{|\xi_g^k|}\sum_{\zeta\in \xi_g^k}\nabla^2 f(\bx_k; \zeta).
\end{equation}
We compute $\bnabla_{\bx}\mL_k$, $\bnabla_{\bx}^2\mL_k$, $\barM_k$ and $\barT_k$ as in \eqref{def:MT}. Different from the estimation~in Section \ref{sec:3}, we do not require $\barg_k$ and $\barH_k$ to be independent. For some constants $\kappa_{grad}>0$, $p_{grad}\in(0, 1)$ to be chosen later, we define the event
\begin{align}\label{event:A_k}
\mA_k = &\bigg\{\nbr{\begin{pmatrix}
\barg_k - \nabla f_k + \nu \rbr{\barM_k G_k\bnabla_{\bx}\mL_k -  M_kG_k\nabla_{\bx}\mL_k}\\
\nu G_kG_k^TG_k\rbr{\barg_k - \nabla f_k}
\end{pmatrix}} \nonumber\\
& \hskip1.5cm \leq  \kappa_{grad}\cdot\baralpha_k\nbr{\begin{pmatrix}
\bnabla_{\bx}\mL_k + \nu\barM_k G_k\bnabla_{\bx}\mL_k + G_k^Tc_k\\
\nu G_kG_k^TG_k\bnabla_{\bx}\mL_k
\end{pmatrix}}  \bigg\},
\end{align}
for a monotonically increasing sequence $|\xi_g^k|$ chosen so that
\begin{equation}\label{equ:ran:cond:1}
P(\mA_k^c \mid \bx_k, \blambda_k) \leq p_{grad}.
\end{equation}
The event $\mA_k$ contains all good estimates of $\nabla\mL_{\mu, \nu}^k$ such that $\|\bnabla\mL_{\mu, \nu}^k - \nabla\mL_{\mu, \nu}^k\|$ is small. We note that the difference $\bnabla\mL_{\mu, \nu}^k - \nabla\mL_{\mu, \nu}^k$ is independent~of $\mu$, so that we can generate samples to estimate the gradient before selecting $\mu$. We will show in Lemma \ref{lem:cond:1} that \eqref{equ:ran:cond:1} can be satisfied for sufficiently large $|\xi_g^k|$.

\vskip 2pt
\noindent\textbf{Step 2: Select the penalty parameter.}
Given stochastic estimates $\bnabla_{\bx}\mL_k$ and $\barM_k$, we generate $B_k$ and compute $(\barDelta \bx_k, \barDelta\blambda_k)$ by \eqref{equ:ran:Newton}. Then, we select $\barmu_k$ such that
\begin{equation}\label{equ:ran:cond:2}
\left(\begin{smallmatrix}
\bnabla_{\bx}\mL_{\barmu_k, \nu}^k\\
\bnabla_{\blambda}\mL_{\barmu_k, \nu}^k
\end{smallmatrix}\right)^T\left(\begin{smallmatrix}
\barDelta\bx_k\\
\barDelta\blambda_k
\end{smallmatrix}\right) \leq - \frac{ \gamma_{RH}\wedge\nu}{2}\nbr{\left(\begin{smallmatrix}
\barDelta\bx_k\\
G_k\bnabla_{\bx}\mL_k
\end{smallmatrix}\right)}^2 \; \text{\ and\ }\; \nbr{c_k}\leq \|\bnabla\mL_{\barmu_k, \nu}^k\|,
\end{equation}
where $\gamma_{RH}\in(0,1]$ is the lower bound of $B_k$ in the null space of $G_k$ (cf.~Assumption \ref{ass:ran:1}), which is an input of our algorithm, and $\bnabla\mL_{\barmu_k, \nu}^k$ can be expressed~as
\begin{align}\label{N:13}
\begin{pmatrix}
\bnabla_{\bx}\mL_{\barmu_k, \nu}^k\\
\bnabla_{\blambda}\mL_{\barmu_k, \nu}^k
\end{pmatrix} = \begin{pmatrix}
\rbr{I+\nu\barM_kG_k}\bnabla_{\bx}\mL_k + \barmu_kG_k^Tc_k\\
c_k + \nu G_kG_k^TG_k\bnabla_{\bx}\mL_k
\end{pmatrix}.
\end{align}
The first condition in \eqref{equ:ran:cond:2} is similar to Line 4 in Algorithm \ref{alg:SQP}, except that we do not update $\delta_k$ but fix it to be $(\gamma_{RH}\wedge \nu)/2$ for simplicity. The second condition is our adjustment introduced at the beginning of this section, which is not~required in Algorithm \ref{alg:SQP} but is critical in our StoSQP. This condition bounds the feasibility error by the magnitude of the gradient of the augmented Lagrangian, so that we could converge to a feasible stationary point as the gradient vanishes, which is a KKT point as implied by Proposition \ref{prop:1}. We will show in Lemma \ref{lem:9} that both conditions in \eqref{equ:ran:cond:2} can be satisfied for sufficiently large $\barmu_k$. Note that the other penalty parameter $\nu >0$ is an input of our algorithm and need not to be updated with iteration.

\vskip 2pt
\noindent\textbf{Step 3: Estimate the merit function.} Given the selected $\barmu_k$ from Step 2, we estimate the merit function that is used in the line search step. Recall that $\baralpha_k$ is from the $(k-1)$-th step. We let $\bx_{s_k} = \bx_k + \baralpha_k\barDelta \bx_k$ (similar for $\blambda_{s_k}$) be the test point, and let $c_{s_k} = c(\bx_{s_k})$, $G_{s_k} = G(\bx_{s_k})$. We then generate a set of independent realizations $\xi_f^k$ and let
\begin{align*}
\barf_k & = \frac{1}{|\xi_f^k|}\sum_{\zeta\in \xi_f^k}f(\bx_k; \zeta),  \quad\quad\quad  \barf_{s_k} =  \frac{1}{|\xi_f^k|}\sum_{\zeta\in \xi_f^k}f(\bx_{s_k}; \zeta), \\
\bnabla  f_k & = \frac{1}{|\xi_f^k|}\sum_{\zeta\in \xi_f^k}\nabla f(\bx_k; \zeta), \quad \; \bnabla  f_{s_k} = \frac{1}{|\xi_f^k|}\sum_{\zeta\in \xi_f^k}\nabla f(\bx_{s_k}; \zeta).
\end{align*}
Note that we distinguish $\bnabla f_k$ from $\barg_k$ in \eqref{def:g}, although they are both estimates of $\nabla f_k$. This simplifies our analysis as the randomness in each step~is~independent from other steps. Then, we let
\begin{equation}\label{N:12}
\begin{aligned}
\barL_{\barmu_k, \nu}^k & = \barf_k + c_k^T\blambda_k + \frac{\barmu_k}{2}\|c_k\|^2 + \frac{\nu}{2}\|G_k(\bnabla f_k + G^T_k\blambda_k) \|^2,\\
\barL_{\barmu_k, \nu}^{s_k} & =  \barf_{s_k} + c_{s_k}^T\blambda_{s_k}+\frac{\barmu_k}{2}\|c_{s_k}\|^2 + \frac{\nu}{2}\|G_{s_k}(\bnabla f_{s_k} + G_{s_k}^T\blambda_{s_k}) \|^2
\end{aligned}
\end{equation}
be the estimates of $\mL_{\barmu_k, \nu}^k$ and $\mL_{\barmu_k, \nu}^{s_k}$. For some constants $\kappa_f>0, p_f\in(0,1)$ to be chosen later, we define the event
\begin{equation}\label{event:B_k}
\mB_k = \bigg\{ \abr{\barL_{\barmu_k, \nu}^k - \mL_{\barmu_k, \nu}^k} \vee \abr{\barL_{\barmu_k, \nu}^{s_k} - \mL_{\barmu_k, \nu}^{s_k}} \leq  -\kappa_f\baralpha_k^2\begin{pmatrix}
\bnabla_{\bx}\mL_{\barmu_k, \nu}^k\\
\bnabla_{\blambda}\mL_{\barmu_k, \nu}^k
\end{pmatrix}^T\begin{pmatrix}
\barDelta\bx_k\\
\barDelta\blambda_k
\end{pmatrix} \bigg\},
\end{equation}
for $|\xi_f^k|$ large enough so that
\begin{equation}\label{equ:ran:cond:3}
P(\mB_k^c \mid \bx_k, \blambda_k, \barDelta\bx_k, \barDelta\blambda_k)\leq p_f
\end{equation}
and
\begin{equation}\label{equ:ran:cond:4}
\mE_{\xi_f^k}[|\barL_{\barmu_k, \nu}^k - \mL_{\barmu_k, \nu}^k|^2] \vee \mE_{\xi_f^k}[|\barL_{\barmu_k, \nu}^{s_k} - \mL_{\barmu_k, \nu}^{s_k}|^2] \leq \barepsilon_k^2.
\end{equation}
Here $\barepsilon_k$ is updated with the iteration. Again, we will show in Lemma \ref{lem:cond:2} that \eqref{equ:ran:cond:3} and \eqref{equ:ran:cond:4} can be satisfied for sufficiently large $|\xi_f^k|$. Different from $\xi_g^k$, we do not require $|\xi_f^k|$ to be an increasing sequence.

\vskip 2pt
\noindent\textbf{Step 4: Perform the line search step.}
Using the estimates defined above, we update the iterate based on whether the Armijo condition is satisfied.

\noindent(a) If the Armijo condition holds:
\begin{equation}\label{equ:ran:Armijo}
\barL_{\barmu_k, \nu}^{s_k} \leq \barL_{\barmu_k, \nu}^k +
\baralpha_k\beta\begin{pmatrix}
\bnabla_{\bx}\mL_{\barmu_k, \nu}^k\\
\bnabla_{\blambda}\mL_{\barmu_k, \nu}^k
\end{pmatrix}^T\begin{pmatrix}
\barDelta\bx_k\\
\barDelta\blambda_k
\end{pmatrix} ,
\end{equation}
then we let $\bx_{k+1} = \bx_{s_k}$, $\blambda_{k+1} = \blambda_{s_k}$ and increase the stepsize by $\baralpha_{k+1} = \rho\baralpha_k \wedge \alpha_{max}$, with $\rho>1$. Moreover, if we observe a sufficient decrease, that is,
\begin{equation}\label{equ:sufficient:dec}
-\baralpha_k\beta\begin{pmatrix}
\bnabla_{\bx}\mL_{\barmu_k, \nu}^k\\
\bnabla_{\blambda}\mL_{\barmu_k, \nu}^k
\end{pmatrix}^T\begin{pmatrix}
\barDelta\bx_k\\
\barDelta\blambda_k
\end{pmatrix} \geq \barepsilon_k,
\end{equation}
we then increase $\barepsilon_k$ as $\barepsilon_{k+1} = \rho\barepsilon_k$, otherwise $\barepsilon_{k+1} = \barepsilon_k/\rho$.

\noindent(b) If the Armijo condition \eqref{equ:ran:Armijo} does not hold, we do not update the current iterate and let $\bx_{k+1} = \bx_k$, $\blambda_{k+1} = \blambda_k$, and decrease the stepsize $\baralpha_k$ by $\baralpha_{k+1} = \baralpha_k/\rho$ and $\barepsilon_k$ by $\barepsilon_{k+1} = \barepsilon_k/\rho$.

The four steps are summarized in Algorithm \ref{alg:ASto:SQP}. Before delving into the algorithm, we provide few remarks.

\begin{algorithm}[!tp]
\caption{An Adaptive StoSQP with Exact Augmented Lagrangian}\label{alg:ASto:SQP}
\begin{algorithmic}[1]
\State \textbf{Input:} initial iterate $(\bx_0, \blambda_0)$, parameters $\gamma_{RH}\in(0, 1]$, $\baralpha_0 = \alpha_{max}>0$, $\nu$, $\barmu_0$, $\barepsilon_0$, $\kappa_{grad}>0$, $\rho>1$, $p_{grad}, p_f, \beta\in(0, 1)$, $\kappa_f\in(0, \beta/4\alpha_{max}]$;
		
\For{$k = 0,1,2\ldots$}
\State Generate $\xi_g^k$ and compute $\bnabla_{\bx}\mL_k, \barM_k$, such that $|\xi_g^k|\geq |\xi_g^{k-1}|+1$ ($|\xi_g^{-1}| = 0$) and
\begin{equation*}
P(\mA_k^c \mid \bx_k, \blambda_k) \leq p_{grad};
\end{equation*}
		
\State Generate $B_k$ such that $\bx^TB_k\bx \geq \gamma_{RH}\|\bx\|^2$, for any $\bx \in \{\bx: G_k\bx = \0, \bx\neq \0\}$, and compute $(\barDelta\bx_k, \barDelta\blambda_k)$ by solving \eqref{equ:ran:Newton};
		
\While{$\begin{pmatrix}
	\bnabla_{\bx}\mL_{\barmu_k, \nu}^k\\
	\bnabla_{\blambda}\mL_{\barmu_k, \nu}^k
	\end{pmatrix}^T\begin{pmatrix}
	\barDelta\bx_k\\
	\barDelta \blambda_k
	\end{pmatrix}>-\frac{\gamma_{RH}\wedge\nu}{2}\nbr{\begin{pmatrix}
		\barDelta\bx_k\\
		G_k\bnabla_{\bx}\mL_k
		\end{pmatrix} }^2$ OR $\|c_k\|>\|\bnabla\mL_{\barmu_k, \nu}^k\|$}
\State Let $\barmu_k = \rho\barmu_k$;
\EndWhile
		
\State Generate $\xi_f^k$ and compute $\barL_{\barmu_k, \nu}^k$ and $\barL_{\barmu_k, \nu}^{s_k}$, such that
\begin{align*}
P(\mB_k^c \mid \bx_k, \blambda_k, \barDelta\bx_k, \barDelta\blambda_k)\leq & p_f,\\
\mE_{\xi_f^k}[|\barL_{\barmu_k, \nu}^k - \mL_{\barmu_k, \nu}^k|^2] \vee \mE_{\xi_f^k}[|\barL_{\barmu_k, \nu}^{s_k} - \mL_{\barmu_k, \nu}^{s_k}|^2] \leq & \barepsilon_k^2;
\end{align*}
		
\If{$\barL_{\barmu_k, \nu}^{s_k} \leq \barL_{\barmu_k, \nu}^k + \baralpha_k \beta \begin{pmatrix}
	\bnabla_{\bx}\mL_{\barmu_k, \nu}^k\\
	\bnabla_{\blambda}\mL_{\barmu_k, \nu}^k
	\end{pmatrix}^T\begin{pmatrix}
	\barDelta\bx_k\\
	\barDelta\blambda_k
	\end{pmatrix}$} \Comment{Successful step}
\State Let $\begin{pmatrix}
\bx_{k+1}\\
\blambda_{k+1}
\end{pmatrix} = \begin{pmatrix}
\bx_k\\
\blambda_k
\end{pmatrix} + \baralpha_k\begin{pmatrix}
\barDelta\bx_k\\
\barDelta\blambda_k
\end{pmatrix}$;
\State Set $\baralpha_{k+1} = \alpha_{max} \wedge \rho\baralpha_k$;
\If{$- \baralpha_k\beta \begin{pmatrix}
	\bnabla_{\bx}\mL_{\barmu_k, \nu}^k\\
	\bnabla_{\blambda}\mL_{\barmu_k, \nu}^k
	\end{pmatrix}^T\begin{pmatrix}
	\barDelta\bx_k\\
	\barDelta\blambda_k
	\end{pmatrix} \geq \barepsilon_k$} \Comment{Reliable step}
\State Let $\barepsilon_{k+1} = \rho\barepsilon_k$;
\Else \Comment{Unreliable step}
\State Let $\barepsilon_{k+1} = \barepsilon_k/\rho$;
\EndIf
\Else \Comment{Unsuccessful step}
\State Let $\begin{pmatrix}
\bx_{k+1}\\
\blambda_{k+1}
\end{pmatrix} = \begin{pmatrix}
\bx_k\\
\blambda_k
\end{pmatrix}, \baralpha_{k+1} = \baralpha_k/\rho, \barepsilon_{k+1} = \barepsilon_k/\rho$;
\EndIf
\State $\barmu_{k+1} = \barmu_k$;
\EndFor
\end{algorithmic}
\end{algorithm}

\begin{remark}

Our condition \eqref{equ:ran:cond:4} simplifies the condition of \cite[(2.3)]{Paquette2020Stochastic}, which~requires the variance of the merit function estimates to be bounded by $\max\{\barepsilon_k^2, \T_k\}$ for a term $\T_k$ that depends on the gradient $\nabla f_k$. As explained in \cite[Remark 1]{Paquette2020Stochastic}, $\T_k$ is unknown in stochastic optimization. Adding the term $\T_k$ is just to~make algorithm more flexible, in case one has external knowledge of $\nabla f_k$. Our paper does not assume the access to $\nabla f_k$; hence we remove $\T_k$ to make the condition \eqref{equ:ran:cond:4} checkable in practice.

\end{remark}

\begin{remark}
Compared to \cite{Berahas2021Sequential} which generates a single sample in each iteration, and Algorithm \ref{alg:NSto:SQP} which generates two samples, the line search StoSQP~in~Algorithm \ref{alg:ASto:SQP} requires much more samples. Seeing from conditions \eqref{equ:ran:cond:1}, \eqref{equ:ran:cond:3}, and~\eqref{equ:ran:cond:4}, we require a more precise model estimation to make the selected stochastic stepsize informative. The benefit of generating more samples is that the line search scheme is more adaptive than \cite{Berahas2021Sequential} and Algorithm \ref{alg:NSto:SQP}, in the sense that~no prespecified sequence, which highly affects the performance and determines the convergence behavior, is required. As introduced earlier, achieving  adaptivity at the cost of sample complexity is common in the~literature~\citep{Friedlander2012Hybrid, Byrd2012Sample, Krejic2013Line, De2017Automated, Bollapragada2018Adaptive}. As part of the line search, Algorithm \ref{alg:ASto:SQP} has to evaluate the merit function for checking the Armijo condition, which is not required for Algorithm \ref{alg:NSto:SQP}, and also for \cite{Berahas2021Sequential} if the Lipschitz constants of objective and constraints are given.

\end{remark}

We also comment on the randomness of the iteration. Let $\mF_0\subseteq \mF_1\subseteq\mF_2\ldots$ be a filtration of $\sigma$-algebras where $\mF_k$ is generated by $\{\xi_f^j, \xi_g^j\}_{j=0}^k$. Let $\mF_{k-0.5}$~be the $\sigma$-algebra generated by $\{\xi_f^j, \xi_g^j\}_{j=0}^{k-1} \cup \xi_g^k$. We have $\mF_{k-1} \subseteq\mF_{k-0.5}\subseteq\mF_k$. Finally, we let $\mF_{-1} = \sigma(\bx_0, \blambda_0)$. By $\barmu_k$ we denote the quantity obtained after the While loop in Line 5 of Algorithm \ref{alg:ASto:SQP}. By the construction, it is easy to see
\begin{equation*}
\sigma(\bx_k, \blambda_k)\cup \sigma(\baralpha_k)\cup \sigma(\barepsilon_k) \subseteq \mF_{k-1}, \quad \sigma(\barDelta\bx_k, \barDelta\blambda_k) \cup \sigma(\barmu_k) \subseteq \mF_{k-0.5}, \quad \forall k\geq 0.
\end{equation*}
In particular, the filtration $\mF_{k-1}$ contains all the randomness before the $k$-th iteration in Line 3. In the $k$-th iteration, we first generate $\xi_g^k$ and obtain $\mF_{k-0.5}$. Then, we compute the search direction $(\barDelta\bx_k, \barDelta\blambda_k)$ and $\barmu_k$. Finally, we generate $\xi_f^k$ and obtain $\mF_k$, allowing us to obtain $(\bx_{k+1}, \blambda_{k+1})$.

In the line search step, which starts from Line 9 in Algorithm~\ref{alg:ASto:SQP}, $\barepsilon_k$~characterizes the reliability of the decrease we observe on the stochastic merit function. In particular, each iteration is divided into a successful step (Armijo condition is satisfied) or an unsuccessful step (Armijo condition is not satisfied). If the step is successful, we further divide it into a reliable step or an unreliable step. When the amount of decrease on the stochastic merit function is greater than $\barepsilon_k$, we classify it as a reliable step, because it reduces the deterministic merit function with high probability as well. We then increase $\barepsilon_k$ and require a less accurate model in the next iteration (see \eqref{equ:ran:cond:4}). When the observed decrease is less than $\barepsilon_k$, we decrease $\barepsilon_k$, as we are not confident that the deterministic merit function is decreased. Thus, we require a more accurate model in the next iteration.

\subsection{Well-posedness of Algorithm \ref{alg:ASto:SQP}}\label{sec:4.2}

We study the well-posedness of Algorithm \ref{alg:ASto:SQP} and show that each of its step can be performed in finite time. We first lay out the assumption.

\begin{assumption}\label{ass:A:1}

We assume that Assumption \ref{ass:ran:1} holds for the iterates generated by Algorithm~\ref{alg:ASto:SQP}. In particular, the iterates $(\bx_k, \blambda_k)$ and trial points $(\bx_{s_k}, \blambda_{s_k})$ are contained in a convex compact set $\mX \times \Lambda$. The functions $f$ and $c$ are thrice continuously differentiable over $\mX$; and the Jacobian $G(\bx) = \nabla^Tc(\bx)$ has full row rank over $\bx\in\mX$. The sequence $\{B_k\}_k$ satisfies $\bx^TB_k\bx \geq \gamma_{RH} \|\bx\|^2$ for any~$\bx\in \{\bx: G_k\bx = \0, \bx \neq \0\}$, and $\|B_k\|\leq \kappa_B$ for constants $0<\gamma_{RH}\leq 1\leq \kappa_B$. Furthermore, we assume $\|\nabla\mL_k\|\wedge  \|(\barDelta\bx_k, \barDelta\blambda_k)\|>0$.

\end{assumption}

Note that $\|\nabla\mL_k\|\wedge \|(\barDelta\bx_k, \barDelta\blambda_k)\|>0$ is imposed only for analytical reasons. It allows us to generate an infinite sequence of iterates $\{(\bx_k, \blambda_k)\}_k$. A similar condition can be found in (2.2) in \cite{Bollapragada2018Adaptive}. We point out that a practical algorithm should stop whenever $ \|\bnabla \mL_k\| \wedge \|(\barDelta\bx_k, \barDelta\blambda_k)\|\leq \tau$ for a tolerance~$\tau$. By Assumption~\ref{ass:A:1}, bounds in \eqref{Prop:1} hold immediately. We also strengthen the~bounded variance condition in Assumption \ref{ass:ran:2} to the following boundedness condition.

\begin{assumption}\label{ass:A:2}
We assume that, for any $\xi \sim \mP$ and $\bx\in \mX$, $|f(\bx; \xi) - f(\bx)| \leq \Omega_0$, $\|\nabla f(\bx;\xi) - \nabla f(\bx)\|\leq \Omega_1$, and $\|\nabla^2 f(\bx;\xi) - \nabla^2 f(\bx)\|\leq \Omega_2$ for constants $\Omega_0, \Omega_1, \Omega_2 >0$.
\end{assumption}

We require Assumption \ref{ass:A:2} when applying (matrix) Bernstein concentration inequality \citep[Theorem 6.1.1]{Tropp2015Introduction}, which is used for characterizing the batch sizes required to ensure conditions, \eqref{equ:ran:cond:1}, \eqref{equ:ran:cond:3}, and \eqref{equ:ran:cond:4}, on the model precision to~hold. The same concentration result holds if random errors have~sub-exponential tail. See \citep[Theorems 6.1 and 6.2]{Tropp2011User} for definitions and details. Either condition on the error is stronger than the bounded variance condition, which does not imply an exponential tail in general. On the other hand, Assumption \ref{ass:A:2} is widely used in subsampling analysis \citep{Tripuraneni2018Stochastic, RoostaKhorasani2018Sub}; and reasonably holds if iterates are bounded (as assumed in Assumption \ref{ass:A:1}) and we target a finite-sum problem. Assuming that the gradient error is bounded is also required for ensuring the penalty parameter to stabilize in \cite[Proposition 3.18]{Berahas2021Sequential}, and we obtain similar guarantee in Lemma~\ref{lem:9}.

We are now ready for analyzing Algorithm \ref{alg:ASto:SQP}. To show that Algorithm \ref{alg:ASto:SQP} is well-posed, it suffices to show that the condition \eqref{equ:ran:cond:1} in Line 3, the condition \eqref{equ:ran:cond:2} in Line 5, and the conditions \eqref{equ:ran:cond:3} and \eqref{equ:ran:cond:4} in Line 8 can be satisfied. We study them in the next three lemmas.

\begin{lemma}\label{lem:cond:1}
Under Assumptions \ref{ass:A:1} and \ref{ass:A:2}, the condition \eqref{equ:ran:cond:1} can be satisfied~using Algorithm \ref{alg:sample} with a large enough constant $C_{grad}$. Moreover, Algorithm \ref{alg:sample} terminates in finite time.
\end{lemma}

\begin{algorithm}[!tp]
	\caption{Sample size selection}
	\label{alg:sample}
	\begin{algorithmic}[1]
		\State \hskip-1.5pt\textbf{Input:} initial $|\xi_g^k|$, known variables $\baralpha_k, c_k, G_k, \kappa_{grad}, p_{grad}, d, \nu$, scalars $\rho>1$, $C_{grad}>0$;
		\While{True}
		\State Generate $|\xi_g^k|$ samples to compute $\bnabla_{\bx}\mL_k$ and $\barM_k$ as in \eqref{def:MT};
		\If{\begin{equation}\label{cond:11}
			|\xi_g^k| < \frac{C_{grad}\log\rbr{\frac{4d}{p_{grad}} }}{\kappa_{grad}^2\cdot \baralpha_k^2\nbr{\begin{pmatrix}
					\bnabla_{\bx}\mL_k + \nu\barM_k G_k\bnabla_{\bx}\mL_k + G_k^Tc_k\\
					\nu G_kG_k^TG_k\bnabla_{\bx}\mL_k
					\end{pmatrix}}^2\wedge 1}
			\end{equation}}
		\State $|\xi_g^k| = \rho|\xi_g^k|$;
		\Else
		\State \textbf{Break};
		\EndIf
		\EndWhile
	\end{algorithmic}
\end{algorithm}

\begin{proof}

Let $P_{\xi_g^k}(\cdot) = P(\cdot \mid \bx_k, \blambda_k)$ be the conditional probability over randomness in $\xi_g^k$. We have
\begin{align*}
&\nbr{\begin{pmatrix}
\barg_k - \nabla f_k + \nu\rbr{\barM_kG_k\bnabla_{\bx}\mL_k - M_kG_k\nabla_{\bx}\mL_k}\\
\nu G_kG_k^TG_k\rbr{\barg_k - \nabla f_k}
\end{pmatrix}}\\
& \leq \nbr{\begin{pmatrix}
\barg_k - \nabla f_k  \\
\nu G_kG_k^TG_k\rbr{\barg_k - \nabla f_k}
\end{pmatrix}} + \nu\nbr{\barM_kG_k\bnabla_{\bx}\mL_k - M_kG_k\nabla_{\bx}\mL_k}\\
& \stackrel{\mathclap{\eqref{Prop:1}}}{\leq} (1+\nu\kappa_{2,G}^{3/2})\nbr{\barg_k - \nabla f_k} + \nu\kappa_{2,G}^{1/2}\big(\|\barM_k - M_k\|\|\barg_k - \nabla f_k\| + \kappa_{\nabla_{\bx}\mL}\|\barM_k - M_k\| \\
&\quad \quad + \kappa_M\|\barg_k - \nabla f_k\|\big)\\
& =  (1 + \nu\kappa_{2,G}^{3/2} + \nu\kappa_{2,G}^{1/2}\kappa_M)\|\barg_k - \nabla f_k\| + \nu\kappa_{2,G}^{1/2}\rbr{\|\barg_k - \nabla f_k\|+ \kappa_{\nabla_{\bx}\mL} } \|\barM_k - M_k\|.
\end{align*}
Moreover, we have
\begin{align}\label{bound:M}
\|\barM_k - M_k\| & \stackrel{\eqref{def:MT}}{\leq}  \sqrt{\kappa_{2,G}}\|\barH_k - \nabla^2f_k\| + \|\barT_k - T_k\| \nonumber\\
& \stackrel{\eqref{def:MT}}{\leq} \sqrt{\kappa_{2,G}}\|\barH_k - \nabla^2f_k\| + \{\|\barg_k - \nabla f_k\|^2\sum_{i=1}^{m}\|\nabla^2c_i(\bx_k)\|^2\}^{1/2} \nonumber\\
& \stackrel{\eqref{Prop:1}}{\leq}\sqrt{\kappa_{2,G}}\|\barH_k - \nabla^2f_k\| + \kappa_{\nabla_{\bx}^2c}\cdot\|\barg_k - \nabla f_k\|.
\end{align}
Combining the above two displays, we obtain
\begin{multline}\label{N:18}
\bigg\|\begin{pmatrix}
\barg_k - \nabla f_k + \nu\rbr{\barM_kG_k\bnabla_{\bx}\mL_k - M_kG_k\nabla_{\bx}\mL_k}\\
\nu G_kG_k^TG_k\rbr{\barg_k - \nabla f_k}
\end{pmatrix}\bigg\|\\
\leq \big\{1+ \nu\kappa_{2,G}^{3/2} + \nu\kappa_{2,G}^{1/2}\kappa_M + \nu\kappa_{2,G}^{1/2}\kappa_{\nabla_{\bx}^2c}\rbr{\|\barg_k - \nabla f_k\| + \kappa_{\nabla_{\bx}\mL}} \big\}\|\barg_k - \nabla f_k\|\\
+ \nu\kappa_{2,G} \rbr{\|\barg_k - \nabla f_k\| + \kappa_{\nabla_{\bx}\mL}}\|\barH_k - \nabla^2f_k\|.
\end{multline}
Let us denote 
\begin{align*}
t_k = \kappa_{grad}\cdot\baralpha_k\nbr{\begin{pmatrix}
\bnabla_{\bx}\mL_k + \nu\barM_k G_k\bnabla_{\bx}\mL_k + G_k^Tc_k\\
\nu G_kG_k^TG_k\bnabla_{\bx}\mL_k
\end{pmatrix}},
\end{align*}
then the above display implies that $\mA_k$ happens by requiring
\begin{align*}
\|\barg_k - \nabla f_k\| & \leq  \frac{t_k\wedge 1 }{2(1+\nu\kappa_{2,G}^{3/2} + \nu\kappa_{2,G}^{1/2}\kappa_M + 2\nu\kappa_{2,G}^{1/2}\kappa_{\nabla_{\bx}\mL}\kappa_{\nabla_{\bx}^2c}) }  \eqqcolon t_1^k, \\
\|\barH_k - \nabla^2 f_k\| & \leq  \frac{t_k }{4\nu\kappa_{2,G}\kappa_{\nabla_{\bx}\mL}} \eqqcolon t_2^k.
\end{align*}
By Assumption \ref{ass:A:2}, we apply Bernstein inequality \cite[Theorem 6.1.1]{Tropp2015Introduction} and have
\begin{align*}
P_{\xi_g^k}\rbr{\|\barg_k - \nabla f_k\| \leq t_1^k} \geq 1 - \frac{p_{grad}}{2},\quad \quad &\text{if }\quad  |\xi_g^k| \geq \frac{4\Omega_1^2}{(t_1^k)^2}\log\rbr{\frac{4d}{p_{grad}}},\\
P_{\xi_g^k}\rbr{\|\barH_k - \nabla^2 f_k\| \leq t_2^k} \geq 1 - \frac{p_{grad}}{2},\quad \quad &\text{if }\quad  |\xi_g^k| \geq \frac{4\Omega_2^2}{(t_2^k)^2}\log\rbr{\frac{4d}{p_{grad}}}.
\end{align*}
Plugging the definitions of $t_1^k$ and $t_2^k$ and using the fact that $\kappa_{2,G}\wedge \kappa_{\nabla_{\bx}^2c} \geq 1$ (cf. \eqref{Prop:1}), we know $P_{\xi_g^k}(\mA_k^c ) \leq p_{grad}$ provided
\begin{equation}\label{cond:1}
|\xi_g^k| \geq \frac{C_{grad}\log\rbr{\frac{4d}{p_{grad}} }}{t_k^2\wedge 1}
\end{equation}
with $C_{grad} = 16\max\{\Omega_1, \Omega_2\}^2\{1+\nu\kappa_{2,G}^{3/2} + \nu\kappa_{2,G}^{1/2}\kappa_M + 2\nu\kappa_{2,G}\kappa_{\nabla_{\bx}\mL}\kappa_{\nabla_{\bx}^2c}\}^2$.~The above condition \eqref{cond:1} is consistent with \eqref{cond:11} in Algorithm \ref{alg:sample}.

To complete the proof, we have to show that \eqref{cond:1} can be satisfied by~Algorithm \ref{alg:sample} and Algorithm \ref{alg:sample} terminates in finite time. This is not immediate, as the right hand side depends on $\xi_g^k$ as well. Intuitively, as $|\xi_g^k|$ increases, the right hand side term will approach a fixed quantity with \textit{nonzero} denominator, so that \eqref{cond:1} can be satisfied finally. Let us define
\begin{equation*}
t_k^{\star} = \kappa_{grad}\cdot\baralpha_k\nbr{\begin{pmatrix}
\nabla_{\bx}\mL_k + \nu M_k G_k\nabla_{\bx}\mL_k + G_k^Tc_k\\
\nu G_kG_k^TG_k\nabla_{\bx}\mL_k
\end{pmatrix}}.
\end{equation*}
By Assumption \ref{ass:A:1}, we have $t_k^{\star}>0$. Otherwise, $\|\nu G_kG_k^TG_k\nabla_{\bx}\mL_k\| = 0$ and thus $\|G_k\nabla_{\bx}\mL_k\| = 0$. Then $\0 = G_k(\nabla_{\bx}\mL_k + \nu M_k G_k\nabla_{\bx}\mL_k + G_k^Tc_k) = G_kG_k^Tc_k$ and thus $c_k = \0$. Finally, $\nabla_{\bx}\mL_k = \0$ and $\nabla\mL_k = \0$, which contradicts $\|\nabla\mL_k\|>0$ in Assumption \ref{ass:A:1}. Moreover, we have
\begin{equation}\label{NN:D:2}
\abr{t_k - t_k^{\star}} \leq \kappa_{grad}\cdot\baralpha_k\nbr{\begin{pmatrix}
	\barg_k - \nabla f_k + \nu\rbr{\barM_kG_k\bnabla_{\bx}\mL_k - M_kG_k\nabla_{\bx}\mL_k}\\
	\nu G_kG_k^TG_k\rbr{\barg_k - \nabla f_k}
	\end{pmatrix}}.
\end{equation}
Similar to the sample complexity shown in \eqref{cond:1}, we let the right hand side of \eqref{NN:D:2} be bounded by $t_k^{\star}/2$, that is
\begin{multline*}
\nbr{\begin{pmatrix}
\barg_k - \nabla f_k + \nu\rbr{\barM_kG_k\bnabla_{\bx}\mL^k - M_kG_k\nabla_{\bx}\mL^k}\\
\nu G_kG_k^TG_k\rbr{\barg_k - \nabla f_k}
\end{pmatrix}}\\ \leq \frac{1}{2}\nbr{\begin{pmatrix}
\nabla_{\bx}\mL_k + \nu M_k G_k\nabla_{\bx}\mL_k + G_k^Tc_k\\
\nu G_kG_k^TG_k\nabla_{\bx}\mL_k.
\end{pmatrix}} = \frac{t_k^\star}{2\kappa_{grad}\baralpha_k}.
\end{multline*}
Replacing $t_k$ in the denominator of \eqref{cond:1} by the right hand side term of the above display, we know that
\begin{align}\label{pequ:Event}
P_{\xi_g^k}(|t_k - t_k^{\star}|  \leq t_k^{\star}/2)\geq 0.99
\end{align}
provided
\begin{equation}\label{N:11}
|\xi_g^k|\geq   \frac{4\kappa_{grad}^2\baralpha_k^2C_{grad}\log\rbr{\frac{4d}{0.01}}}{(t_k^{\star})^2 \wedge 4\kappa_{grad}^2\baralpha_k^2}.
\end{equation}
Under the event \eqref{pequ:Event}, we have $t_k \geq t_k^\star/2 >0$, and hence the condition \eqref{cond:1} is implied by
\begin{equation*}
|\xi_g^k| \geq \frac{4C_{grad}\log\rbr{\frac{4d}{p_{grad}}}}{(t_k^\star)^2\wedge 4}.
\end{equation*}
Finally, combining the above display with \eqref{N:11}, we know that \eqref{cond:1} holds with a nonzero denominator with probability at least 0.99 as long as
\begin{equation}\label{equ:N:1}
|\xi_g^k| \geq \frac{4C_{grad}(\kappa_{grad}^2\baralpha_k^2\vee 1) \log \rbr{\frac{4d}{p_{grad}\wedge 0.01}}}{(t_k^{\star})^2 \wedge 4\kappa_{grad}^2\baralpha_k^2\wedge 4}.
\end{equation}
Note that the right hand side term is a deterministic quantity conditional on $(\bx_k, \blambda_k)$. Since Algorithm \ref{alg:sample} increases $|\xi_g^k|$ by a factor of $\rho$ in each iteration, the above requirement will finally be satisfied. After $|\xi_g^k|$ exceeds the right hand side threshold, \eqref{cond:1} (i.e. \eqref{cond:11} in Algorithm \ref{alg:sample}) is satisfied with probability at least 0.99 in each While loop iteration. Thus, the While loop will stop in finite time with probability one. This completes the proof. \qed
\end{proof}

As shown in \eqref{cond:11}, the sample complexity $|\xi_g^k|$ depends on a tuning parameter $C_{grad}$. This dependence seems unavoidable and appears in different forms in \cite{Bollapragada2018Adaptive, Paquette2020Stochastic}. The expression of $C_{grad}$ is provided after \eqref{cond:1}. Fortunately, noting~that the denominator in \eqref{cond:11} is bounded by $\|\bnabla\mL_k\|^2$ which, as the estimate of~$\|\nabla\mL_k\|^2$, should be close to zero when $k$ is large, the effect of the tuning parameter $C_{grad}$ is negligible. The sample complexity is proportional to the reciprocal of the square of the gradient, which is also standard in the literature \citep{Paquette2020Stochastic}.

In addition, we note that the right-hand side term of \eqref{cond:11} is computed by the generated samples $\xi_g^k$. As we showed in \eqref{equ:N:1}, there is a deterministic (conditional on $\mF_{k-1}$) threshold on $|\xi_g^k|$ such that, if $|\xi_g^k|$ is above the threshold, then the inequality \eqref{cond:11} on $|\xi_g^k|$ does not hold with high probability.  Thus, the~While~loop in Algorithm \ref{alg:sample} always terminates in finite time (with probability one).~In the implementation, if we encounter a rare situation where the denominator~in \eqref{cond:11} is zero, we can just keep increasing $|\xi_g^k|$ by Line 5. As we proved (cf. statement before \eqref{NN:D:2}), the denominator is finally nonzero since
\begin{equation*}
\left(\begin{smallmatrix}
\bnabla_{\bx}\mL_k + \nu\barM_k G_k\bnabla_{\bx}\mL_k + G_k^Tc_k\\
\nu G_kG_k^TG_k\bnabla_{\bx}\mL_k
\end{smallmatrix}\right) \stackrel{|\xi_g^k|\rightarrow\infty}{\longrightarrow}\left(\begin{smallmatrix}
\nabla_{\bx}\mL_k + \nu M_k G_k\nabla_{\bx}\mL_k + G_k^Tc_k\\
\nu G_kG_k^TG_k\nabla_{\bx}\mL_k
\end{smallmatrix}\right)\neq \0.
\end{equation*}

The following result shows that the conditions \eqref{equ:ran:cond:3} and \eqref{equ:ran:cond:4} can be satisfied. The analysis follows the same structure as the one in Lemma \ref{lem:cond:1}, and applies Bernstein inequality \cite[Theorem 6.1.1]{Tropp2015Introduction}. We defer the proof to the appendix.

\begin{lemma}\label{lem:cond:2}
Under Assumptions \ref{ass:A:1} and \ref{ass:A:2}, the conditions \eqref{equ:ran:cond:3} and \eqref{equ:ran:cond:4} are satisfied if
\begin{equation}\label{cond:2}
|\xi_f^k|  \geq \frac{C_f\log\rbr{\frac{8d}{p_f}}}{\cbr{\kappa_f\baralpha_k^2\begin{pmatrix}
\bnabla_{\bx}\mL_{\barmu_k, \nu}^k\\
\bnabla_{\blambda}\mL_{\barmu_k, \nu}^k
\end{pmatrix}^T\begin{pmatrix}
\barDelta\bx_k\\
\barDelta\blambda_k
\end{pmatrix}}^2\wedge \barepsilon_k^2\wedge 1}
\end{equation}
for a large enough constant $C_f$.
\end{lemma}

\begin{proof}
See Appendix \ref{pf:lem:cond:2}. \qed
\end{proof}

Unlike Lemma \ref{lem:cond:1}, the denominator in \eqref{cond:2} is computable given $\mF_{k-0.5}$. As a result, there is no need to apply a While loop to find $\xi_f^k$ iteratively, as we did in Algorithm \ref{alg:sample} for $\xi_g^k$. The tuning parameter $C_f$ is similar to $C_{grad}$ in Lemma \ref{lem:cond:1} and has an ignorable effect on performance due to the small magnitude of the denominator. The expression of $C_f$ is provided after \eqref{equ:N:2}. Again, we should emphasize that the denominator in \eqref{cond:2} is nonzero. Otherwise, by \eqref{equ:ran:cond:2} we know $\barDelta\bx_k = \0$ and $G_k\bnabla_{\bx}\mL_k = \0$. Then, by \eqref{equ:ran:Newton} we know $\barDelta\blambda_k = \0$, which contradicts $\|(\barDelta\bx_k, \barDelta\blambda_k)\|>0$ in Assumption \ref{ass:A:1}.

The last result of this subsection shows that the condition \eqref{equ:ran:cond:2} on $\barmu_k$ can be satisfied, so that the While loop in Line 5 of Algorithm \ref{alg:ASto:SQP} will stop in finite time. Furthermore, like in deterministic SQP in Algorithm~\ref{alg:SQP}, $\barmu_k$ will stabilize~after a number of iterations.

\begin{lemma}\label{lem:9}
Under Assumptions \ref{ass:A:1} and \ref{ass:A:2}, the condition \eqref{equ:ran:cond:2} can be satisfied by the While loop in Line 5 of Algorithm \ref{alg:ASto:SQP}. Furthermore, there exists a deterministic constant $\tmu>0$ such that $\barmu_k = \barmu_{\barK} \leq \tmu$, $\forall k\geq \barK$ for some $\barK <\infty$.
\end{lemma}

\begin{proof}

It suffices to show that \eqref{equ:ran:cond:2} is satisfied if $\barmu_k$ is sufficiently large, and the threshold has a deterministic upper bound that is independent of $k$. By analogy to \eqref{equ:inner}, we have from Newton's system \eqref{equ:ran:Newton} that
\begin{align*}
\begin{pmatrix}
\bnabla_{\bx}\mL_{\barmu_k, \nu}^k\\
\bnabla_{\blambda}\mL_{\barmu_k, \nu}^k
\end{pmatrix}^T&\begin{pmatrix}
\barDelta \bx_k\\
\barDelta \blambda_k
\end{pmatrix} \\
& \stackrel{\mathclap{\eqref{equ:inner}}}{=} (\barDelta\bx_k)^T\bnabla_{\bx}\mL_k + c_k^T\barDelta\blambda_k- \barmu_k\|c_k\|^2 -\nu\|G_k\bnabla_{\bx}\mL_k\|^2\\
& \stackrel{\mathclap{\eqref{equ:ran:Newton}}}{=} -(\barDelta\bx_k)^TB_k\barDelta \bx_k + c_k^T(\barDelta \blambda_k + \tDelta \blambda_k) - \barmu_k\|c_k\|^2 - \nu\|G_k\bnabla_{\bx}\mL_k\|^2.
\end{align*}
Let $\barDelta\bx_k = \barDelta \bu_k + \barDelta\bv_k$ with $\barDelta\bu_k \in \text{Image}(G_k^T)$ and $G_k\barDelta\bv_k = \0$, then
\begin{align*}
&\begin{pmatrix}
\bnabla_{\bx}\mL_{\barmu_k, \nu}^k\\
\bnabla_{\blambda}\mL_{\barmu_k, \nu}^k
\end{pmatrix}^T\begin{pmatrix}
\barDelta \bx_k\\
\barDelta \blambda_k
\end{pmatrix}\\
& =  -(\barDelta\bv_k)^TB_k\barDelta \bv_k - 2(\barDelta\bv_k)^TB_k\barDelta \bu_k - (\barDelta\bu_k)^TB_k\barDelta \bu_k + c_k^T(\barDelta \blambda_k + \tDelta \blambda_k) \\
&\quad - \barmu_k\|c_k\|^2 - \nu\|G_k\bnabla_{\bx}\mL_k\|^2\\
&\leq -\gamma_{RH}\nbr{\barDelta \bv_k}^2 + 2\|\barDelta\bv_k\|\|B_k\barDelta\bu_k\| -  (\barDelta\bu_k)^TB_k\barDelta \bu_k + c_k^T(\barDelta \blambda_k + \tDelta \blambda_k)\\
&\quad - \barmu_k\|c_k\|^2 - \nu\|G_k\bnabla_{\bx}\mL_k\|^2.
\end{align*}
Using the fact that $2\nbr{\barDelta\bv_k}\nbr{B_k\barDelta\bu_k} \leq \gamma_{RH}\|\barDelta\bv_k\|^2/4 + 4\|B_k\barDelta\bu_k\|^2/\gamma_{RH}$, and $\|\barDelta\bx_k\|^2 = \|\barDelta\bv_k\|^2 + \|\barDelta\bu_k\|^2$, we further have
\begin{align*}
&\begin{pmatrix}
\bnabla_{\bx}\mL_{\barmu_k, \nu}^k\\
\bnabla_{\blambda}\mL_{\barmu_k, \nu}^k
\end{pmatrix}^T\begin{pmatrix}
\barDelta \bx_k\\
\barDelta \blambda_k
\end{pmatrix}\\
& \leq -\frac{3\gamma_{RH}}{4}\|\barDelta \bv_k\|^2 + (\barDelta\bu_k)^T\rbr{\frac{4B_k^2}{\gamma_{RH}} - B_k}\barDelta \bu_k + c_k^T(\barDelta \blambda_k + \tDelta \blambda_k)\\
&\quad  - \barmu_k\|c_k\|^2 - \nu\|G_k\bnabla_{\bx}\mL_k\|^2\\
& = -\frac{3\gamma_{RH}}{4}\|\barDelta\bx_k\|^2 + (\barDelta\bu_k)^T\rbr{\frac{4B_k^2}{\gamma_{RH}}+ \frac{3\gamma_{RH}}{4}I - B_k}\barDelta \bu_k + c_k^T(\barDelta \blambda_k + \tDelta\blambda_k)\\
&\quad - \barmu_k\|c_k\|^2 - \nu\|G_k\bnabla_{\bx}\mL_k\|^2\\
& \leq -\frac{\gamma_{RH}\wedge \nu}{2}\nbr{\begin{pmatrix}
\barDelta\bx_k\\
G_k\bnabla_{\bx}\mL_k
\end{pmatrix}}^2 + D_k,
\end{align*}
where
\begin{multline*}
D_k =(\barDelta\bu_k)^T\rbr{\frac{4B_k^2}{\gamma_{RH}} + \frac{3\gamma_{RH}}{4}I - B_k}\barDelta \bu_k + c_k^T(\barDelta \blambda_k + \tDelta\blambda_k) \\ - \frac{\gamma_{RH}}{4}\|\barDelta\bx_k\|^2 
- \frac{\nu}{2}\|G_k\bnabla_{\bx}\mL_k\|^2 - \barmu_k\|c_k\|^2.
\end{multline*}
We now bound $\|\barDelta\bu_k\|$, $\|\barDelta\blambda_k\|$, and $\|\tDelta\blambda_k\|$. By \eqref{bound:M} and Assumption \ref{ass:A:2}, we~have
\begin{multline}\label{equ:bound:M}
\|\barM_k\| \leq \|M_k\| + \|\barM_k - M_k\|   \stackrel{\mathclap{\eqref{bound:M}}}{\leq} \|M_k\|+\sqrt{\kappa_{2,G}}\|\barH_k - \nabla^2f_k\| + \kappa_{\nabla_{\bx}^2c}\|\barg_k - \nabla f_k\| \\
\stackrel{\mathclap{\eqref{Prop:1}}}{\leq} \kappa_{M} + \Omega_2\sqrt{\kappa_{2,G}} + \Omega_1\kappa_{\nabla_{\bx}^2c} \eqqcolon \kappa_{\barM}.
\end{multline}
Since $\barDelta\bu_k \in \text{Image}(G_k^T)$ and $G_k\barDelta\bv_k = \0$, using $G_k\barDelta\bx_k = -c_k$ leads to~$\barDelta\bu_k = -G_k^T(G_kG_k^T)^{-1}c_k$. Thus, $\|\barDelta\bu_k\| \leq \|c_k\|/\sqrt{\kappa_{1,G}}$. Furthermore,
\begin{align*}
\|\barDelta\blambda_k\| & \stackrel{\eqref{equ:ran:Newton}}{=}\| (G_kG_k^T)^{-1}(G_k\bnabla_{\bx}\mL_k + \barM_k^T\barDelta\bx_k) \| \leq \frac{1}{\kappa_{1,G}}(\|G_k\bnabla_{\bx}\mL_k\| + \kappa_{\barM}\|\barDelta\bx_k\|),\\
\|\tDelta\blambda_k\| & \stackrel{\eqref{equ:ran:Newton}}{=}\| (G_kG_k^T)^{-1}(G_k\bnabla_{\bx}\mL_k + G_kB_k\barDelta\bx_k) \| \\
&\hskip5cm \leq \frac{1}{\kappa_{1,G}}(\|G_k\bnabla_{\bx}\mL_k\| + \sqrt{\kappa_{2,G}}\kappa_B\|\barDelta\bx_k\|).
\end{align*}
Using the above results and Assumption \ref{ass:A:1}, we can bound $D_k$ as
\begin{align*}
D_k & \leq  \cbr{\frac{1}{\kappa_{1,G}}\rbr{\frac{4\kappa_B^2}{\gamma_{RH}}+ \frac{3\gamma_{RH}}{4}+\kappa_B }-\barmu_k }\|c_k\|^2 + \frac{2}{\kappa_{1,G}}\|c_k\|\|G_k\bnabla_{\bx}\mL_k\| \\
&\quad + \frac{\kappa_{\barM} + \sqrt{\kappa_{2,G}}\kappa_B }{\kappa_{1,G}}\|c_k\|\|\barDelta\bx_k\| - \frac{\gamma_{RH}}{4}\|\barDelta\bx_k\|^2 -  \frac{\nu}{2}\|G_k\bnabla_{\bx}\mL_k\|^2\\
& \leq  \cbr{\frac{1}{\kappa_{1,G}}\rbr{\frac{4\kappa_B^2}{\gamma_{RH}}+ \frac{3\gamma_{RH}}{4}+\kappa_B } + \frac{2}{\nu\kappa_{1,G}^2} + \frac{(\kappa_{\barM} + \sqrt{\kappa_{2,G}}\kappa_B)^2}{\gamma_{RH}\kappa_{1,G}^2}-\barmu_k }\|c_k\|^2\\
& \leq \rbr{\frac{6\kappa_B^2}{\gamma_{RH} \kappa_{1,G}} + \frac{2+(\kappa_{\barM} + \sqrt{\kappa_{2,G}}\kappa_B)^2}{\kappa_{1,G}^2(\gamma_{RH}\wedge\nu)} -\barmu_k}\|c_k\|^2\\
& \leq \rbr{\frac{7(\kappa_{\barM} + \sqrt{\kappa_{2,G}}\kappa_B)^2}{\kappa_{1,G}^2(\gamma_{RH}\wedge\nu)} -\barmu_k}\|c_k\|^2,
\end{align*}
where the third inequality uses $\kappa_B\geq \gamma_{RH}$, and the last inequality uses $\kappa_M\wedge \kappa_{2,G}\wedge\kappa_B\geq 1$. Thus, as long as
\begin{equation}\label{pequ:mu:1}
\barmu_k \geq \frac{7(\kappa_{\barM} + \sqrt{\kappa_{2,G}}\kappa_B)^2}{\kappa_{1,G}^2(\gamma_{RH}\wedge\nu)} \eqqcolon \tmu_1 / \rho,
\end{equation}
we know $D_k\leq 0$ and hence the first condition in \eqref{equ:ran:cond:2} holds.

We now check the second condition. We have
\begin{align*}
\rbr{\frac{1}{\nu}I + G_k\barM_k}&(G_kG_k^T)^{-1}\bnabla_{\blambda}\mL_{\barmu_k, \nu}^k \\
& \stackrel{\eqref{N:13}}{=} \rbr{\frac{1}{\nu}I + G_k\barM_k}(G_kG_k^T)^{-1}c_k + \rbr{I + \nu G_k\barM_k}G_k\bnabla_{\bx}\mL_k\\
& \stackrel{\eqref{N:13}}{=} \rbr{\frac{1}{\nu}I + G_k\barM_k}(G_kG_k^T)^{-1}c_k + G_k\bnabla_{\bx}\mL_{\barmu_k, \nu}^k - \barmu_k G_kG_k^Tc_k,
\end{align*}
which implies
\begin{multline*}
\rbr{\barmu_k G_kG_k^T - \rbr{\frac{1}{\nu}I + G_k\barM_k}(G_kG_k^T)^{-1} }c_k \\= G_k\bnabla_{\bx}\mL_{\barmu_k, \nu}^k - \rbr{\frac{1}{\nu}I + G_k\barM_k}(G_kG_k^T)^{-1}\bnabla_{\blambda}\mL_{\barmu_k, \nu}^k
\end{multline*}
and, by \eqref{Prop:1} and \eqref{equ:bound:M},
\begin{equation*}
\rbr{\barmu_k\kappa_{1,G}-\frac{1 + \nu\kappa_{\barM}\sqrt{\kappa_{2,G}}}{\kappa_{1,G}\nu}  }\|c_k\|
\leq  \frac{1 + \nu\sqrt{\kappa_{2,G}}(\kappa_{M} + \kappa_{1,G}) }{\kappa_{1,G}\nu} \nbr{\begin{pmatrix}
	\bnabla_{\bx}\mL_{\barmu_k, \nu}^k\\
	\bnabla_{\blambda}\mL_{\barmu_k, \nu}^k.
	\end{pmatrix}}.
\end{equation*}
Therefore, as long as
\begin{equation}\label{pequ:mu:2}
\barmu_k \geq \frac{2 + 2\nu\sqrt{\kappa_{2,G}}(\kappa_{M} + \kappa_{1,G}) }{\kappa_{1,G}^2\nu} \eqqcolon \tmu_2/\rho,
\end{equation}
the second condition in \eqref{equ:ran:cond:2} holds. Combining \eqref{pequ:mu:1} and \eqref{pequ:mu:2} together, we see that if $\barmu_k \geq (\tmu_1\vee \tmu_2)/\rho$, then \eqref{equ:ran:cond:2} holds. Moreover, we see that the threshold is independent of $k$. As Algorithm~\ref{alg:ASto:SQP} increases $\barmu_k$ by a factor of $\rho$ in each while loop, there must exist $\barK<\infty$ such that $\barmu_k= \barmu_{\barK}$ for all $k\geq \barK$, and $\barmu_{\barK} \leq \tmu\coloneqq \tmu_1\vee \tmu_2$. This completes the proof. \qed
\end{proof}

Lemma~\ref{lem:9} suggests that, for each run of Algorithm~\ref{alg:ASto:SQP}, the merit function is invariant after certain number of iterations. This property is critical for establishing global convergence. Since we always let $k\rightarrow \infty$, we do not have to study the iteration before $\barK$. We note that the threshold~$\barK$ is random,~which~might be different for each run. However, it must exist and be finite. A similar result is shown in \cite{Berahas2021Sequential}. Finally, we note that $\barmu_{\barK}$ has a deterministic upper bound $\tmu$.

\subsection{Convergence analysis}

Our convergence analysis is based on the potential function
\begin{equation*}
\Phi_{\barmu_{\barK}, \nu, \omega}^k = \omega\mL_{\barmu_{\barK}, \nu}^k + \frac{(1-\omega)}{2}\barepsilon_k + \frac{(1-\omega)}{2}\baralpha_k\|\nabla\mL_{\barmu_{\barK}, \nu}^k\|^2,
\end{equation*}
where $\omega \in (0, 1)$ is a deterministic parameter specified later. This function is a linear combination of $\mL_{\barmu_{\barK}, \nu}^k$, $\barepsilon_k$, and $\baralpha_k\|\nabla\mL_{\barmu_{\barK}, \nu}^k\|^2$. Throughout the analysis, we condition on $\mF_{\barK}$ and only study the iterates after $\barK$ iterations, where $\barK$ is the iteration index determined in Lemma \ref{lem:9}, after which $\mu_k$ stabilizes. 

We begin by presenting a proposition that relates different critical quantities. We rely on these relations in later analysis. The results are shown by straightforward calculations and the proofs are deferred to Appendix \ref{pf:N:14}.

\begin{proposition}\label{prop:2}

Under Assumptions \ref{ass:A:1} and \ref{ass:A:2}, there exist constants $\{\Upsilon_i\}_{i=1}^4$ that are independent of parameters $(\alpha_{max}, \beta, \kappa_{grad},p_{grad},\kappa_f,p_f)$ such that
\begin{equation}\label{N:14}
\begin{aligned}
\nbr{\begin{pmatrix}
\bnabla_{\bx}\mL_k + \nu\barM_k G_k\bnabla_{\bx}\mL_k + G_k^Tc_k\\
\nu G_kG_k^TG_k\bnabla_{\bx}\mL_k
\end{pmatrix}} \leq \Upsilon_1\nbr{\begin{pmatrix}
\barDelta\bx_k\\
G_k\bnabla_{\bx}\mL_k
\end{pmatrix}},\\
\Upsilon_2\nbr{\begin{pmatrix}
\barDelta\bx_k\\
G_k\bnabla_{\bx}\mL_k
\end{pmatrix}} \leq \nbr{\begin{pmatrix}
\barDelta\bx_k\\
\barDelta\blambda_k
\end{pmatrix}} \leq \Upsilon_3\nbr{\begin{pmatrix}
\barDelta\bx_k\\
G_k\bnabla_{\bx}\mL_k
\end{pmatrix}},\\
\nbr{\begin{pmatrix}
\bnabla_{\bx}\mL_{\barmu_{\barK}, \nu}^k\\
\bnabla_{\blambda}\mL_{\barmu_{\barK}, \nu}^k
\end{pmatrix}} \leq \Upsilon_4\nbr{\begin{pmatrix}
\barDelta\bx_k\\
G_k\bnabla_{\bx}\mL_k
\end{pmatrix}}.
\end{aligned}
\end{equation}

\end{proposition}

\begin{proof}
See Appendix \ref{pf:N:14}. \qed
\end{proof}

We emphasize that constants $\{\Upsilon_i\}_{i=1}^4$ (whose expressions are provided in the appendix) are independent of probability parameters $(p_{grad},p_f)$. In Theorem \ref{thm:5}, we impose a condition on $p_f, p_{grad}$. Due to independence, the condition is not an implicit function of $p_f, p_{grad}$.

We have the following two lemmas, which connect the stochastic line search with deterministic line search.

\begin{lemma}\label{lem:8}
For $k\geq \barK$, we suppose the event $\mA_k\cap\mB_k$ happens, where $\mA_k$ is defined in \eqref{event:A_k} and $\mB_k$ is defined in \eqref{event:B_k}. If
\begin{equation*}
\baralpha_k\leq \frac{(1-\beta)(\gamma_{RH}\wedge\nu)}{2(\kappa_{\mL_{\tmu, \nu}} \Upsilon_3^2 + \kappa_{grad}\Upsilon_1\Upsilon_3 + \kappa_f\gamma_{RH}) },
\end{equation*}
where $\Upsilon_1, \Upsilon_3$ are given by Proposition \ref{prop:2} and $\kappa_{\mL_{\tmu, \nu}}$ is defined in \eqref{Prop:1} with $\tmu$ given by Lemma \ref{lem:9}, then the $k$-th step is a successful step (i.e. the Armijo condition is satisfied).	
\end{lemma}

\begin{proof}
Our proof resembles \cite[Lemma 4.2]{Paquette2020Stochastic}, but targets a different merit~function, and relies on the properties in Proposition \ref{prop:2}. On the events of $\mA_k$ and~$\mB_k$, the merit function and its gradient are precisely estimated. Since $\barmu_{\barK} \leq \tmu$ by Lemma \ref{lem:9}, $\kappa_{\mL_{\barmu_{\barK}, \nu}}\leq \kappa_{\mL_{\tmu, \nu}}$ by \eqref{Prop:1}. We apply the Taylor expansion and get
\begin{align}\label{N:16}
&\mL_{\barmu_{\barK}, \nu}^{s_k} \leq  \mL_{\barmu_{\barK}, \nu}^k + \baralpha_k\begin{pmatrix}
\nabla_{\bx}\mL_{\barmu_{\barK}, \nu}^k\\
\nabla_{\blambda}\mL_{\barmu_{\barK}, \nu}^k
\end{pmatrix}^T\begin{pmatrix}
\barDelta\bx_k\\
\barDelta\blambda_k
\end{pmatrix} + \frac{\kappa_{\mL_{\tmu, \nu}}\baralpha_k^2}{2}\nbr{\begin{pmatrix}
\barDelta\bx_k\\
\barDelta \blambda_k
\end{pmatrix}}^2 \nonumber\\
& =  \mL_{\barmu_{\barK}, \nu}^k + \baralpha_k\begin{pmatrix}
\nabla_{\bx}\mL_{\barmu_{\barK}, \nu}^k - \bnabla_{\bx}\mL_{\barmu_{\barK}, \nu}^k\\
\nabla_{\blambda}\mL_{\barmu_{\barK}, \nu}^k- \bnabla_{\blambda}\mL_{\barmu_{\barK}, \nu}^k
\end{pmatrix}^T\begin{pmatrix}
\barDelta\bx_k\\
\barDelta\blambda_k
\end{pmatrix} + \baralpha_k\begin{pmatrix}
\bnabla_{\bx}\mL_{\barmu_{\barK}, \nu}^k\\
\bnabla_{\blambda}\mL_{\barmu_{\barK}, \nu}^k
\end{pmatrix}^T\begin{pmatrix}
\barDelta\bx_k\\
\barDelta\blambda_k
\end{pmatrix}  \nonumber\\
& \quad + \frac{\kappa_{\mL_{\tmu, \nu}}\baralpha_k^2}{2}\nbr{\begin{pmatrix}
\barDelta\bx_k\\
\barDelta \blambda_k
\end{pmatrix}}^2 \nonumber\\
& \stackrel{\eqref{event:A_k}}{\leq}  \mL_{\barmu_{\barK}, \nu}^k  + \baralpha_k\begin{pmatrix}
\bnabla_{\bx}\mL_{\barmu_{\barK}, \nu}^k\\
\bnabla_{\blambda}\mL_{\barmu_{\barK}, \nu}^k
\end{pmatrix}^T\begin{pmatrix}
\barDelta\bx_k\\
\barDelta\blambda_k
\end{pmatrix}  + \frac{\kappa_{\mL_{\tmu, \nu}}\baralpha_k^2}{2}\nbr{\begin{pmatrix}
\barDelta\bx_k\\
\barDelta \blambda_k
\end{pmatrix}}^2 \nonumber\\
&\quad + \baralpha_k^2\kappa_{grad}\nbr{\begin{pmatrix}
\bnabla_{\bx}\mL_k + \nu\barM_kG_k\bnabla_{\bx}\mL_k + G_k^Tc_k\\
\nu G_kG_k^TG_k\bnabla_{\bx}\mL_k
\end{pmatrix}}\nbr{\begin{pmatrix}
\barDelta\bx_k\\
\barDelta\blambda_k
\end{pmatrix}} \nonumber\\
& \stackrel{\eqref{N:14}}{\leq}  \mL_{\barmu_{\barK}, \nu}^k  + \baralpha_k\begin{pmatrix}
\bnabla_{\bx}\mL_{\barmu_{\barK}, \nu}^k\\
\bnabla_{\blambda}\mL_{\barmu_{\barK}, \nu}^k
\end{pmatrix}^T\begin{pmatrix}
\barDelta\bx_k\\
\barDelta\blambda_k
\end{pmatrix}  + (\kappa_{\mL_{\tmu, \nu}} \Upsilon_3^2 + \kappa_{grad}\Upsilon_1\Upsilon_3)\baralpha_k^2\nbr{\begin{pmatrix}
\barDelta\bx_k\\
G_k\bnabla_{\bx}\mL_k
\end{pmatrix}}^2 \nonumber\\
& \hskip-3pt \stackrel{\eqref{equ:ran:cond:2}}{\leq}   \mL_{\barmu_{\barK}, \nu}^k + \baralpha_k \cbr{1 - \frac{2(\kappa_{\mL_{\tmu, \nu}} \Upsilon_3^2 + \kappa_{grad}\Upsilon_1\Upsilon_3)\baralpha_k}{\gamma_{RH}\wedge\nu}}\begin{pmatrix}
\bnabla_{\bx}\mL_{\barmu_{\barK}, \nu}^k\\
\bnabla_{\blambda}\mL_{\barmu_{\barK}, \nu}^k
\end{pmatrix}^T\begin{pmatrix}
\barDelta\bx_k\\
\barDelta\blambda_k
\end{pmatrix}.
\end{align}
Moreover, by event $\mB_k$ we have
\begin{align*}
&\barL_{\barmu_{\barK}, \nu}^{s_k} \stackrel{\eqref{event:B_k}}{\leq}  \mL_{\barmu_{\barK}, \nu}^{s_k} - \kappa_f\baralpha_k^2\begin{pmatrix}
\bnabla_{\bx}\mL_{\barmu_{\barK}, \nu}^k\\
\bnabla_{\blambda}\mL_{\barmu_{\barK}, \nu}^k
\end{pmatrix}^T\begin{pmatrix}
\barDelta\bx_k\\
\barDelta\blambda_k
\end{pmatrix}\\
& \stackrel{\eqref{N:16}}{\leq} \mL_{\barmu_{\barK}, \nu}^k + \baralpha_k \cbr{1 - \frac{2(\kappa_{\mL_{\tmu, \nu}} \Upsilon_3^2 + \kappa_{grad}\Upsilon_1\Upsilon_3) + \kappa_f\gamma_{RH}}{\gamma_{RH}\wedge\nu}\baralpha_k}\begin{pmatrix}
\bnabla_{\bx}\mL_{\barmu_{\barK}, \nu}^k\\
\bnabla_{\blambda}\mL_{\barmu_{\barK}, \nu}^k
\end{pmatrix}^T\begin{pmatrix}
\barDelta\bx_k\\
\barDelta\blambda_k
\end{pmatrix}\\
& \stackrel{\eqref{event:B_k}}{\leq}  \barL_{\barmu_{\barK}, \nu}^k + \baralpha_k \cbr{1 - \frac{2(\kappa_{\mL_{\tmu, \nu}} \Upsilon_3^2 + \kappa_{grad}\Upsilon_1\Upsilon_3 + \kappa_f\gamma_{RH}) }{\gamma_{RH}\wedge\nu}\baralpha_k}\begin{pmatrix}
\bnabla_{\bx}\mL_{\barmu_{\barK}, \nu}^k\\
\bnabla_{\blambda}\mL_{\barmu_{\barK}, \nu}^k
\end{pmatrix}^T\begin{pmatrix}
\barDelta\bx_k\\
\barDelta\blambda_k
\end{pmatrix}.
\end{align*}
Thus, if
\begin{equation*}
1 - \frac{2(\kappa_{\mL_{\tmu, \nu}} \Upsilon_3^2 + \kappa_{grad}\Upsilon_1\Upsilon_3 + \kappa_f\gamma_{RH}) }{\gamma_{RH}\wedge\nu}\baralpha_k \geq \beta \Longleftrightarrow \baralpha_k\leq \frac{(1-\beta)(\gamma_{RH}\wedge\nu)}{2(\kappa_{\mL_{\tmu, \nu}} \Upsilon_3^2 + \kappa_{grad}\Upsilon_1\Upsilon_3 + \kappa_f\gamma_{RH})},
\end{equation*}
we have
\begin{equation*}
\barL_{\barmu_{\barK}, \nu}^{s_k} \leq \barL_{\barmu_{\barK}, \nu}^k + \baralpha_k\beta\begin{pmatrix}
\bnabla_{\bx}\mL_{\barmu_{\barK}, \nu}^k\\
\bnabla_{\blambda}\mL_{\barmu_{\barK}, \nu}^k
\end{pmatrix}^T\begin{pmatrix}
\barDelta\bx_k\\
\barDelta\blambda_k
\end{pmatrix},
\end{equation*}
which completes the proof. \qed
\end{proof}

\begin{lemma}\label{lem:A:2}
For $k\geq K$, we suppose $\mB_k$ happens. If the $k$-th step is a successful step, then
\begin{equation*}
\mL_{\barmu_{\barK}, \nu}^{s_k} \leq \mL_{\barmu_{\barK}, \nu}^k + \frac{\baralpha_k\beta}{2}\begin{pmatrix}
\bnabla_{\bx}\mL_{\barmu_{\barK}, \nu}^k\\
\bnabla_{\blambda}\mL_{\barmu_{\barK}, \nu}^k
\end{pmatrix}^T\begin{pmatrix}
\barDelta\bx_k\\
\barDelta\blambda_k
\end{pmatrix}.
\end{equation*}
\end{lemma}

\begin{proof}
Since $\mB_k$ happens and the $k$-th step is successful, we have
\begin{align*}
\mL_{\barmu_{\barK}, \nu}^{s_k} & \stackrel{\eqref{event:B_k}}{\leq}  \barL_{\barmu_{\barK}, \nu}^{s_k} -\kappa_f\baralpha_k^2\begin{pmatrix}
\bnabla_{\bx}\mL_{\barmu_{\barK}, \nu}^k\\
\bnabla_{\blambda}\mL_{\barmu_{\barK}, \nu}^k
\end{pmatrix}^T\begin{pmatrix}
\barDelta\bx_k\\
\barDelta\blambda_k
\end{pmatrix} \\
& \stackrel{\eqref{equ:ran:Armijo}}{\leq}  \barL_{\barmu_K, \nu}^k + \baralpha_k\beta\begin{pmatrix}
\bnabla_{\bx}\mL_{\barmu_{\barK}, \nu}^k\\
\bnabla_{\blambda}\mL_{\barmu_{\barK}, \nu}^k
\end{pmatrix}^T\begin{pmatrix}
\barDelta\bx_k\\
\barDelta\blambda_k
\end{pmatrix}  -\kappa_f\baralpha_k^2\begin{pmatrix}
\bnabla_{\bx}\mL_{\barmu_K, \nu}^k\\
\bnabla_{\blambda}\mL_{\barmu_K, \nu}^k
\end{pmatrix}^T\begin{pmatrix}
\barDelta\bx_k\\
\barDelta\blambda_k
\end{pmatrix} \\
& \stackrel{\eqref{event:B_k}}{\leq}  \mL_{\barmu_{\barK}, \nu}^k + \baralpha_k\beta\begin{pmatrix}
\bnabla_{\bx}\mL_{\barmu_{\barK}, \nu}^k\\
\bnabla_{\blambda}\mL_{\barmu_{\barK}, \nu}^k
\end{pmatrix}^T\begin{pmatrix}
\barDelta\bx_k\\
\barDelta\blambda_k
\end{pmatrix} \rbr{1 - \frac{2\kappa_f\baralpha_k}{\beta}}\\
& \leq  \mL_{\barmu_{\barK}, \nu}^k + \baralpha_k\beta\begin{pmatrix}
\bnabla_{\bx}\mL_{\barmu_{\barK}, \nu}^k\\
\bnabla_{\blambda}\mL_{\barmu_{\barK}, \nu}^k
\end{pmatrix}^T\begin{pmatrix}
\barDelta\bx_k\\
\barDelta\blambda_k
\end{pmatrix} \rbr{1 - \frac{2\kappa_f\alpha_{max}}{\beta}},
\end{align*}
where the last inequality is due to $\baralpha_k\leq \alpha_{max}$ from Algorithm \ref{alg:ASto:SQP}. Using the input condition on $\kappa_f$ in Line 1 of Algorithm \ref{alg:ASto:SQP} completes the proof. \qed
\end{proof}

Using the above two lemmas, we establish the one-step error recursion for $\Phi_{\barmu_{\barK}, \nu, \omega}^k$. Our analytical structure follows \cite[Theorem 4.6]{Paquette2020Stochastic} but generalizes it to a different merit function for studying equality-constrained problems. We consider three cases of approximation of Algorithm \ref{alg:ASto:SQP}: $\mA_k\cap \mB_k$, $\mA_k^c\cap \mB_k$, and $\mB_k^c$. We study them in the next three~lemmas.

\begin{lemma}\label{lem:4}
For $k \geq \barK$, we suppose the event $\mA_k\cap\mB_k$ happens. Let
\begin{equation}\label{equ:cond:omega:1}
\frac{1-\omega}{\omega} \leq \frac{\beta(\gamma_{RH}\wedge\nu)}{32\rho \cbr{\kappa_{\mL_{\tmu, \nu}}\alpha_{max}\Upsilon_3 \vee \rbr{\kappa_{grad}\alpha_{max}\Upsilon_1 + \Upsilon_4} }^2}\wedge \frac{1}{4(\rho-1)}
\end{equation}
where $\Upsilon_1, \Upsilon_3, \Upsilon_4$ are given by Proposition \ref{prop:2} and $\kappa_{\mL_{\tmu, \nu}}$ is defined in \eqref{Prop:1} with $\tmu$ given by Lemma \ref{lem:9}. Then
\begin{equation*}
\Phi_{\barmu_{\barK}, \nu, \omega}^{k+1} - \Phi_{\barmu_{\barK}, \nu, \omega}^k = -\frac{1}{2}(1-\omega)\rbr{1-\frac{1}{\rho}}\rbr{\barepsilon_k + \baralpha_k\nbr{\begin{pmatrix}
\nabla_{\bx}\mL_{\barmu_{\barK}, \nu}^k\\
\nabla_{\blambda}\mL_{\barmu_{\barK}, \nu}^k
\end{pmatrix}}^2}.
\end{equation*}	
\end{lemma}

\begin{proof}

For each iteration, we have the following three types of steps.

\noindent\textbf{Case 1: reliable step.} We have
\begin{align}\label{N:17}
&\nbr{\begin{pmatrix}
\nabla_{\bx}\mL_{\barmu_{\barK}, \nu}^k\\
\nabla_{\blambda}\mL_{\barmu_{\barK}, \nu}^k
\end{pmatrix}} \leq \nbr{\begin{pmatrix}
\bnabla_{\bx}\mL_{\barmu_{\barK}, \nu}^k - \nabla_{\bx}\mL_{\barmu_{\barK}, \nu}^k\\
\bnabla_{\blambda}\mL_{\barmu_{\barK}, \nu}^k - \nabla_{\blambda}\mL_{\barmu_{\barK}, \nu}^k
\end{pmatrix}} + \nbr{\begin{pmatrix}
\bnabla_{\bx}\mL_{\barmu_{\barK}, \nu}^k\\
\bnabla_{\blambda}\mL_{\barmu_{\barK}, \nu}^k
\end{pmatrix}} \nonumber\\
& \hskip1cm \stackrel{\mathclap{\eqref{event:A_k}}}{\leq}  \kappa_{grad}\baralpha_k\nbr{\begin{pmatrix}
	\bnabla_{\bx}\mL_k + \nu\barM_k G_k\bnabla_{\bx}\mL_k + G_k^Tc_k\\
	\nu G_kG_k^TG_k\bnabla_{\bx}\mL_k
	\end{pmatrix}} + \nbr{\begin{pmatrix}
	\bnabla_{\bx}\mL_{\barmu_{\barK}, \nu}^k\\
	\bnabla_{\blambda}\mL_{\barmu_{\barK}, \nu}^k
	\end{pmatrix}} \nonumber\\
&\hskip1cm \stackrel{\mathclap{\eqref{N:14}}}{\leq}   \rbr{\kappa_{grad}\alpha_{max}\Upsilon_1 + \Upsilon_4} \nbr{\begin{pmatrix}
	\barDelta\bx_k\\
	G_k\bnabla_{\bx}\mL_k
	\end{pmatrix}}.
\end{align}
By Lemma \ref{lem:A:2} and the reliability of the step, we further have
\begin{align}\label{pequ:8}
&\mL_{\barmu_{\barK}, \nu}^{k+1} -\mL_{\barmu_{\barK}, \nu}^k \leq  \frac{\baralpha_k\beta}{2}\begin{pmatrix}
\bnabla_{\bx}\mL_{\barmu_{\barK}, \nu}^k\\
\bnabla_{\blambda}\mL_{\barmu_{\barK}, \nu}^k
\end{pmatrix}^T\begin{pmatrix}
\barDelta\bx_k\\
\barDelta\blambda_k
\end{pmatrix} \stackrel{\eqref{equ:sufficient:dec}}{\leq} \frac{\baralpha_k\beta}{4}\begin{pmatrix}
\bnabla_{\bx}\mL_{\barmu_{\barK}, \nu}^k\\
\bnabla_{\blambda}\mL_{\barmu_{\barK}, \nu}^k
\end{pmatrix}^T\begin{pmatrix}
\barDelta\bx_k\\
\barDelta\blambda_k
\end{pmatrix}  - \frac{\barepsilon_k}{4} \nonumber\\
& \stackrel{\eqref{equ:ran:cond:2}}{\leq}  -\frac{\baralpha_k\beta(\gamma_{RH}\wedge\nu)}{8}\nbr{\begin{pmatrix}
\barDelta\bx_k\\
G_k\bnabla_{\bx}\mL_k
\end{pmatrix}}^2 - \frac{\barepsilon_k}{4} \stackrel{\eqref{N:17}}{\leq} -\frac{\baralpha_k\beta(\gamma_{RH}\wedge\nu)}{16}\nbr{\begin{pmatrix}
\barDelta\bx_k\\
G_k\bnabla_{\bx}\mL_k
\end{pmatrix}}^2 \nonumber\\
& \quad\quad - \frac{\baralpha_k\beta(\gamma_{RH}\wedge\nu)}{16\rbr{\kappa_{grad}\alpha_{max}\Upsilon_1 + \Upsilon_4}^2 }\nbr{\begin{pmatrix}
\nabla_{\bx}\mL_{\barmu_{\barK}, \nu}^k\\
\nabla_{\blambda}\mL_{\barmu_{\barK}, \nu}^k
\end{pmatrix}}^2 - \frac{\barepsilon_k}{4}.
\end{align}
Moreover, by Line 13 of Algorithm \ref{alg:ASto:SQP}, $\barepsilon_{k+1} - \barepsilon_k = (\rho-1)\barepsilon_k$ and, by Taylor expansion,
\begin{align}\label{pequ:9}
\baralpha_{k+1}\nbr{\begin{pmatrix}
	\nabla_{\bx}\mL_{\barmu_{\barK}, \nu}^{k+1}\\
	\nabla_{\blambda}\mL_{\barmu_{\barK}, \nu}^{k+1}
	\end{pmatrix}}^2& -  \baralpha_k\nbr{\begin{pmatrix}
	\nabla_{\bx}\mL_{\barmu_{\barK}, \nu}^k\\
	\nabla_{\blambda}\mL_{\barmu_{\barK}, \nu}^k
	\end{pmatrix}}^2 \nonumber\\
&\hskip-2.5cm \leq  2\baralpha_{k+1}\rbr{ \nbr{\begin{pmatrix}
		\nabla_{\bx}\mL_{\barmu_{\barK}, \nu}^k\\
		\nabla_{\blambda}\mL_{\barmu_{\barK}, \nu}^k
		\end{pmatrix}}^2 + \kappa_{\mL_{\tmu, \nu}}^2\baralpha_k^2\nbr{\begin{pmatrix}
		\barDelta\bx_k\\
		\barDelta\blambda_k
		\end{pmatrix}}^2 } \nonumber\\
&\hskip-2.5cm \leq  2\rho\baralpha_k\rbr{ \nbr{\begin{pmatrix}
		\nabla_{\bx}\mL_{\barmu_{\barK}, \nu}^k\\
		\nabla_{\blambda}\mL_{\barmu_{\barK}, \nu}^k
		\end{pmatrix}}^2 + \kappa_{\mL_{\tmu, \nu}}^2\alpha_{max}^2\Upsilon_3^2\nbr{\begin{pmatrix}
		\barDelta\bx_k\\
		G_k\bnabla_{\bx}\mL_k
		\end{pmatrix}}^2 }.
\end{align}
The last inequality uses $\baralpha_{k+1} \leq \rho\baralpha_k$, $\baralpha_k\leq \alpha_{max}$, and \eqref{N:14}. Combining~the~above three displays,
\begin{multline}\label{pequ:11}
\Phi_{\barmu_{\barK}, \nu, \omega}^{k+1} - \Phi_{\barmu_{\barK}, \nu, \omega}^k \leq -\frac{\beta(\gamma_{RH}\wedge\nu)\omega}{32}\baralpha_k\nbr{\begin{pmatrix}
\barDelta\bx_k\\
G_k\bnabla_{\bx}\mL_k
\end{pmatrix}}^2 \\
- \frac{\beta(\gamma_{RH}\wedge\nu)\omega}{32\rbr{\kappa_{grad}\alpha_{max}\Upsilon_1 + \Upsilon_4}^2 } \baralpha_k\nbr{\begin{pmatrix}
\nabla_{\bx}\mL_{\barmu_{\barK}, \nu}^k\\
\nabla_{\blambda}\mL_{\barmu_{\barK}, \nu}^k
\end{pmatrix}}^2 - \frac{\omega\barepsilon_k}{8},
\end{multline}
since 
\begin{align*}
-\frac{\omega\beta(\gamma_{RH}\wedge\nu)}{16} + (1-\omega)\rho\kappa_{\mL_{\tmu, \nu}}^2\alpha_{max}^2\Upsilon_3^2 \leq&  -\frac{\omega\beta(\gamma_{RH}\wedge\nu)}{32},\\
-\frac{\omega\beta(\gamma_{RH}\wedge\nu)}{16\rbr{\kappa_{grad}\alpha_{max}\Upsilon_1 + \Upsilon_4}^2 } + (1-\omega)\rho \leq & -\frac{\omega\beta(\gamma_{RH}\wedge\nu)}{32\rbr{\kappa_{grad}\alpha_{max}\Upsilon_1 + \Upsilon_4}^2 },\\
-\frac{\omega}{4} + \frac{1}{2}(1-\omega)(\rho-1) \leq & -\frac{\omega}{8},
\end{align*}
as implied by the condition \eqref{equ:cond:omega:1}.

\noindent\textbf{Case 2: unreliable step.} For unreliable step, we apply Lemma \ref{lem:A:2} and \eqref{pequ:8} is changed to
\begin{align*}
\mL_{\barmu_{\barK}, \nu}^{k+1} - \mL_{\barmu_{\barK}, \nu}^k & \leq  \frac{\baralpha_k\beta}{2}\begin{pmatrix}
\bnabla_{\bx}\mL_{\barmu_{\barK}, \nu}^k\\
\bnabla_{\blambda}\mL_{\barmu_{\barK}, \nu}^k
\end{pmatrix}^T\begin{pmatrix}
\barDelta\bx_k\\
\barDelta\blambda_k
\end{pmatrix} \stackrel{\eqref{equ:ran:cond:2}}{\leq} -\frac{\baralpha_k\beta(\gamma_{RH}\wedge\nu)}{4}\nbr{\begin{pmatrix}
\barDelta\bx_k\\
G_k\bnabla_{\bx}\mL_k
\end{pmatrix}}^2 \nonumber\\
& \hskip-2.3cm\stackrel{\eqref{N:17}}{\leq}  -\frac{\baralpha_k\beta(\gamma_{RH}\wedge\nu)}{8}\nbr{\begin{pmatrix}
	\barDelta\bx_k\\
	G_k\bnabla_{\bx}\mL_k
	\end{pmatrix}}^2 - \frac{\baralpha_k\beta(\gamma_{RH}\wedge\nu)}{8\rbr{\kappa_{grad}\alpha_{max}\Upsilon_1 + \Upsilon_4}^2 }\nbr{\begin{pmatrix}
	\nabla_{\bx}\mL_{\barmu_{\barK}, \nu}^k\\
	\nabla_{\blambda}\mL_{\barmu_{\barK}, \nu}^k
	\end{pmatrix}}^2.
\end{align*}
By Line 15 of Algorithm \ref{alg:ASto:SQP}, $\barepsilon_{k+1} - \barepsilon_k = -(1-1/\rho)\barepsilon_k$, while \eqref{pequ:9} is still the same. Thus, under the condition \eqref{equ:cond:omega:1}, we obtain
\begin{multline}\label{pequ:12}
\Phi_{\barmu_{\barK}, \nu, \omega}^{k+1} - \Phi_{\barmu_{\barK}, \nu, \omega}^k
\leq -\frac{\beta(\gamma_{RH}\wedge\nu)\omega}{32}\baralpha_k\nbr{\begin{pmatrix}
	\barDelta\bx_k\\
	G_k\bnabla_{\bx}\mL_k
	\end{pmatrix}}^2 \\
-\frac{\beta(\gamma_{RH}\wedge\nu)\omega}{32\rbr{\kappa_{grad}\alpha_{max}\Upsilon_1 + \Upsilon_4}^2 } \baralpha_k\nbr{\begin{pmatrix}
	\nabla_{\bx}\mL_{\barmu_{\barK}, \nu}^k\\
	\nabla_{\blambda}\mL_{\barmu_{\barK}, \nu}^k
	\end{pmatrix}}^2 - \frac{1}{2}(1-\omega)\rbr{1-\frac{1}{\rho}}\barepsilon_k.
\end{multline}
\noindent\textbf{Case 3: unsuccessful step.} Noting that $(\bx_{k+1}, \blambda_{k+1}) = (\bx_k, \blambda_k)$,
\begin{multline}\label{pequ:13}
\Phi_{\barmu_{\barK}, \nu, \omega}^{k+1} - \Phi_{\barmu_{\barK}, \nu, \omega}^k \\
= -\frac{1}{2}(1-\omega)\rbr{1-\frac{1}{\rho}}\barepsilon_k - \frac{1}{2}(1-\omega)\rbr{1 - \frac{1}{\rho}} \baralpha_k\nbr{\begin{pmatrix}
	\nabla_{\bx}\mL_{\barmu_{\barK}, \nu}^k\\
	\nabla_{\blambda}\mL_{\barmu_{\barK}, \nu}^k
	\end{pmatrix}}^2.
\end{multline}
Noting that under the condition \eqref{equ:cond:omega:1}, we have
\begin{gather*}
-\frac{\beta(\gamma_{RH}\wedge\nu)\omega}{32\rbr{\kappa_{grad}\alpha_{max}\Upsilon_1 + \Upsilon_4}^2 } \leq - \frac{1}{2}(1-\omega)\rbr{1 - \frac{1}{\rho}},\\
-\frac{\omega}{8} \leq -\frac{1}{2}(1-\omega)\rbr{1-\frac{1}{\rho}}.
\end{gather*}
Therefore, \eqref{pequ:13} holds for all three cases, which completes the proof. \qed
\end{proof}

The proofs of Lemmas \ref{lem:5} and \ref{lem:7} are similar to Lemma \ref{lem:4}, hence deferred to the appendix.

\begin{lemma}\label{lem:5}
For $k\geq \barK$, we suppose the event $\mA_k^c\cap\mB_k$ happens. Suppose $\omega$ satisfies \eqref{equ:cond:omega:1}. Then
\begin{equation*}
\Phi_{\barmu_{\barK}, \nu, \omega}^{k+1} - \Phi_{\barmu_{\barK}, \nu, \omega}^k \leq \rho(1-\omega)\baralpha_k\nbr{\begin{pmatrix}
	\nabla_{\bx}\mL_{\barmu_{\barK}, \nu}^k\\
	\nabla_{\blambda}\mL_{\barmu_{\barK}, \nu}^k
	\end{pmatrix}}^2.
\end{equation*}	
\end{lemma}

\begin{proof}
See Appendix \ref{pf:lem:5}. \qed
\end{proof}

\begin{lemma}\label{lem:7}
For $k\geq \barK$, we suppose the event $\mB_k^c$ happens. Suppose $\omega$ satisfies \eqref{equ:cond:omega:1}. Then
\begin{multline*}
\Phi_{\barmu_{\barK}, \nu, \omega}^{k+1} - \Phi_{\barmu_{\barK}, \nu, \omega}^k
\leq  \rho(1-\omega)\baralpha_k\nbr{\begin{pmatrix}
	\nabla_{\bx}\mL_{\barmu_{\barK}, \nu}^k\\
	\nabla_{\blambda}\mL_{\barmu_{\barK}, \nu}^k
	\end{pmatrix}}^2 \\+ \omega( |\barL_{\barmu_{\barK}, \nu}^{s_k} - \mL_{\barmu_{\barK}, \nu}^{s_k}| + |\barL_{\barmu_{\barK}, \nu}^k - \mL_{\barmu_{\barK}, \nu}^k|).
\end{multline*}	
\end{lemma}

\begin{proof}
See Appendix \ref{pf:lem:7}. \qed
\end{proof}

The above three lemmas show how $\Phi_{\barmu_{\barK}, \nu, \omega}^k$ changes in each iteration. We observe from Lemmas \ref{lem:5} and \ref{lem:7} that if either the function or the gradient are imprecisely estimated, then there is no guarantee that $\Phi_{\barmu_{\barK}, \nu, \omega}^k$ will decrease. The following theorem shows that the increase of $\Phi_{\barmu_{\barK}, \nu, \omega}^k$ can be controlled, when the probabilities $p_f, p_{grad}$ of bad events are small enough. In particular, in expectation, $\Phi_{\barmu_{\barK}, \nu, \omega}^k$ always decreases. Our analysis is conditional on $\mF_{\barK}$, so that $(\bx_{\barK+1}, \blambda_{\barK+1})$~and $\barmu_{\barK}$ are fixed. Our proof resembles \cite[Theorem 4.6]{Paquette2020Stochastic}, but the conditions~on $p_{grad}, p_f$ are different. To simplify notation, we let $\Phi_{\barmu_{\barK}}^k$ denote $\Phi_{\barmu_{\barK}, \nu, \omega}^k$. Recall that $\omega$ is not a parameter of the algorithm.

\begin{theorem}[One-step error recursion]\label{thm:5}
For $k>\barK$, suppose $\omega$ satisfies \eqref{equ:cond:omega:1} and $p_f, p_{grad}$ satisfy
\begin{equation}\label{cond:pp}
\frac{p_{grad} + \sqrt{p_f}}{(1-p_{grad})(1-p_f)} \leq \frac{\rho-1}{8\rho}\rbr{\frac{1}{\rho} \wedge \frac{1-\omega}{\omega}}.
\end{equation}
Then,
\begin{multline*}
\mE[\Phi_{\barmu_{\barK}}^{k+1} - \Phi_{\barmu_{\barK}}^k \mid \mF_{k-1}] \\
\leq  -\frac{1}{4}(1-p_{grad})(1-p_f)(1-\omega)\rbr{1-\frac{1}{\rho}}\rbr{\barepsilon_k + \baralpha_k\nbr{\begin{pmatrix}
		\nabla_{\bx}\mL_{\barmu_{\barK}, \nu}^k\\
		\nabla_{\blambda}\mL_{\barmu_{\barK}, \nu}^k
		\end{pmatrix}}^2}.
\end{multline*}
\end{theorem}

\begin{proof}

By \eqref{equ:ran:cond:1} and \eqref{equ:ran:cond:3}, we have
\begin{equation}\label{NN:D:21}
P\rbr{\mA_k \mid \mF_{k-1}} \geq 1-p_{grad} \quad \text{ and } \quad P\rbr{\mB_k \mid \mF_{k-0.5}} \geq 1-p_f.
\end{equation}
We further have
\begin{align}\label{pequ:19}
\mE[\Phi_{\barmu_{\barK}}^{k+1} & - \Phi_{\barmu_{\barK}}^k \mid \mF_{k-1}] 
=  \mE[\b1_{\mA_k\cap \mB_k}(\Phi_{\barmu_{\barK}}^{k+1} - \Phi_{\barmu_{\barK}}^k) \mid \mF_{k-1}] \nonumber\\
+& \mE[\b1_{\mA_k^C\cap \mB_k}(\Phi_{\barmu_{\barK}}^{k+1} - \Phi_{\barmu_{\barK}}^k )\mid \mF_{k-1}] 
+ \mE[\b1_{\mB_k^C}(\Phi_{\barmu_{\barK}}^{k+1} - \Phi_{\barmu_{\barK}}^k) \mid \mF_{k-1}].
\end{align}
Using Lemma \ref{lem:4},
\begin{align}\label{pequ:20}
&\mE[\b1_{\mA_k\cap \mB_k}(\Phi_{\barmu_{\barK}}^{k+1} - \Phi_{\barmu_{\barK}}^k) \mid \mF_{k-1}] \nonumber\\
& \leq -\frac{1-\omega}{2}\mE\sbr{\b1_{\mA_k\cap \mB_k} \mid \mF_{k-1}}  \rbr{1-\frac{1}{\rho}}\rbr{\barepsilon_k + \baralpha_k\nbr{\begin{pmatrix}
		\nabla_{\bx}\mL_{\barmu_{\barK}, \nu}^k\\
		\nabla_{\blambda}\mL_{\barmu_{\barK}, \nu}^k
		\end{pmatrix}}^2} \nonumber\\
& = -\frac{1-\omega}{2}\mE\sbr{\b1_{\mA_k}\mE\sbr{\b1_{\mB_k} \mid \mF_{k-0.5}} \mid \mF_{k-1}}  \rbr{1-\frac{1}{\rho}}\rbr{\barepsilon_k + \baralpha_k\nbr{\begin{pmatrix}
		\nabla_{\bx}\mL_{\barmu_{\barK}, \nu}^k\\
		\nabla_{\blambda}\mL_{\barmu_{\barK}, \nu}^k
		\end{pmatrix}}^2} \nonumber\\
& \stackrel{\mathclap{\eqref{NN:D:21}}}{\leq} -\frac{1}{2}(1-p_{grad})(1-p_f)(1-\omega)\rbr{1-\frac{1}{\rho}}\rbr{\barepsilon_k + \baralpha_k\nbr{\begin{pmatrix}
		\nabla_{\bx}\mL_{\barmu_{\barK}, \nu}^k\\
		\nabla_{\blambda}\mL_{\barmu_{\barK}, \nu}^k
		\end{pmatrix}}^2}.
\end{align}
Using Lemma \ref{lem:5},
\begin{align}\label{pequ:21}
\mE[\b1_{\mA_k^C\cap \mB_k}(\Phi_{\barmu_{\barK}}^{k+1} - \Phi_{\barmu_{\barK}}^k )\mid \mF_{k-1}] & \leq  (1-\omega)\mE[\b1_{\mA_k^C\cap \mB_k} \mid \mF_{k-1}] \rho\baralpha_k\nbr{\begin{pmatrix}
	\nabla_{\bx}\mL_{\barmu_{\barK}, \nu}^k\\
	\nabla_{\blambda}\mL_{\barmu_{\barK}, \nu}^k
	\end{pmatrix}}^2 \nonumber\\
& \leq  (1-\omega)\mE[\b1_{\mA_k^C} \mid \mF_{k-1}] \rho\baralpha_k\nbr{\begin{pmatrix}
	\nabla_{\bx}\mL_{\barmu_{\barK}, \nu}^k\\
	\nabla_{\blambda}\mL_{\barmu_{\barK}, \nu}^k
	\end{pmatrix}}^2 \nonumber\\
&  \stackrel{\mathclap{\eqref{NN:D:21}}}{\leq}  p_{grad}(1-\omega)\rho\baralpha_k\nbr{\begin{pmatrix}
	\nabla_{\bx}\mL_{\barmu_{\barK}, \nu}^k\\
	\nabla_{\blambda}\mL_{\barmu_{\barK}, \nu}^k
	\end{pmatrix}}^2.
\end{align}
Using Lemma \ref{lem:7},
\begin{align}\label{pequ:22}
\mE[\b1_{\mB_k^C}&(\Phi_{\barmu_{\barK}}^{k+1} - \Phi_{\barmu_{\barK}}^k) \mid \mF_{k-1}] \nonumber\\
& \leq  p_f(1-\omega)\rho\baralpha_k\nbr{\begin{pmatrix}
	\nabla_{\bx}\mL_{\barmu_{\barK}, \nu}^k\\	\nabla_{\blambda}\mL_{\barmu_{\barK}, \nu}^k
	\end{pmatrix}}^2  + \omega\big(\mE[\b1_{\mB_k^C}|\barL_{\barmu_{\barK}, \nu}^{k} - \mL_{\barmu_{\barK}, \nu}^{k}| \mid \mF_{k-1}] \nonumber\\
& \quad + \mE[\b1_{\mB_k^C}|\barL_{\barmu_{\barK}, \nu}^{s_k} - \mL_{\barmu_{\barK}, \nu}^{s_k}| \mid \mF_{k-1}] \big) \nonumber\\
&  \stackrel{\mathclap{\eqref{equ:ran:cond:4}}}{\leq}  p_f(1-\omega)\rho\baralpha_k\nbr{\begin{pmatrix}
	\nabla_{\bx}\mL_{\barmu_{\barK}, \nu}^k\\	\nabla_{\blambda}\mL_{\barmu_{\barK}, \nu}^k
	\end{pmatrix}}^2 + 2\omega\sqrt{p_f}\barepsilon_k.
\end{align}
The last inequality also uses H\"older's inequality. Combining \eqref{pequ:19}, \eqref{pequ:20}, \eqref{pequ:21}, \eqref{pequ:22}, we obtain
\begin{align*}
\mE[\Phi_{\barmu_{\barK}}^{k+1} & - \Phi_{\barmu_{\barK}}^k \mid \mF_{k-1}]\\
& \leq  -\frac{1}{2}(1-p_{grad})(1-p_f)(1-\omega)\rbr{1-\frac{1}{\rho}}\rbr{\barepsilon_k + \baralpha_k\nbr{\begin{pmatrix}
		\nabla_{\bx}\mL_{\barmu_{\barK}, \nu}^k\\
		\nabla_{\blambda}\mL_{\barmu_{\barK}, \nu}^k
		\end{pmatrix}}^2}\\
& \quad  + (p_{grad} + p_f)(1-\omega)\rho\baralpha_k\nbr{\begin{pmatrix}
	\nabla_{\bx}\mL_{\barmu_{\barK}, \nu}^k\\
	\nabla_{\blambda}\mL_{\barmu_{\barK}, \nu}^k
	\end{pmatrix}}^2 + 2\omega\sqrt{p_f}\barepsilon_k.
\end{align*}
Under the condition \eqref{cond:pp}, we have
\begin{equation*}
(p_{grad} + p_f)(1-\omega)\rho  \vee 2\omega \sqrt{p_f}  \leq   \frac{1}{4}(1-p_{grad})(1-p_f)(1-\omega)\rbr{1-\frac{1}{\rho}}.
\end{equation*}
Combining the above two displays completes the proof.\qed
\end{proof}

Based on the one-step error recursion, the following theorem shows the convergence of $\baralpha_k\|\nabla\mL_k\|^2$.

\begin{theorem}\label{thm:6}
Suppose the conditions of Theorem \ref{thm:5} are satisfied. Then, almost surely, we have $\lim\limits_{k\rightarrow \infty}\baralpha_k \|\nabla\mL_k\|^2 = 0$.	
\end{theorem}

\begin{proof}
By Lemma \ref{lem:9}, for any realization of the iteration sequence $\{(\bx_k, \blambda_k)\}_k$, with probability 1 there exists $\barK<\infty$ such that $\barmu_k$ is stabilized after $\barK$. Note that Algorithm \ref{alg:ASto:SQP} ensures that for $k> \barK$,
\begin{align*}
\|c_k\| = \mE\sbr{\|c_k\| \mid \mF_{k-1}}& \leq \mE[ \|\bnabla\mL_{\barmu_{\barK}, \nu}^k \| \mid \mF_{k-1}] \\
& \leq  \|\nabla\mL_{\barmu_{\barK}, \nu}^k\| + \mE[ \|\bnabla\mL_{\barmu_{\barK}, \nu}^k - \nabla\mL_{\barmu_{\barK}, \nu}^k\| \mid \mF_{k-1}]\\
& \stackrel{\mathclap{\eqref{N:18}}}{\leq}  \|\nabla\mL_{\barmu_{\barK}, \nu}^k\| + \Upsilon_5(\mE_{\xi^k_g}\sbr{\|\barg_k - \nabla f_k\| + \|\barH_k - \nabla^2f_k\|} ),
\end{align*}
where we let $\Upsilon_5 \coloneqq 1 + \nu\kappa_{2,G}^{3/2} + \nu\kappa_{2,G}^{1/2}+\nu\kappa_{2,G}\kappa_{\nabla_{\bx}^2c}(\Omega_1+\kappa_{\nabla_{\bx}\mL})$. Since $|\xi_g^k|\geq |\xi_g^{k-1}| +1$, we know $|\xi_g^k| \geq k$. Therefore,
\begin{equation*}
\mE_{\xi^k_g}\sbr{\|\barg_k - \nabla f_k\|}\leq (\mE_{\xi^k_g}\sbr{\|\barg_k - \nabla f_k\|^2})^{1/2} \leq \frac{\Omega_1}{\sqrt{k}}, 
\end{equation*}
and the above result holds similarly for $\mE_{\xi^k_g}[\|\barH_k - \nabla^2f_k\|]$. Combining the above two displays, we obtain
\begin{equation*}
\|c_k\| \leq \|\nabla\mL_{\barmu_{\barK}, \nu}^k\| + \frac{\Upsilon_5\max\{\Omega_1, \Omega_2\}}{\sqrt{k}}.
\end{equation*}
Moreover, by the definition of $\nabla\mL_{\barmu_{\barK}, \nu}^k$ in \eqref{equ:derivative:AL} and the bound on $\barmu_{\barK}\leq \tmu$ in Lemma~\ref{lem:9}, we have
\begin{equation*}
\|\nabla_{\bx}\mL_k\| \leq  \|\nabla_{\bx}\mL_{\barmu_{\barK}, \nu}^k\| + \nu\kappa_M\|G_k\nabla_{\bx}\mL_k\| + \tmu \sqrt{\kappa_{2,G}}\|c_k\|
\end{equation*}
and
\begin{equation*}
\|G_k\nabla_{\bx}\mL_k\| \leq \frac{1}{\nu\kappa_{1,G}}\|\nu G_kG_k^TG_k\nabla_{\bx}\mL_k\|\stackrel{\eqref{equ:derivative:AL}}{\leq} \frac{1}{\nu\kappa_{1,G}}(\|\nabla_{\blambda}\mL_{\barmu_{\barK}, \nu}^k\| + \|c_k\|).
\end{equation*}
Combining the above three displays, 
\begin{equation*}
\|\nabla\mL_k\|\leq \Upsilon_6\|\nabla\mL_{\barmu_{\barK}, \nu}^k\| + \rbr{1 + \frac{\kappa_{M}}{\kappa_{1,G}} + \tmu\sqrt{\kappa_{2,G}}}\frac{\Upsilon_5\max\{\Omega_1, \Omega_2\}}{\sqrt{k}}
\end{equation*}
where $\Upsilon_6 = 2 + 2\kappa_{M}/\kappa_{1,G} + \tmu\sqrt{\kappa_{1,G}}$. Since $\baralpha_k\leq \alpha_{max}$, we have 
\begin{equation*}
\lim\limits_{k\rightarrow \infty} \baralpha_k \|\nabla\mL_k\|^2\leq \Upsilon_6^2\lim\limits_{k\rightarrow \infty} \baralpha_k \|\nabla\mL_{\barmu_{\barK}, \nu}^k\|^2.
\end{equation*}
It hence suffices to show that for any iteration sequence $\lim\limits_{k\rightarrow \infty} \baralpha_k \|\nabla\mL_{\barmu_{\barK}, \nu}^k\|^2 = 0$. By Theorem \ref{thm:5}, we sum the error recursion for $k \geq \barK+1$, compute the conditional expectation given the first $\barK+1$ iterates, and obtain
\begin{align*}
\sum_{k = \barK+1}^{\infty}\mE[\baralpha_k\|\nabla\mL_{\barmu_{\barK}, \nu}^k\|^2 \mid \mF_{\barK}] & \leq  \Upsilon_7\sum_{k=\barK+1}^{\infty}\mE[\Phi_{\barmu_{\barK}}^k \mid \mF_{\barK}] -\mE[\Phi_{\barmu_{\barK}}^{k+1} \mid \mF_{\barK}]\\
& \leq  \Upsilon_7(\Phi_{\barmu_{\barK}}^{\barK+1} - \omega\min_{\mX\times\Lambda} \mL_{\barmu_{\barK}, \nu}(\bx, \blambda))<\infty,
\end{align*}
where $\Upsilon_7 = 4\rho/\cbr{(1-p_{grad})(1-p_f)(1-\omega)(\rho-1)}$. Thus, we have
\begin{equation*}
\lim\limits_{k\rightarrow \infty}\mE[\baralpha_k \|\nabla\mL_{\barmu_{\barK}, \nu}^k\|^2 \mid\mF_{\barK}] = 0.
\end{equation*}
Since $\baralpha_k\leq \alpha_{max}$ and $\{(\bx_k, \blambda_k)\}_k$ are in a compact set $\mX\times \Lambda$, we have $\baralpha_k \|\nabla\mL_{\barmu_{\barK}, \nu}^k\|^2 \leq \Upsilon_8$ for some constant $\Upsilon_8$ uniformly. By dominated convergence theorem \cite[Theorem 1.5.8]{Durrett2019Probability}, we exchange the order of expectation and limit and get $\mE[\lim\limits_{k\rightarrow\infty}\baralpha_k \|\nabla\mL_{\barmu_{\barK}, \nu}^k\|^2 \mid \mF_{\barK}] = 0$. Since $\baralpha_k \|\nabla\mL_{\barmu_{\barK}, \nu}^k\|^2$ is non-negative, we have $\lim\limits_{k\rightarrow\infty}\baralpha_k \|\nabla\mL_{\barmu_{\barK}, \nu}^k\|^2 = 0$ almost surely, which completes the proof. \qed
\end{proof}

Our final result establishes the ``liminf'' convergence of $\nabla\mL_k$.  A key step is to apply the lower bound on the stochastic stepsize in Lemma \ref{lem:8}. Based on this lemma, when both the gradient and function are precisely estimated, a stepsize that is smaller than a fixed threshold will always induce a successful step. If we regard $\{\baralpha_k\}_k$ as a random walk, the supremum of $\baralpha_k$ will not vanish due to a positive upward drift probability.

\begin{theorem}[Global convergence of Algorithm \ref{alg:ASto:SQP}]\label{thm:7}
Consider Algorithm \ref{alg:ASto:SQP} under Assumption \ref{ass:A:1}. Suppose $\omega$ satisfies \eqref{equ:cond:omega:1} and $p_f, p_{grad}$ satisfy \eqref{cond:pp}. Then, almost surely, we have that
\begin{equation*}
\liminf_{k\rightarrow \infty}\|\nabla\mL_k\| = 0.
\end{equation*}
	
\end{theorem}

\begin{proof}

We define two random sequences
\begin{align*}
\phi_k & =\log\rbr{\baralpha_k},\\
\varphi_k & = \min\cbr{\log\rbr{\tau}, \b1_{\mA_{k-1}\cap\mB_{k-1}}\rbr{\log(\rho) + \varphi_{k-1}} + (1-\b1_{\mA_{k-1}\cap\mB_{k-1}})(\varphi_{k-1} - \log(\rho)) },
\end{align*}
where $\varphi_0 =\log\rbr{\baralpha_0}$. Let $\tau$ be a deterministic constant such that
\begin{align*}
\tau \leq \frac{(1-\beta)(\gamma_{RH}\wedge\nu)}{2\rbr{\kappa_{\mL_{\tmu, \nu}} \Upsilon_3^2 + \kappa_{grad}\Upsilon_1\Upsilon_3 + \kappa_f\gamma_{RH}}} \wedge \alpha_{max},
\end{align*}
and $\tau = \rho^i\alpha_{max}$ for some integer $i\leq 0$.  Therefore, by the stepsize updating rule in Algorithm \ref{alg:ASto:SQP}, we have that for any $k$, $\baralpha_k = \rho^{j_k}\tau$ for some integer $j_k$. We lower bound $\phi_k$ by $\varphi_k$ using induction.  First, we observe that $\phi_k$ and $\varphi_k$ are both $\mF_{k-1}$-measurable. Note that $\phi_0 = \varphi_0$. Suppose $\phi_k \geq \varphi_k$ and we consider the following three cases:

\noindent\textbf{(a).} If $\phi_k > \log(\tau)$, then $\phi_k \geq \log(\tau) + \log(\rho)$. Thus, $\phi_{k+1} \geq \phi_k - \log(\rho) \geq \log(\tau) \geq\varphi_{k+1}$.

\noindent\textbf{(b).} If $\phi_k \leq \log(\tau)$ and $\b1_{\mA_k\cap\mB_k} = 1$, then, by Lemma \ref{lem:8},
\begin{equation*}
\phi_{k+1} = \min\cbr{\log(\alpha_{max}), \phi_k + \log(\rho)} \geq \min\{\log(\tau), \varphi_k + \log(\rho)\} = \varphi_{k+1}.
\end{equation*}
Here the inequality is due to the definition of $\tau$ and the induction hypothesis.

\noindent\textbf{(c).} If $\phi_k\leq \log(\tau)$ and $\b1_{\mA_k\cap \mB_k} = 0$, then
\begin{equation*}
\phi_{k+1} \geq \phi_k - \log(\rho) \geq \varphi_k - \log(\rho) \geq \varphi_{k+1}.
\end{equation*}
Thus, $\phi_k\geq \varphi_k$, $\forall k$. Note that $\varphi_k$ is a random walk with a maximum and a drift upward. Thus, by analogy to Example 6.1.2 in \cite{Gallager2013Stochastic} and noting from \eqref{cond:pp} that $p_{grad}+p_f<1/2$, we know from Sections 6.2 and 6.3 (pp 290) in \cite{Gallager2013Stochastic} that $\log(\tau)$ is a positive recurrent state of the process $\varphi_k$. That is, $\varphi_k$ visits $\log(\tau)$ infinite times. Thus, ``limsup'' of $\baralpha_k$ is lower bounded by noting that 
\begin{equation*}
P\big(\limsup_{k\rightarrow \infty} \phi_k \geq \log(\tau)\big) \geq P\big(\limsup_{k\rightarrow \infty} \varphi_k \geq \log(\tau)\big) = 1.
\end{equation*}
Using Theorem \ref{thm:6}, we complete the proof. \qed
\end{proof}

We complete the global analysis of Algorithm \ref{alg:ASto:SQP}. The result in Theorem~\ref{thm:7} is consistent with Theorem \ref{thm:4}(b). However, the stepsize behavior of the two~algorithms is largely different. Algorithm \ref{alg:NSto:SQP} employs a deterministic stepsize~sequence. The condition $\sum_{k=0}^{\infty}\alpha_k^2 < \infty$ implies that $\alpha_k$ decays to zero with a certain rate. Differently, as proved in Theorem \ref{thm:7}, the stepsize $\baralpha_k$ derived~from stochastic line search has a subsequence that is lower bounded away from zero. Therefore, stochastic line search often suggests a stepsize that is larger than the one~suggested by Theorem \ref{thm:4}(b), especially for large $k$, which leads to a better performance. On the other hand, we note that employing stochastic line search requires a more stringent setup than the fully stochastic setup of Algorithm \ref{alg:NSto:SQP}. The former generates batch samples to diminish the model estimation error, while the~latter only generates two samples in~each~iteration~(\cite{Berahas2021Sequential}~only~generates~one~sample). In addition, the analysis of Algorithm \ref{alg:NSto:SQP} requires only the estimation errors to have bounded variance, while the analysis of Algorithm \ref{alg:ASto:SQP} requires the errors to be bounded. Such a boundedness condition ensures that the penalty parameter is finally stabilized.

Both results in Theorems \ref{thm:4} and \ref{thm:7} are weaker than deterministic SQP (cf. \citep[Theorem 4.1]{Lucidi1990Recursive}), where one has $\lim\limits_{k\rightarrow \infty}\|\nabla\mL_k\| = 0$. The intrinsic difference between Algorithm \ref{alg:ASto:SQP} and deterministic SQP still lies in the stepsize behavior. For deterministic SQP, a standard result is that $\alpha_k$ is lower bounded away~from~zero for \textbf{all} sufficiently large $k$, which leads to convergence of the KKT residual~sequence $\|\nabla\mL_k\|$. In contrast, due to the existence of imprecise model estimation, $\baralpha_k$ only provably has a subsequence that is lower~bounded~away from zero. Thus, we can only conclude that there exists a convergent subsequence for~$\|\nabla\mL_k\|$.

We note that \cite{Paquette2020Stochastic} established the expected iteration complexity for a~stochastic line search. Performing a similar complexity analysis for constrained problems is significantly more difficult. The main issue stems from the adaptivity in selecting the penalty parameter. In general, there is no guarantee on how long it takes for a randomly selected penalty parameter sequence to stabilize. Although we believe that a similar analysis to \cite{Paquette2020Stochastic} is applicable after the penalty parameter stabilizes, a more advanced technique and a finer analysis are required to study the early period of the random walk induced by the penalty parameter, where the parameter varies in each iteration. In other words, a deeper understanding on the random iteration threshold $\barK$ (cf. Lemma \ref{lem:9}) is desired. We leave the investigation of this problem to future work.

\section{Experiments}\label{sec:5}

We implement four StoSQP algorithms for solving constrained nonlinear~optimization problems collected in CUTEst test set \citep{Gould2014CUTEst}. We use Julia implementation of CUTEst \citep{Siqueira2020CUTEst.jl}. The four algorithms that we implement are: $\ell_1$ penalized SQP in \cite{Berahas2021Sequential}, NonAdapSQP in Algorithm \ref{alg:NSto:SQP}, AdapSQP in Algorithm \ref{alg:ASto:SQP}, and $\ell_1$ penalized AdapSQP. The $\ell_1$ penalized AdapSQP has a similar iteration scheme to AdapSQP, except that it employs the $\ell_1$ penalized merit function
\begin{equation*}
\mL_{\mu}(\bx) = f(\bx) + \mu\|c(\bx)\|_1.
\end{equation*}
In particular, the condition \eqref{event:A_k} in the first step is replaced by
\begin{equation*}
\mA_k = \cbr{\|\barg_k - \nabla f_k\| \leq\kappa_{grad}\cdot\baralpha_k \|\bnabla\mL_k\|},
\end{equation*}
where the right hand side uses the KKT residual as the counterpart of $\|\barg_k\|$ used in unconstrained problems. For the second step, the search direction is solved  by the first system in \eqref{equ:ran:Newton}. For a prespecified $\rho\in(1,2)$, the penalty~parameter $\barmu_k$ is updated by $\barmu_k = \barg_k^T\barDelta\bx_k/\{(\rho-1)\|c_k\|_1\}$, suggested by \cite[(18.33)]{Nocedal2006Numerical}, to ensure the directional derivative of $\barL_{\barmu_k}^k$ along the direction $\barDelta\bx_k$ yields a~sufficient decrease. Furthermore, the condition \eqref{event:B_k} in the third step is replaced by
\begin{equation*}
\mB_k = \cbr{|\barL_{\barmu_k}^k - \mL_{\barmu_k}^k| \vee |\barL_{\barmu_k}^{s_k} - \mL_{\barmu_k}^{s_k}|\leq -\kappa_f\alpha_k^2\rbr{\barg_k^T\barDelta\bx_k - \barmu_k\|c_k\|_1}}
\end{equation*}
where the right hand side is the directional derivative of $\barL_{\barmu_k}^k$ along $\barDelta\bx_k$. The fourth step for line search is the same as AdapSQP. We should mention that there may exist advanced modifications for SQP using the $\ell_1$ merit function. Here, we only consider the stochastic counterpart of the most basic deterministic SQP scheme. An advanced design~of the $\ell_1$ penalized AdapSQP and a rigorous analysis are left for future work.

Of more than 1000 problems in CUTEst collection, we select problems that have a non-constant objective with only equality constraints, satisfy $d<1000$, and do not report singularity on $G_kG_k^T$ during the iteration process. This results in a total of 47 problems. The implementation details of each algorithm are as follows. Our code is available at \url{https://github.com/senna1128/Constrained-Stochastic-Optimization}.
\begin{enumerate}[label=(\alph*),topsep=0pt]
\setlength\itemsep{0.0em}
\item \textbf{$\ell_1$ SQP} in \cite{Berahas2021Sequential}. We implement \cite[Algorithm 3.1]{Berahas2021Sequential} following the setup in~that paper. In particular, using their notation, we let $\bartau_{-1} = 1$, $\epsilon = 10^{-6}$, $\sigma = 0.5$, $\bar{\xi}_{-1} = 1$, and $\theta = 10$. The Lipschitz constant is estimated around the initialization. We try multiple choices for the stepsize related sequence $\{\beta_k\}_k$. For constant case, we let $\beta_k = \{0.01, 0.1, 0.5, 1\}$ and, for decaying~case, we let $\beta_k = \{1/k^{0.6}, 1/k^{0.9}\}$. Note that \cite{Berahas2021Sequential} only tried $\beta_k = 1$.
	
\item \textbf{NonAdapSQP} in Algorithm \ref{alg:NSto:SQP}. Following the same setup as above, we try $\alpha _k = \{0.01, 0.1, 0.5, 1\}$ and $\alpha_k =\{1/k^{0.6}, 1/k^{0.9}\}$.
	
\item \textbf{AdapSQP} in Algorithm \ref{alg:ASto:SQP}. We let $\nu = 0.001$, $\baralpha_0 = \alpha_{max} = 1.5$, $\barmu_0 =\barepsilon_0 =\kappa_{grad}= 1$, $\rho = 1.2$, $\beta = 0.3$, $p_{grad} = p_f = 0.1$, $\kappa_f =  \beta/(4\alpha_{max}) =  0.05$, and $C_{grad} = C_f  =\{1,5,10,50\}$. Recall that $C_{grad}$ and $C_f$ are used for selecting $|\xi_g^k|$ and $|\xi_f^k|$ using \eqref{cond:11} and \eqref{cond:2}, respectively. We try a wide range of $C_{grad}$ and $C_f$ to investigate how do these two tuning parameters affect the scheme performance. In the above setting, we let $\nu$ be small to make $\mL_{\mu, \nu}$ similar to standard augmented Lagrangian. Another motivation of using small $\nu$ is that, without $G(\bx)$ in the second penalty of \eqref{equ:augmented:L}, $\mL_{\mu, \nu}$ is an exact augmented Lagrangian only if $\nu$ is sufficiently small. We let $\alpha_{max}=1.5> 1$ since, as also shown later, that the selected stepsize may be greater than $1$ in a stochastic scheme. $\barmu_0, \barepsilon_0$ are initialized at 1 and adaptively updated with iterations. $\kappa_{grad}, \kappa_f, p_{grad}, p_f$ all affect the generated sample sizes~$|\xi_g^k|$ and $|\xi_f^k|$ via \eqref{cond:11} and \eqref{cond:2}. They play a similar role to parameters $C_{grad}, C_f$. Thus, for simplicity, we fix them and only try multiple $C_{grad}, C_f$ in experiments. As suggested by Theorem \ref{thm:5}, $p_{grad}, p_f$ should be small, so we set them to be $0.1$; $\kappa_f = \beta/(4\alpha_{max})$ is the largest theoretical value that is allowed. We let $\rho=1.2>1$ be a moderate increasing factor. A large $\rho$ may terminate the While loop in Line 5 of Algorithm \ref{alg:ASto:SQP} (as well as Algorithm \ref{alg:sample}) faster, while also resulting in a rough estimate of penalty parameter. We let $\beta = 0.3$ be a (nearly) middle value of interval $(0, 0.5)$, which is the range to have a fast local rate in deterministic case \citep{Lucidi1990Recursive} (although local behavior is not a focus of the paper and cannot be investigated by our experiments).

\item \textbf{$\ell_1$ AdapSQP}. We adopt the same setup as above, except that the parameter $\nu$ is not required.
\end{enumerate}

\vskip2pt
For all algorithms, the initialization of primal-dual variables is given by CUTEst package. Moreover, the package provides deterministic evaluations of the function, gradient, and Hessian at each iterate. Based on them, we generate our estimators. In particular, the estimator of $f_k$ is drawn from $\N(f_k, \sigma^2)$; the estimator of $\nabla f_k$ is drawn from $\N(\nabla f_k, \sigma^2(I + \b1\b1^T))$, where $\b1$ denotes the $d$-dimensional all one vector; and the $(i, j)$ and $(j, i)$ entries of the estimator of $\nabla^2 f_k$ correspond to the same draw of $\N((\nabla^2 f_k)_{i,j}, \sigma^2)$. We vary $\sigma^2$ from $\{10^{-8}, 10^{-4}, 10^{-2}, 10^{-1},1\}$. Throughout the simulation, we let $B_k = I$ and set the maximum iteration budget to be $10^5$. For each algorithm, under each setup (a combination of a noise level and stepsize/constants $C_{grad}, C_{f}$), we perform 5 independent runs. The stopping criterion is set as
\begin{equation}\label{equ:criteria}
\|\baralpha_k\cdot(\barDelta\bx_k; \barDelta\blambda_k)\| \leq 10^{-6} \quad \text{OR}\quad  \|\nabla \mL_k\| \leq 10^{-4} \quad \text{OR}\quad  k \geq 10^5.
\end{equation}
If the former two cases occur, we say that the algorithm converges, while if the last case occurs, we say that the algorithm does not converge within the prespecified iteration budget. For $\ell_1$ SQP, $\baralpha_k$ is determined by $\beta_k$ (see \cite{Berahas2021Sequential}) and we also drop $\barDelta\blambda_k$ in the first term. We comment that the first criterion in \eqref{equ:criteria} measures the distance between two successive iterates, which depends on the stochastic search direction and stepsize, and thus is computable without knowing true quantities in contrast to the second criterion. We add it to stop algorithms whenever we (tend to) have limited progress for later iterations. It is triggered if $\|(\barDelta\bx_k; \barDelta\blambda_k)\|$ is small or/and $\baralpha_k$ is small. For the first case, by Newton's system we know $\|\bnabla\mL_k\|$ is small and we should naturally stop. For the second case, we note that the stepsize in NonAdapSQP is either constant or a decaying sequence; thus, it is also natural to stop as the stepsize, if it is small, can only get smaller and smaller. For $\ell_1$ SQP in \cite{Berahas2021Sequential}, although $\baralpha_k$ is not monotonically decreasing, it is upper and lower controlled by the sequence $\beta_k$, which is either constant or decaying (cf. Lemma 3.6). Thus, $\baralpha_k$ \textit{roughly} has a similar trend as $\beta_k$, and cannot increase significantly if it is already very small. So we stop~as well. For AdapSQP and $\ell_1$ AdapSQP, by Lemma \ref{lem:8} we know that $\baralpha_k$ cannot be small if the merit function and its gradients are precisely estimated, which has a high probability in each iteration. Further, whenever $\baralpha_k$ decreases, a more precise model will be generated for next round (Line 18 of Algorithm~\ref{alg:ASto:SQP}). Thus, it is not very likely that $\baralpha_k$ decreases to a very small value due to a series of bad estimates, although there is indeed a positive probability for such a bad scenario. Thus, we make the first criterion more restrictive than the second in~\eqref{equ:criteria} to prevent extreme scenarios.

\vskip 4pt
\noindent\textbf{Convergence behavior.} We now evaluate convergence behavior of each~algorithm. For each run of each setup of each algorithm on each problem, we~derive the KKT residual of the last iterate if the algorithm converges, and then average residuals over all convergent runs (across $5$ runs) as the evaluation~result. We find in the implementation that, for two line search algorithms~AdapSQP~and~$\ell_1$ AdapSQP, if they converge for a particular run, they always converge~for~all $5$ runs. However, for NonAdapSQP and $\ell_1$ SQP, occasionally it occurs that some runs (across $5$ runs) do not converge while some runs do. This is likely because the performance of these two algorithms on some problems tends to be more random, due to limited samples that are generated. In this scenario, instead of viewing the problem as a divergent problem, we prefer to average over convergent runs to improve the performance of NonAdapSQP and $\ell_1$ SQP.

For NonAdapSQP and $\ell_1$ SQP, we have six cases for both: four constant stepsizes and two decay stepsizes. For AdapSQP and $\ell_1$ AdapSQP, we have~four cases for both, corresponding to $C_{grad} = C_f = \{1, 5, 10, 50\}$. For each case, we have five different noise levels of $\sigma^2$. We draw residual boxplots for all cases in Figures \ref{fig:1} and \ref{fig:2}. Each box reflects the residual range of convergent problems. Each plot in Figure \ref{fig:1} corresponds to a stepsize setup of NonAdapSQP and $\ell_1$ SQP. The results of AdapSQP and $\ell_1$ AdapSQP in Figure \ref{fig:1} do not change (as their stepsize is selected by stochastic line search), and correspond to the case where $C_{grad} = C_f = 1$. If the box of NonAdapSQP is missing in the~plot, which often happens for constant stepsize with large variance, it~suggests~that NonAdapSQP does not converge for all problems (within the iteration budget). Each plot in Figure \ref{fig:2} corresponds to a setup of constants~$C_{grad}, C_f$~of AdapSQP and $\ell_1$~AdapSQP.

\begin{figure}[!htp]
	\centering     
	\subfigure[$\alpha_k, \beta_k = 0.01$]{\label{KConst1}\includegraphics[width=58mm]{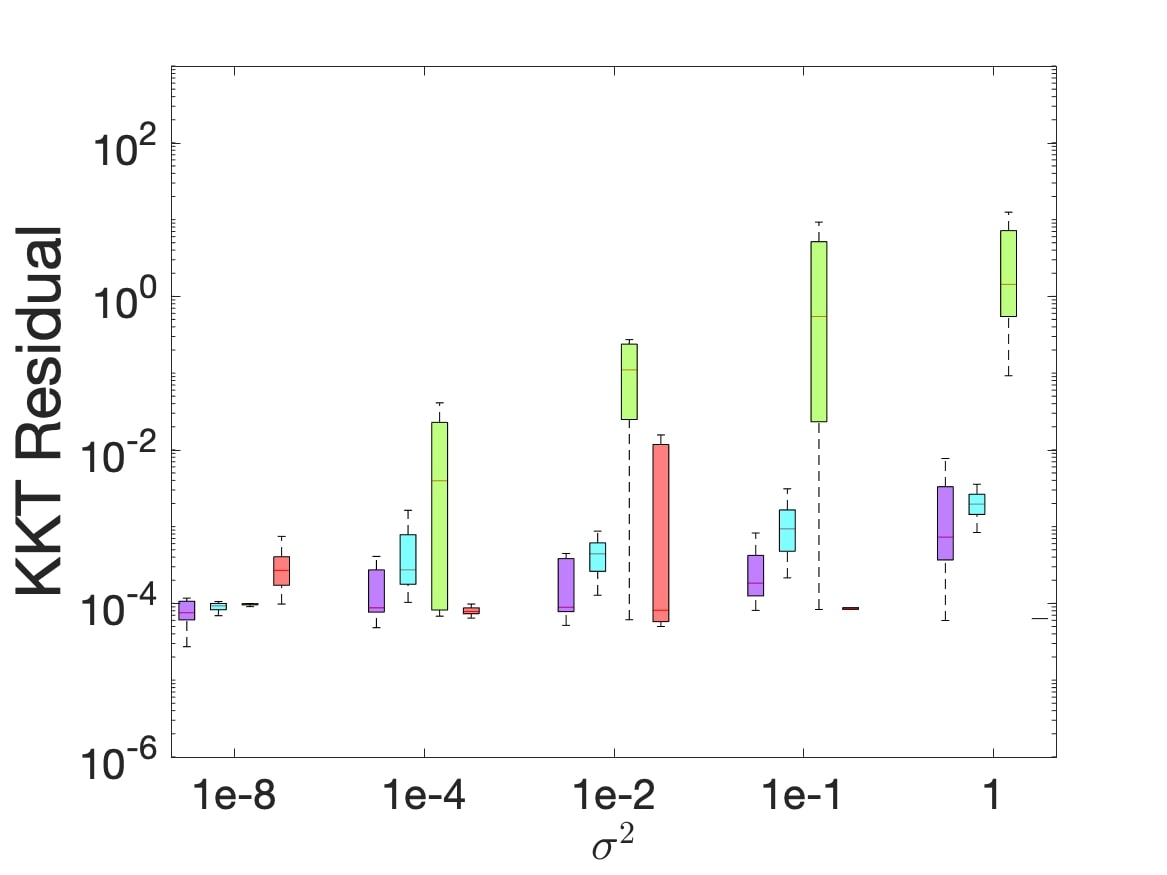}}
	\subfigure[$\alpha_k, \beta_k = 0.1$]{\label{KConst2}\includegraphics[width=58mm]{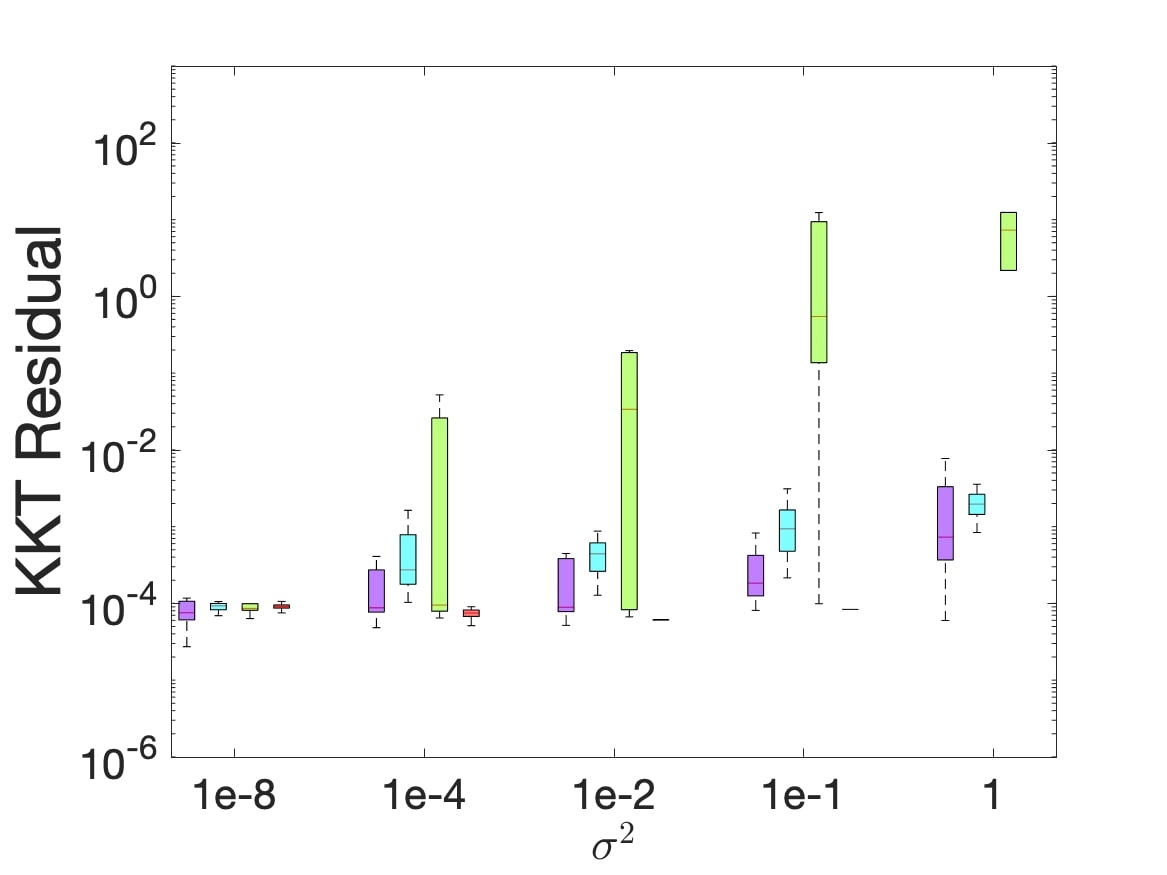}}
	\subfigure[$\alpha_k, \beta_k = 0.5$]{\label{KConst3}\includegraphics[width=58mm]{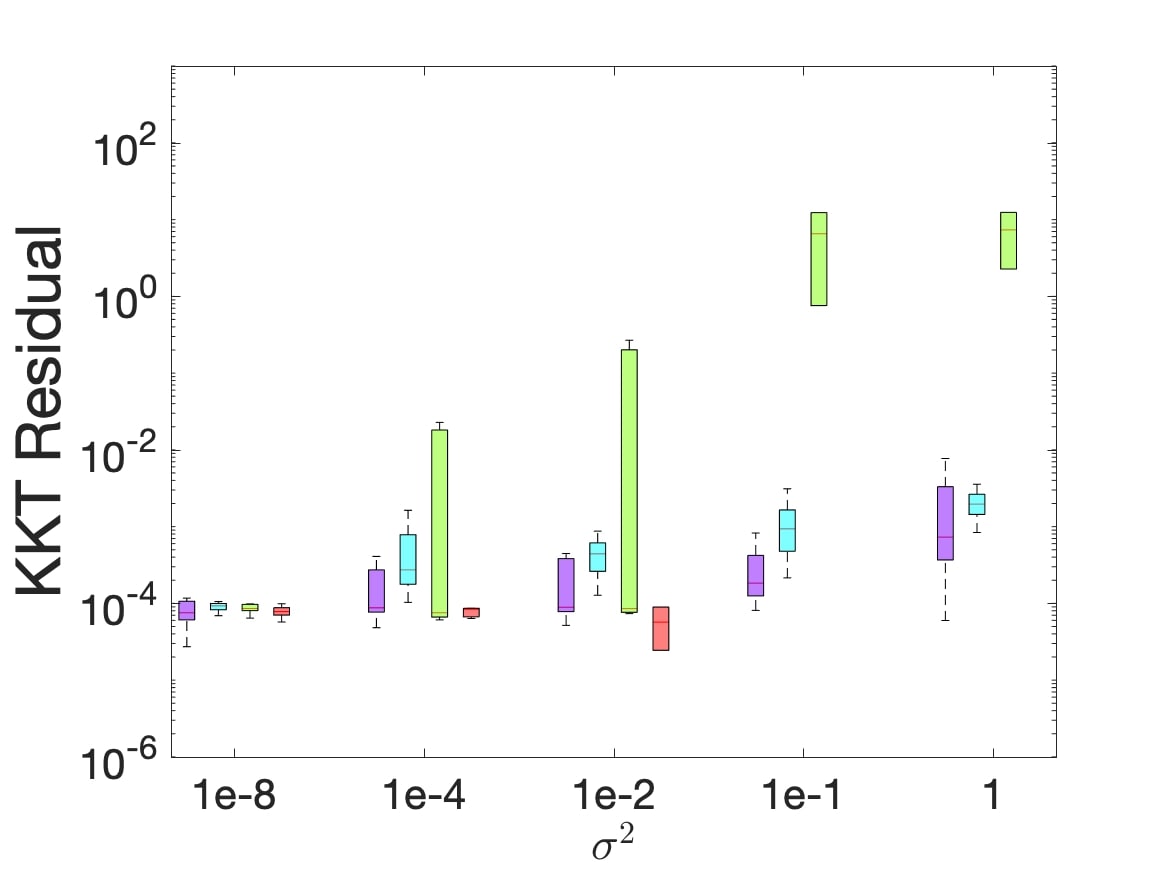}}
	\subfigure[$\alpha_k, \beta_k = 1$]{\label{KConst4}\includegraphics[width=58mm]{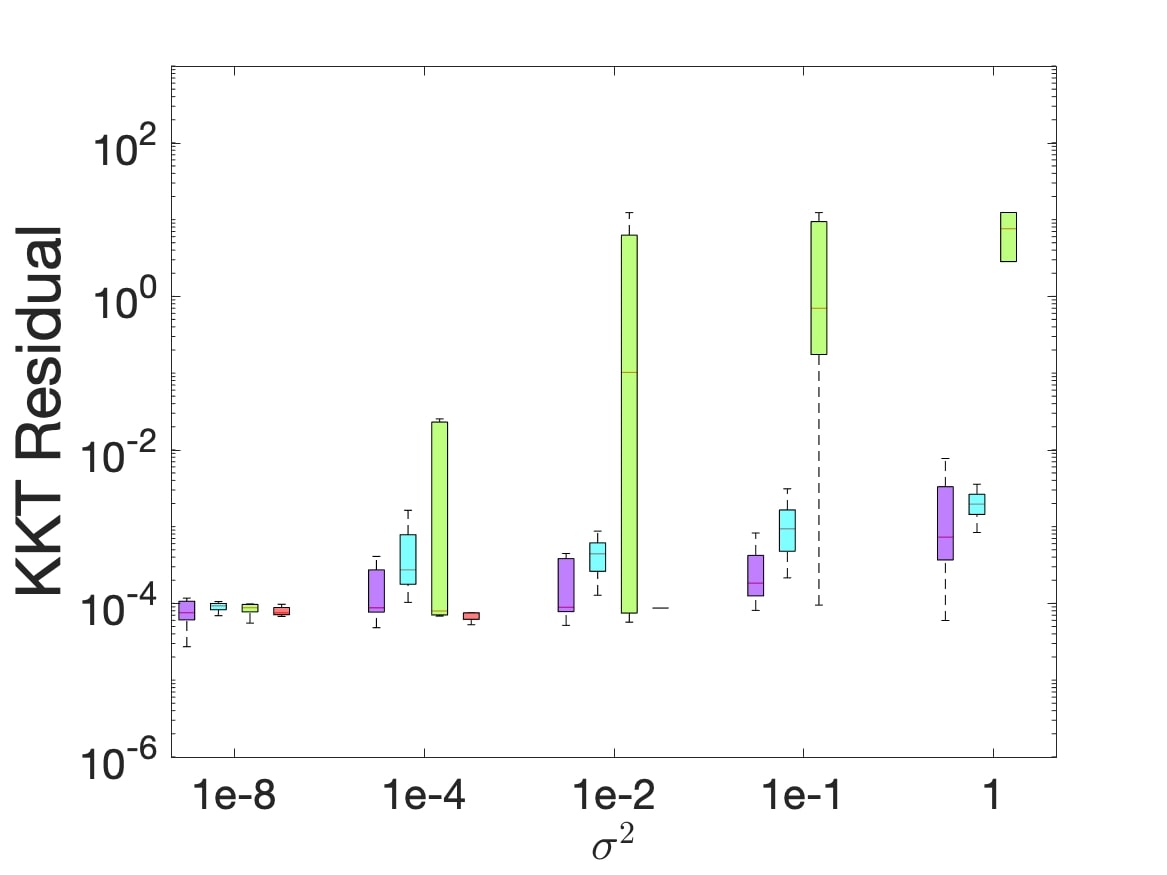}}
	\subfigure[$\alpha_k, \beta_k = k^{-0.6}$]{\label{KDecay1}\includegraphics[width=58mm]{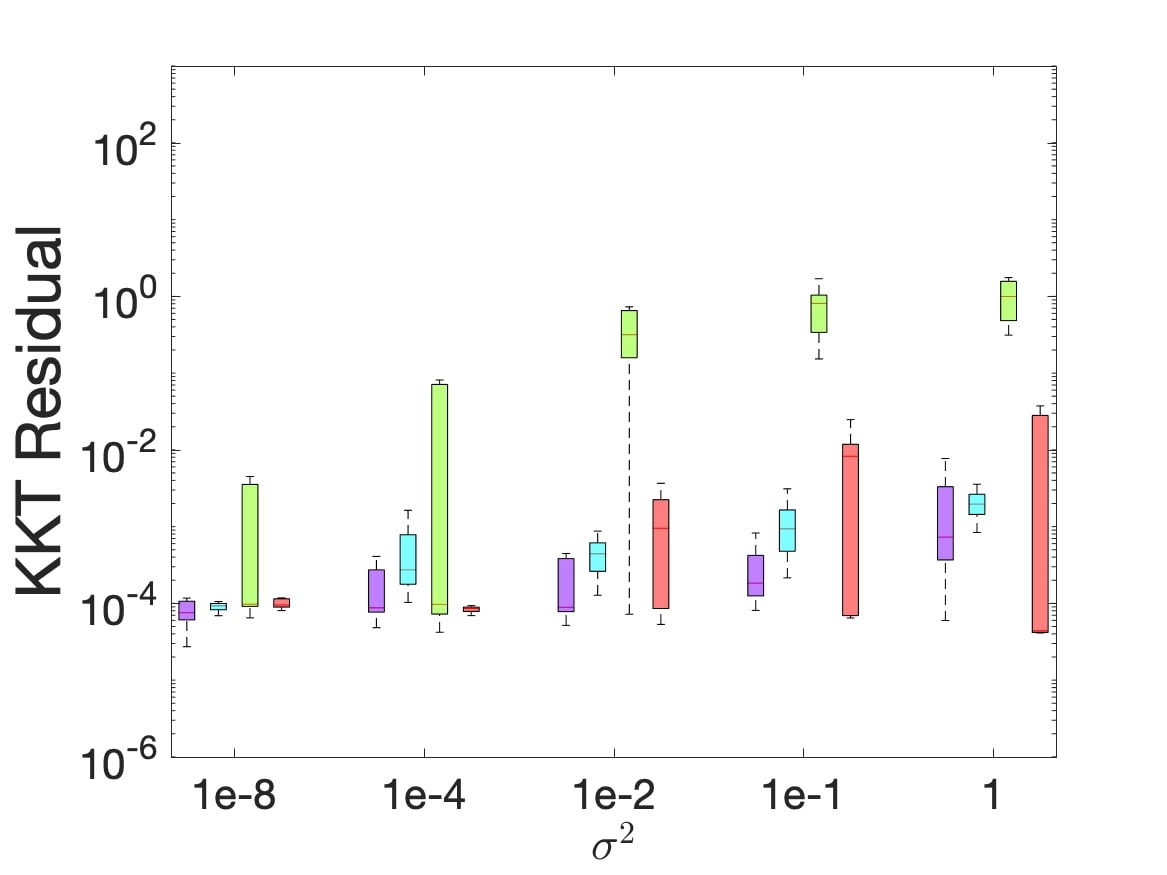}}
	\subfigure[$\alpha_k, \beta_k = k^{-0.9}$]{\label{KDecay2}\includegraphics[width=58mm]{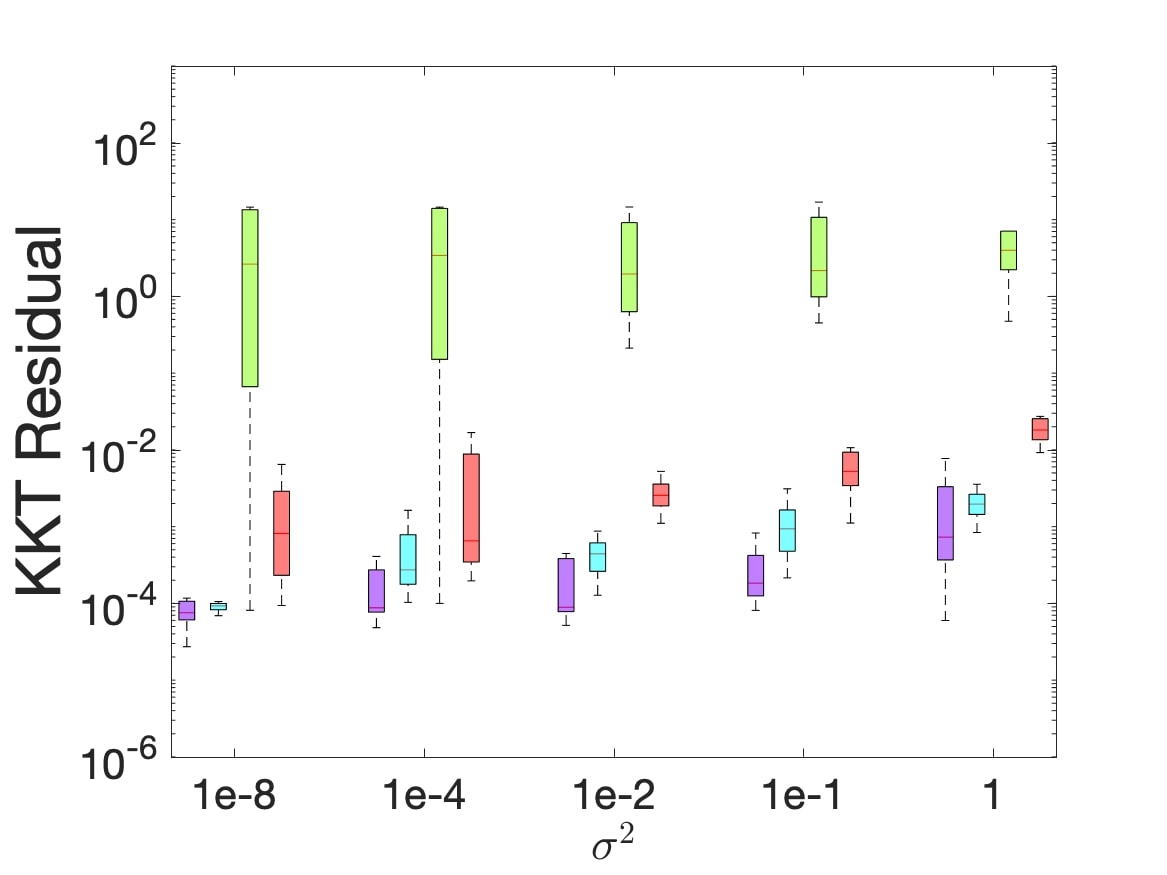}}
	\includegraphics[width=0.5\textwidth]{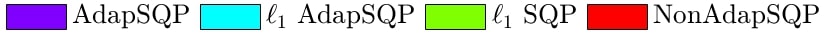}
	\caption{KKT residual boxplot. Each panel corresponds to a stepsize setup, $\alpha_k$ for NonAdapSQP and $\beta_k$ for $\ell_1$ SQP. Each panel has $5$ groups, corresponding to $5$ different $\sigma^2$. The results of AdapSQP and $\ell_1$ AdapSQP on all panels are the same, which correspond to $C_{grad} = C_f = 1$.
	}\label{fig:1}
\end{figure}

\begin{figure}[!htp]
	\centering     
	\subfigure[$C_{grad} =C_f =  1$]{\label{KKConst1}\includegraphics[width=55mm]{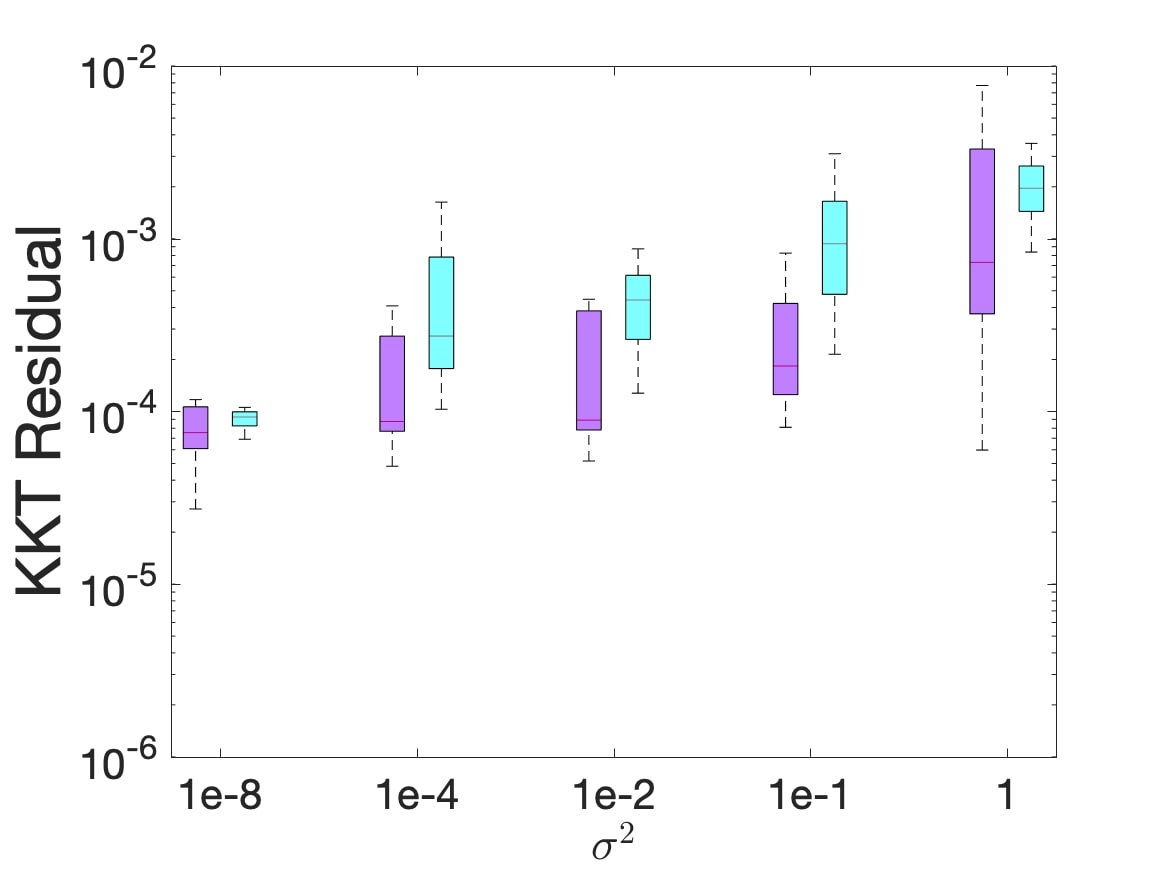}}
	\subfigure[$C_{grad} =C_f =  5$]{\label{KKConst2}\includegraphics[width=55mm]{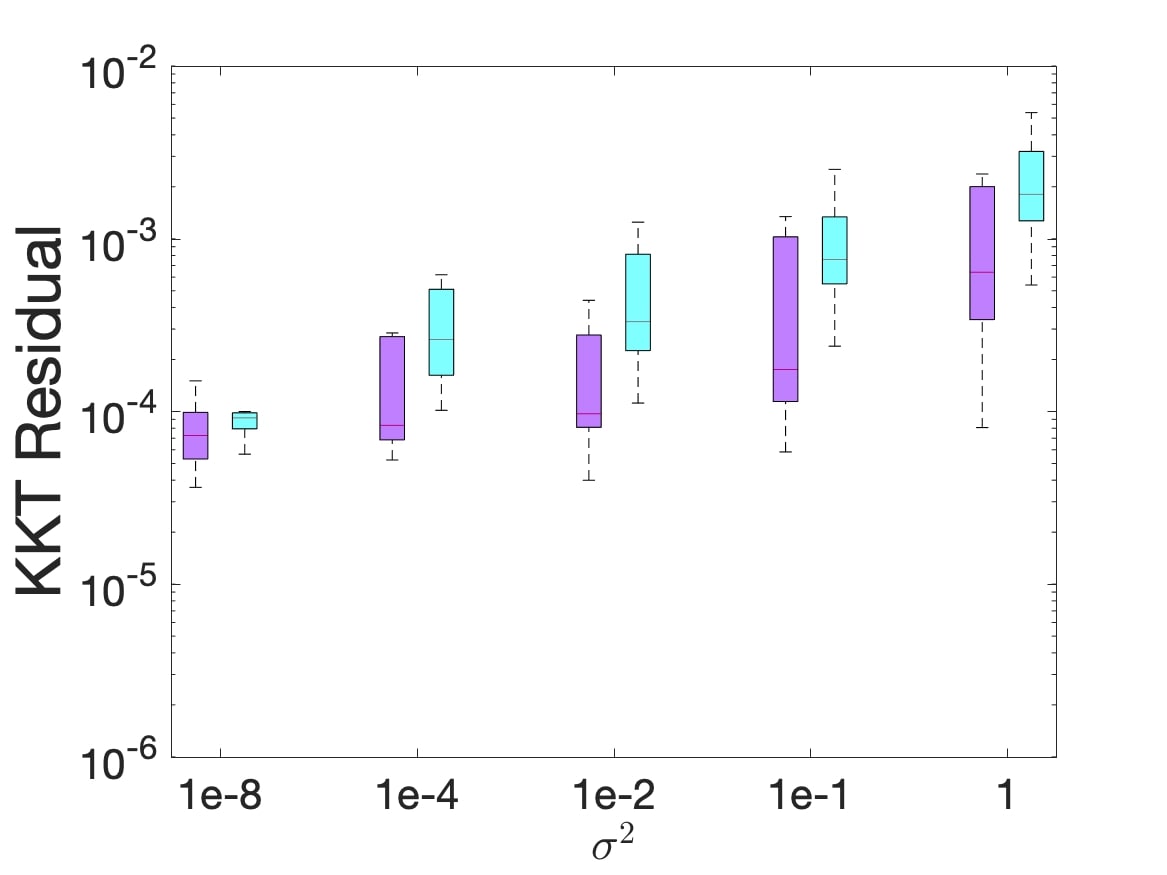}}
	\subfigure[$C_{grad} =C_f =  10$]{\label{KKConst3}\includegraphics[width=55mm]{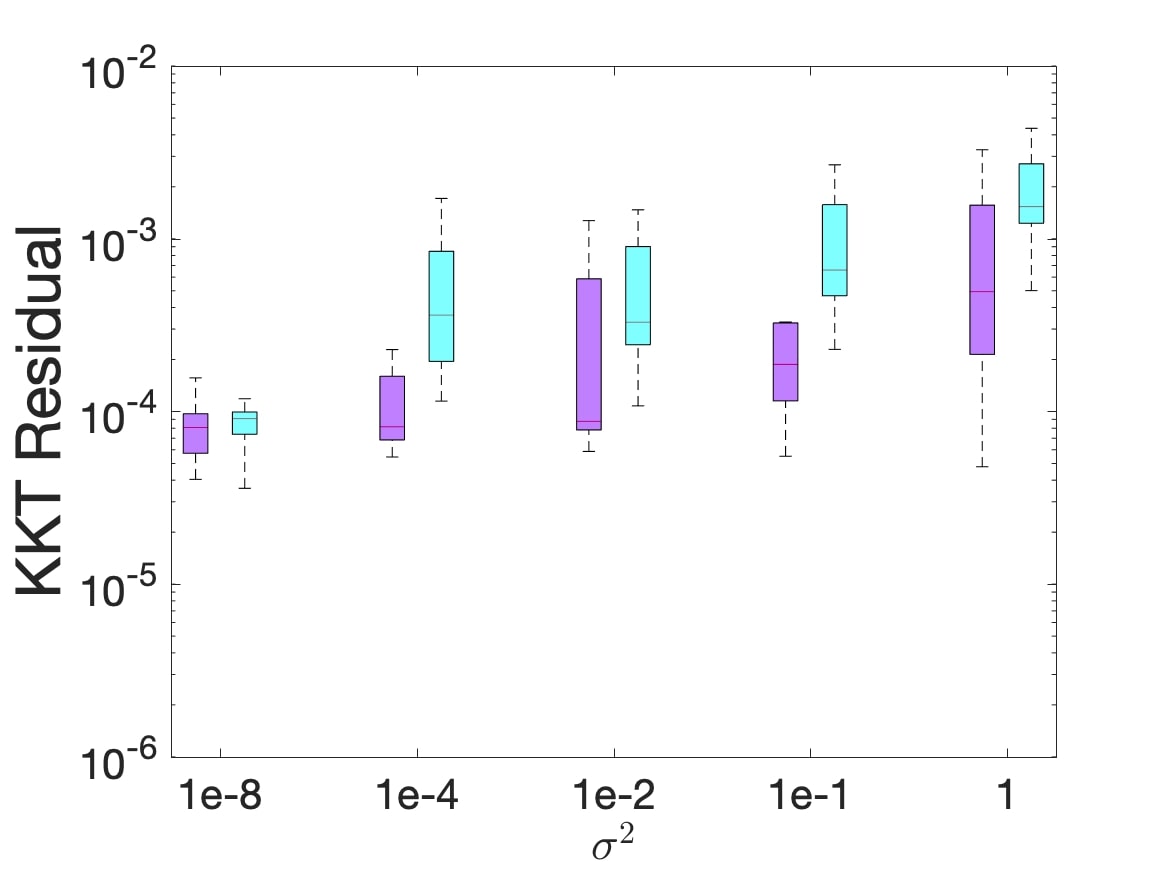}}
	\subfigure[$C_{grad} =C_f =  50$]{\label{KKConst4}\includegraphics[width=55mm]{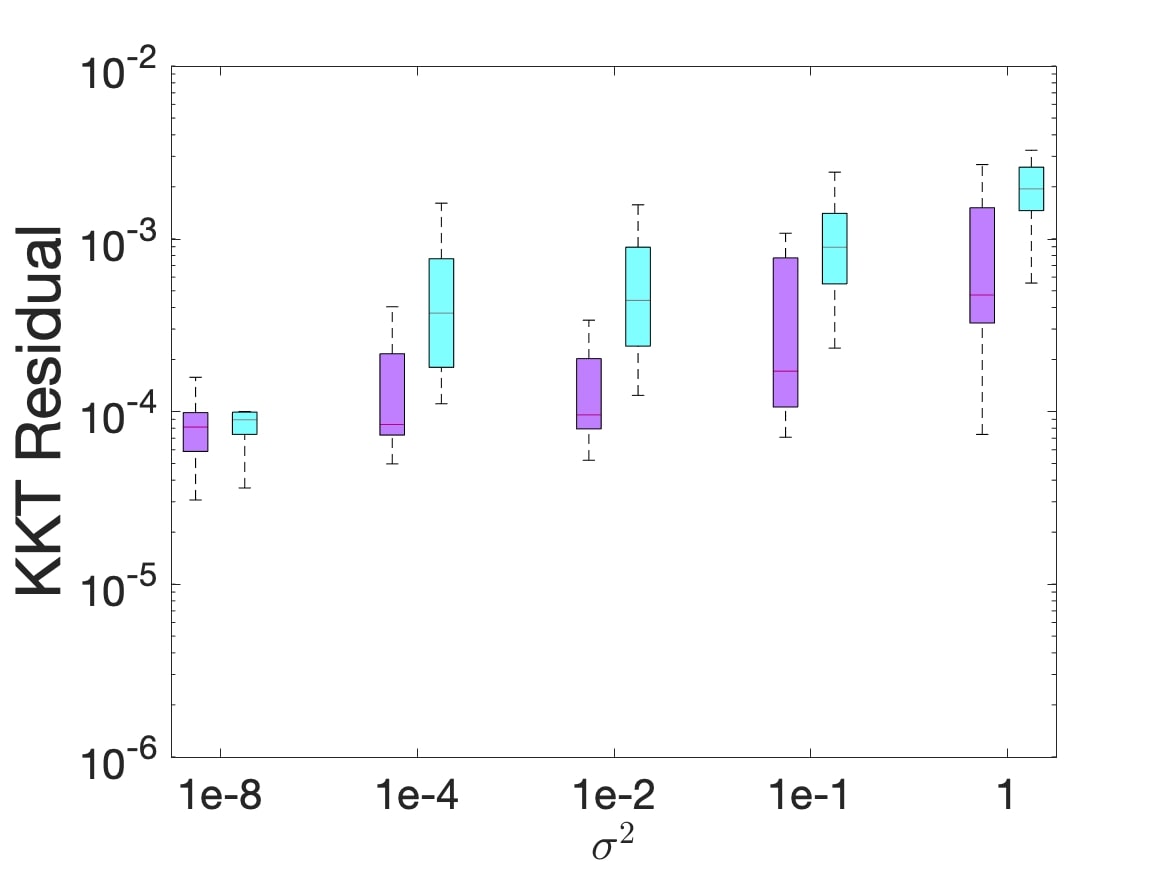}}
	\includegraphics[width=0.25\textwidth]{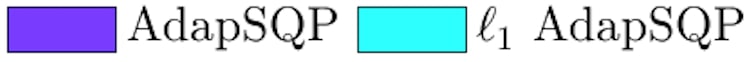}
	\caption{KKT residual boxplot. Each panel corresponds to a setup of constants $C_{grad}, C_f$.}\label{fig:2}
\end{figure}

From Figure \ref{fig:1}, we see that line search schemes AdapSQP and $\ell_1$~AdapSQP perform as well as NonAdapSQP and $\ell_1$ SQP when the noise level is low~(e.g. $\sigma^2 = 10^{-8}$), while perform significantly better when the noise level is high.~This phenomenon can be explained by intrinsic differences of the two~types of~stochastic algorithms. AdapSQP and $\ell_1$ AdapSQP adaptively increase the batch sizes as the iteration proceeds. The large batch size suppresses the effect of the noise variance. In particular, given $|\xi_g^k|$ samples, the estimated gradient has variance $\sigma^2/|\xi_g^k|$, which is notably small for large $|\xi_g^k|$. Thus, different levels of noise can be interpreted as different setups for constants $C_{grad}, C_f$. As we discussed in Section \ref{sec:4} and will empirically show later (Figures \ref{fig:2}, \ref{fig:3}, \ref{fig:4}), these constants have a marginal effect on the performance of line search schemes. Consistently, although their residuals get slightly larger as $\sigma^2$ increases, both AdapSQP and $\ell_1$ AdapSQP have a robust performance~for~different noise~levels.

In contrast, we see from Figure \ref{fig:1} that both NonAdapSQP and $\ell_1$ SQP are heavily affected by noise levels and prespecified sequences. $\ell_1$ SQP performs better than NonAdapSQP for constant sequences, while worse for decaying sequences. Both algorithms have higher residuals as $\sigma^2$ increases. Their sensitivity on $\sigma^2$ is reasonable because these two algorithms only generate one or two~samples in each iteration, which cannot reduce the estimation variance like line search schemes do. These two algorithms are also sensitive to prespecified sequences. $\ell_1$~SQP involves a novel stepsize selection scheme to enhance adaptivity; and we~indeed observe that it is more robust than NonAdapSQP for constant sequences. The latter may converge to a smaller residual compared to AdapSQP for some problems under a particular setup (e.g. $\alpha_k = 0.5, \sigma^2=0.01$), but it also frequently does~not converge within the budget for most, if not all, of problems if the stepsize is large and the noise level is high. The divergence can be explained by the violation of upper boundedness condition on the stepsize; and the upper bound is inversely proportional to $\sigma^2$ (cf. Theorem \ref{thm:4} and \eqref{equ:U_0}). Thus, a larger $\sigma^2$ implies a more stringent condition on the stepsize. A more well-designed NonAdapSQP, for example, similar to $\ell_1$ SQP, may resolve the divergence issue, but we do not investigate this here. However, even for this basic design without any adaptivity, we see from Figures \ref{KDecay1}, \ref{KDecay2} that NonAdapSQP performs better than $\ell_1$ SQP for decaying sequences. This phenomenon can be explained by \cite[Lemma 3.6]{Berahas2021Sequential}. In particular, the selected stepsize of $\ell_1$ SQP is in an interval whose length is proportional to $\beta_k^2$. Thus, for decaying $\beta_k$, the adaptivity gained by the scheme in \cite{Berahas2021Sequential} is less effective, especially when $\beta_k$ has a fast decay rate. Overall, NonAdapSQP and $\ell_1$ SQP are both sensitive to the sequences $\alpha_k$, $\beta_k$. Compared to NonAdapSQP, $\ell_1$ SQP is indeed more robust to constant sequences, while does not behave well for decaying sequences.

From Figure \ref{fig:2}, we observe that AdapSQP consistently performs better~than $\ell_1$ AdapSQP in all noise levels for all constants setups, although the improvement is not significant and we only consider a basic design for $\ell_1$ AdapSQP, so that the two methods are not fully comparable. Furthermore, AdapSQP~utilizes more information than $\ell_1$ AdapSQP, such as the Hessians of constraints that~are typically utilized for accelerating the local convergence of SQP schemes. That said, we believe there is evidence supporting the usage of differentiable merit functions. We leave both the theoretical and empirical investigations of local benefits of the augmented Lagrangian to future work. We also see from Figure \ref{fig:2} that the two methods are robust to tuning parameters $C_{grad}, C_f$---the KKT residuals do not change much when we vary $C_{grad}, C_f$ from $1$ to $50$. This is because the sample complexity to ensure conditions \eqref{equ:ran:cond:1}, \eqref{equ:ran:cond:3}, \eqref{equ:ran:cond:4}, studied in Section \ref{sec:4.2}, is proportional to the reciprocal of certain quantities, which converge to zero as $k\rightarrow\infty$ and dominate the order of the sample complexity.

In summary, we observe that the performance of AdapSQP is better than $\ell_1$ AdapSQP; and these two line search schemes indeed reflect the benefits of generating a large batch set and perform better than NonAdapSQP and~$\ell_1$~SQP. The line search schemes are robust to tuning parameters and noise levels, while $\ell_1$ SQP and NonAdapSQP are sensitive to prespecified sequences and noise levels. Figure \ref{fig:1} also  suggests that line search schemes are not preferable over $\ell_1$ SQP and NonAdapSQP for $\sigma^2 = 10^{-8}$. This is because the latter two methods can have a precise gradient estimate even with one sample. Thus, the benefits gained from a large batch set is marginal.

\vskip2pt
\noindent\textbf{Sample complexity.} We compare the number of evaluations of the objective and its gradient, which is equivalent to comparing the corresponding generated sample sizes. Similar to the KKT residuals, we average generated sample sizes over all convergent runs. The results are shown in Figures \ref{fig:3} and \ref{fig:4}. Each plot in Figures \ref{fig:3} and \ref{fig:4} corresponds to a setup of constants $C_{grad}, C_f$. The boxes of NonAdapSQP and $\ell_1$ SQP do not appear in Figure \ref{fig:4} since both methods do not require objective evaluations (if we prespecify the Lipschitz constants). The results of these two methods in Figure \ref{fig:3} correspond to the setup $\alpha_k, \beta_k = 0.01$. In general, with small stepsizes, more iterations are performed so that more samples are generated. However, even with small stepsizes, we note from Figure \ref{fig:3} that $\ell_1$ SQP and NonAdapSQP require much fewer gradient evaluations than the two line search schemes when $\sigma^2=10^{-8}$, under which all methods converge to similar KKT residuals (cf. Figure \ref{KConst1}). Thus, clearly, line search schemes are more expensive to perform when the noise level is small.

Between the two line search schemes, we see from Figures \ref{fig:3} and \ref{fig:4} that AdapSQP requires slightly fewer gradient and objective evaluations. The difference is more evident for large noise variance. By the observation from Figure~\ref{fig:2}~that AdapSQP converges to smaller KKT residuals, Figures \ref{fig:3} and \ref{fig:4} further justify the choice of the augmented Lagrangian merit function. We point out that the evaluations~of $\ell_1$ AdapSQP may be decreased by a different update of the penalty parameter or a different definition for $\mA_k, \mB_k$; while how to refine the design of $\ell_1$ AdapSQP is not considered here. In addition, we see that~both AdapSQP and $\ell_1$ AdapSQP have stable performance in terms of the generated sample sizes for different constants setups. This observation again suggests that both line search schemes are~robust to tuning parameters.     

We should mention that adopting the augmented Lagrangian merit function also requires to estimate the Hessian of the objective. We do not show the Hessian result since all methods using the $\ell_1$ merit function do not require~such information for global convergence, and the number of Hessian evaluations is dominated by the number of gradient evaluations. The cost of second-order information may be alleviated to some extent if we take local convergence into account (as all methods would have to estimate the Hessian then). 

\begin{figure}[!htp]
	\centering     
	\subfigure[$C_{grad} = C_f = 1$]{\label{G1}\includegraphics[width=43.5mm]{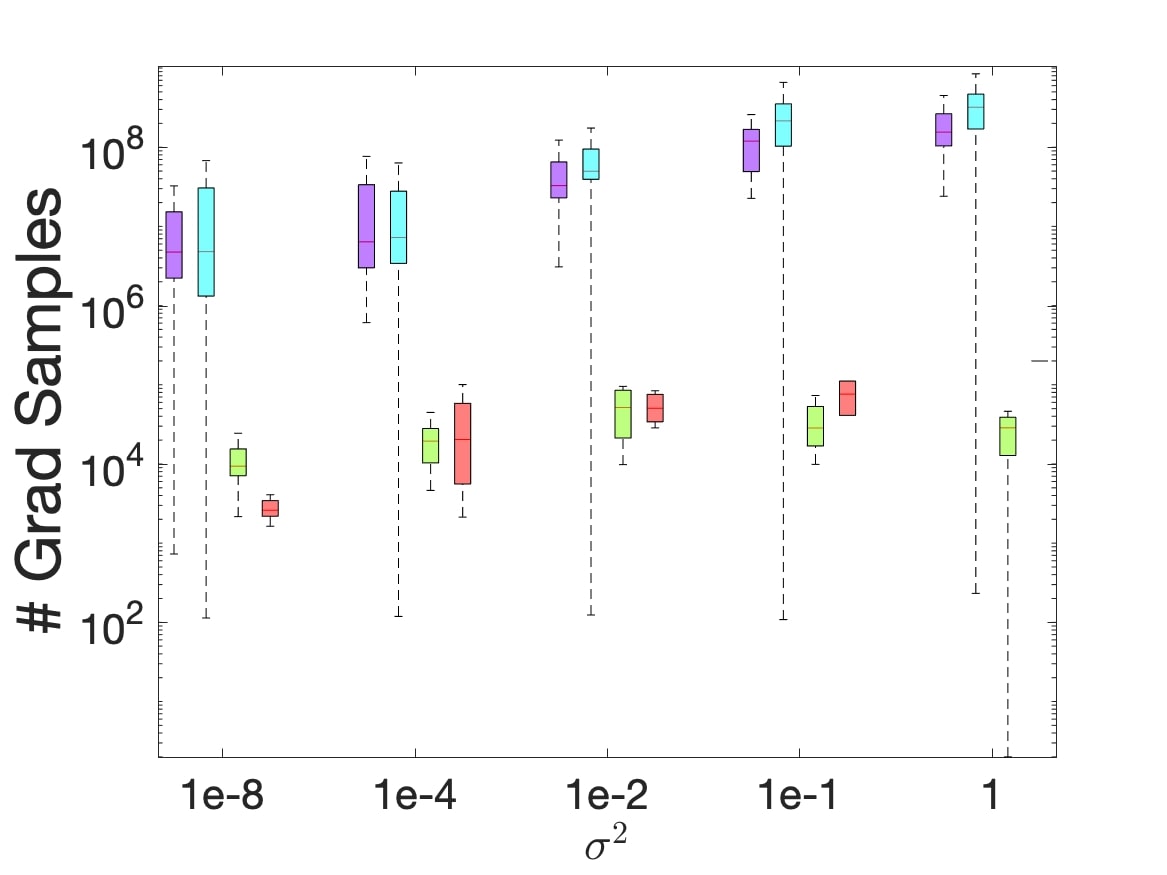}}
	\subfigure[$C_{grad} = C_f = 5$]{\label{G2}\includegraphics[width=43.5mm]{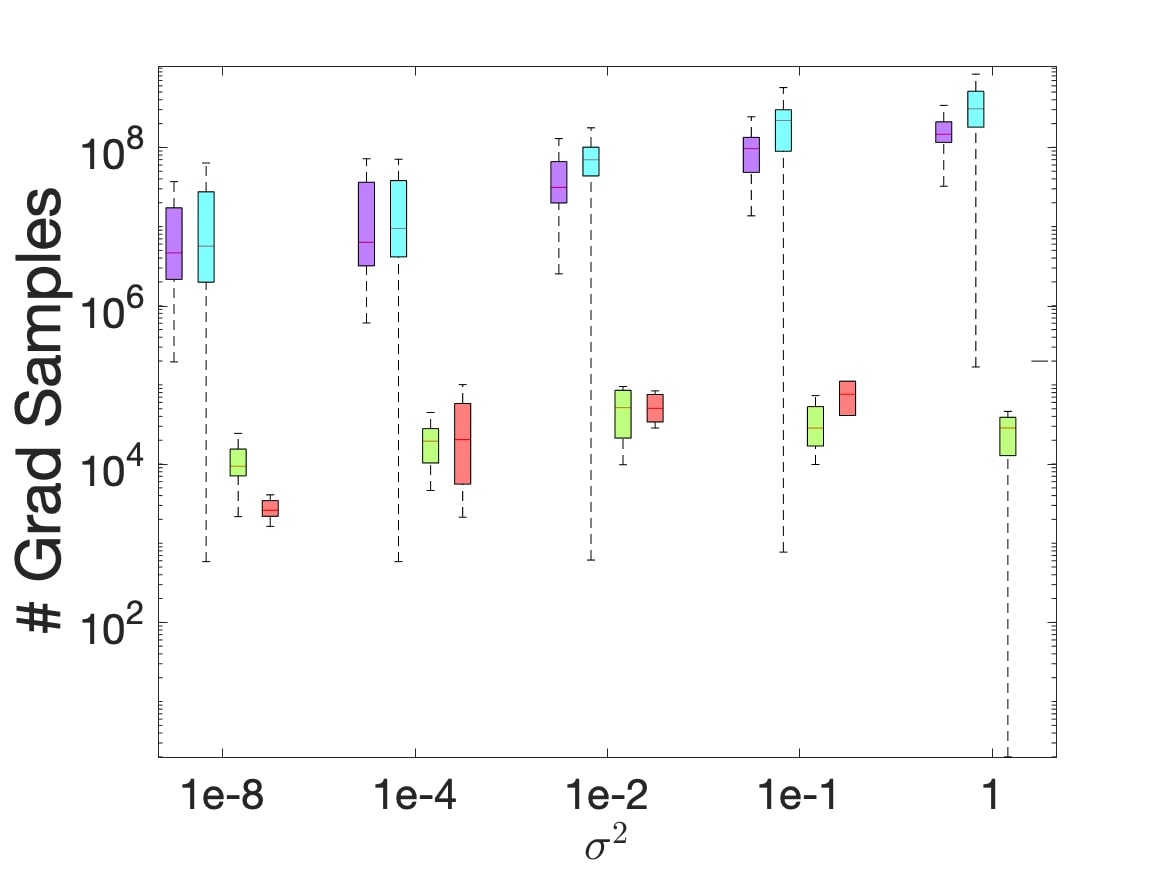}}
	\subfigure[$C_{grad} = C_f = 10$]{\label{G3}\includegraphics[width=43.5mm]{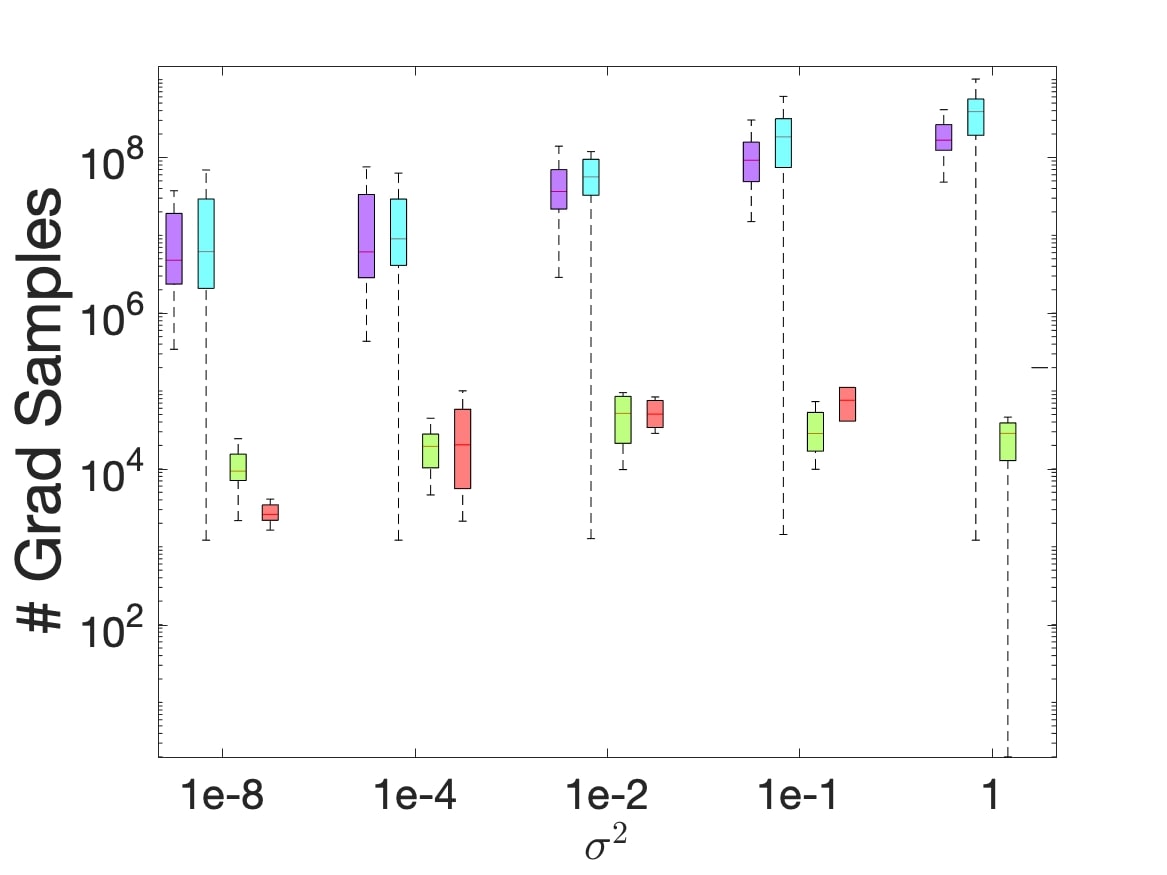}}
	\subfigure[$C_{grad} = C_f = 50$]{\label{G4}\includegraphics[width=43.5mm]{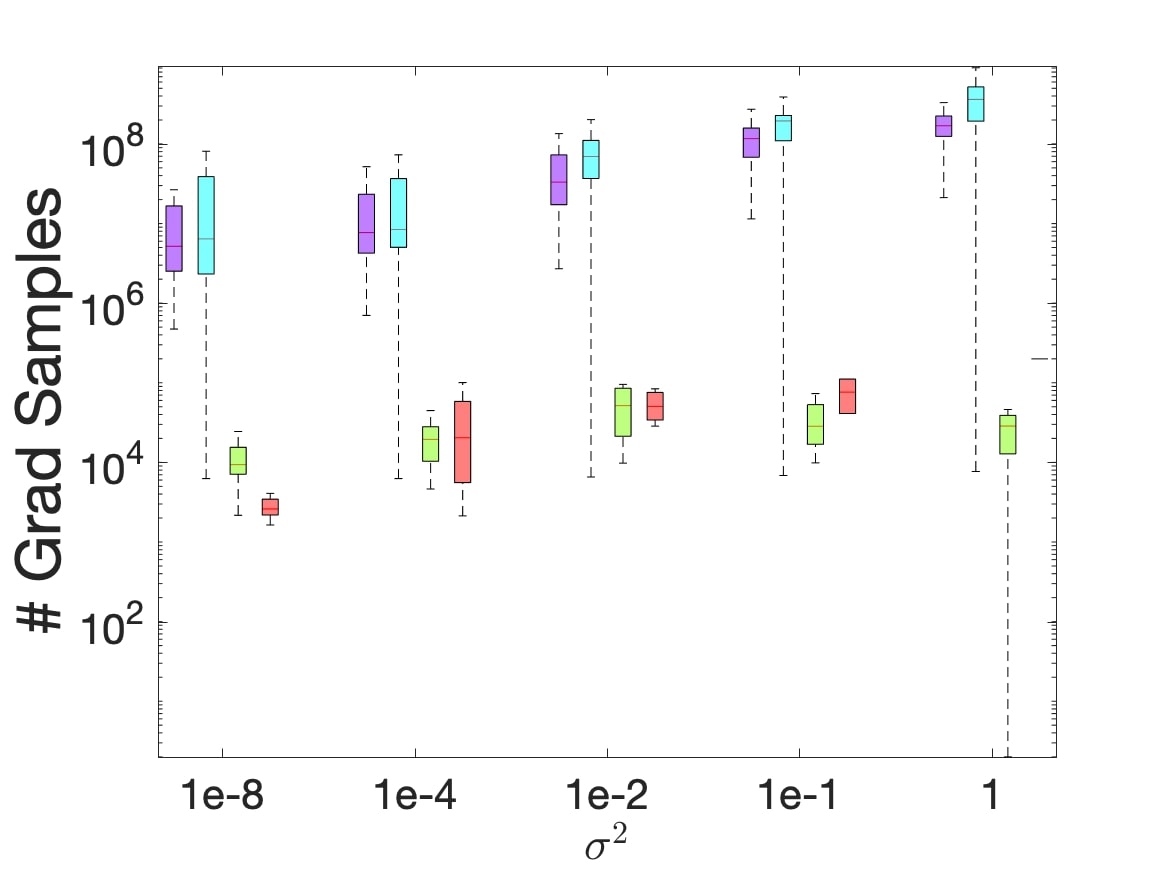}}
	\includegraphics[width=0.5\textwidth]{Figure/legend}
	\caption{Gradient evaluation boxplot. Each panel corresponds to a constants setup. The results of NonAdapSQP and $\ell_1$ SQP on all panels are the same, which correspond to $\alpha_k, \beta_k =  0.01$.}\label{fig:3}
\end{figure}

\begin{figure}[!h]
	\centering     
	\subfigure[$C_{grad} = C_f = 1$]{\label{F1}\includegraphics[width=43.5mm]{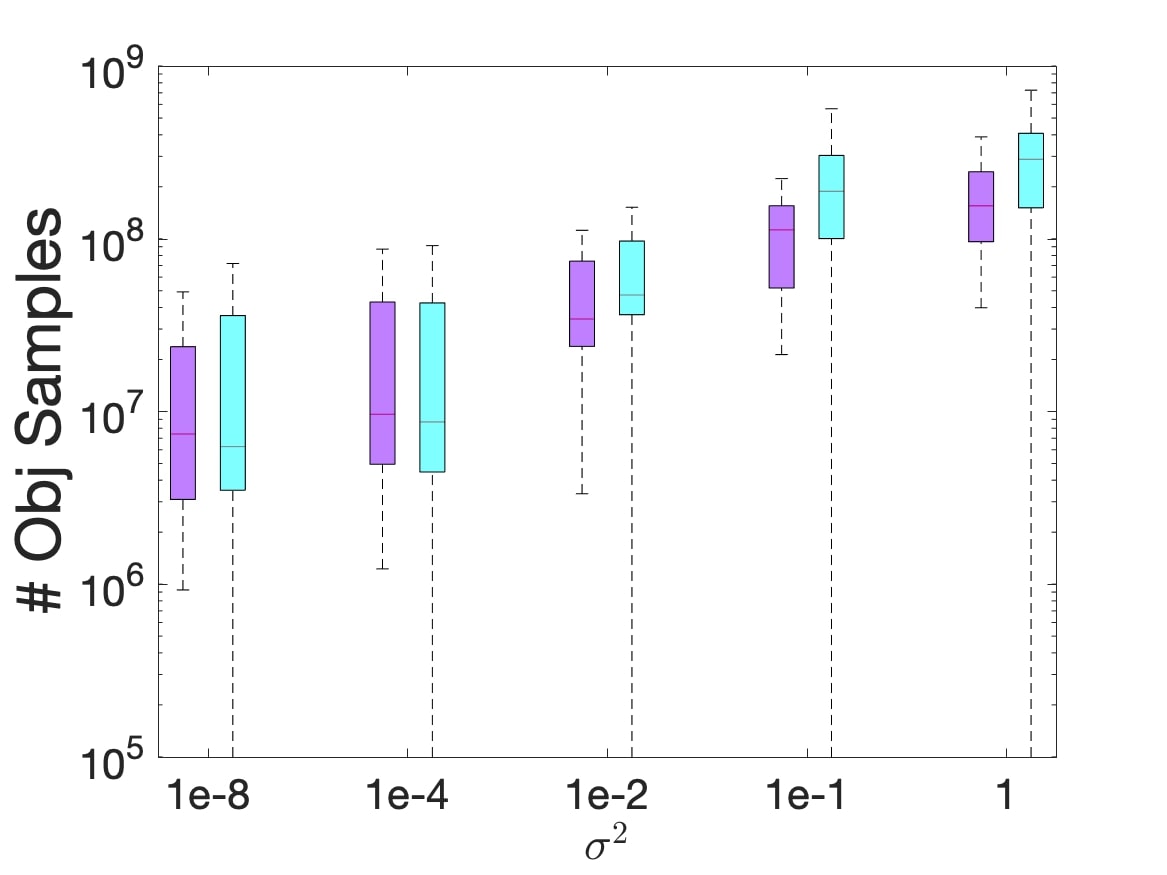}}
	\subfigure[$C_{grad} = C_f = 5$]{\label{F2}\includegraphics[width=43.5mm]{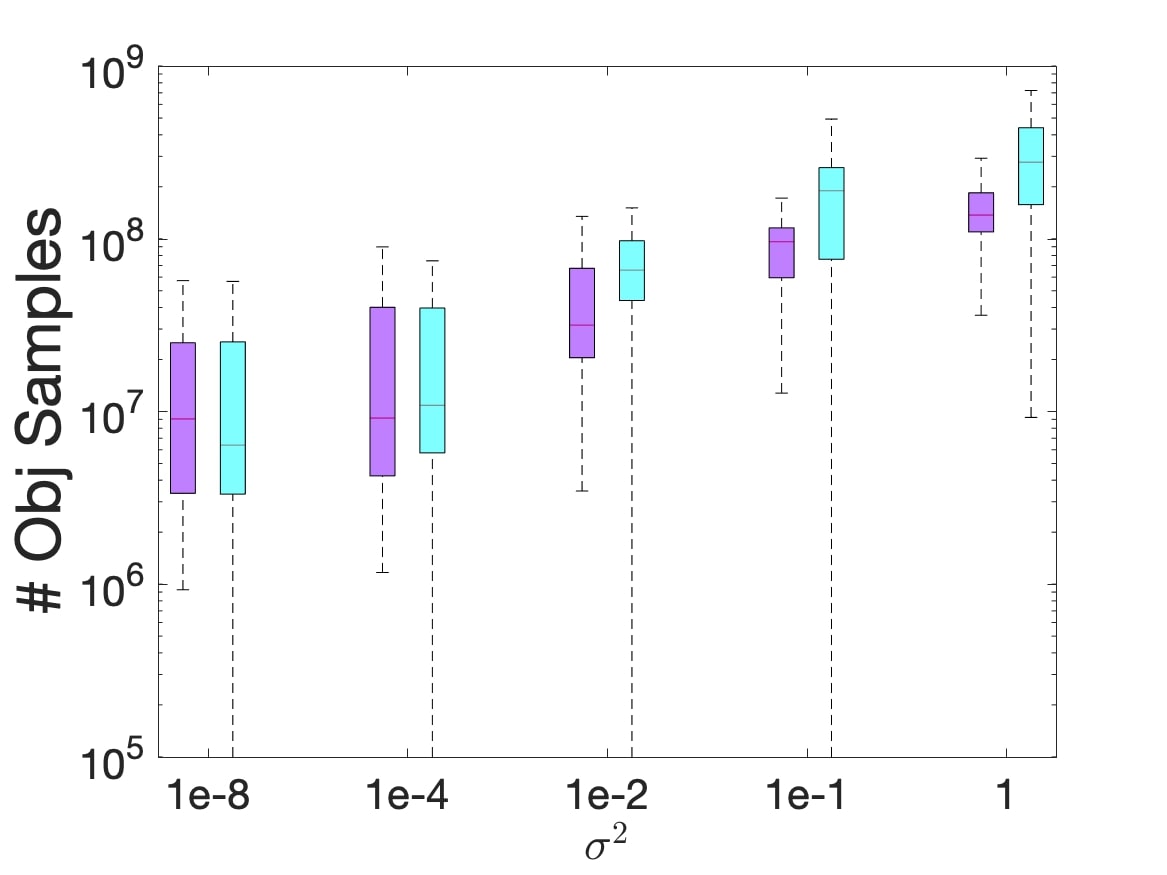}}
	\subfigure[$C_{grad} = C_f = 10$]{\label{F3}\includegraphics[width=43.5mm]{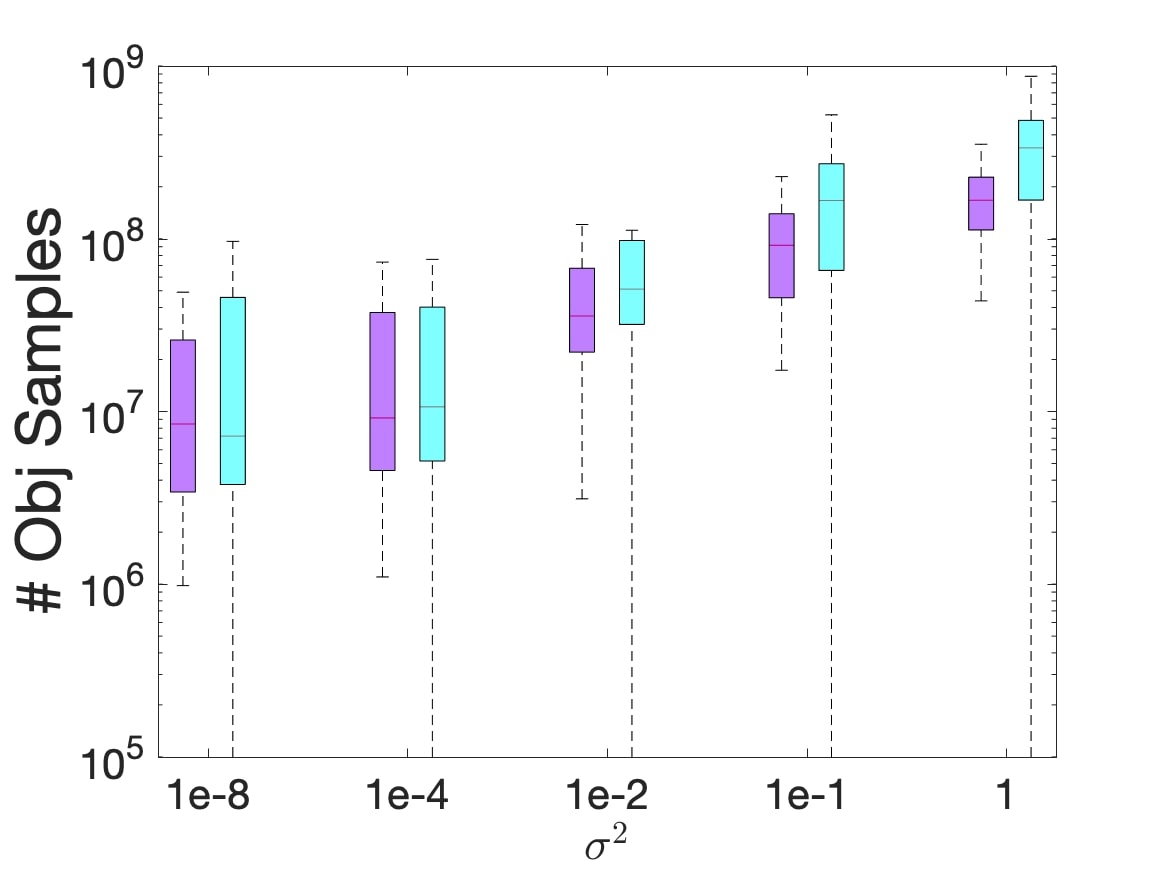}}
	\subfigure[$C_{grad} = C_f = 50$]{\label{F4}\includegraphics[width=43.5mm]{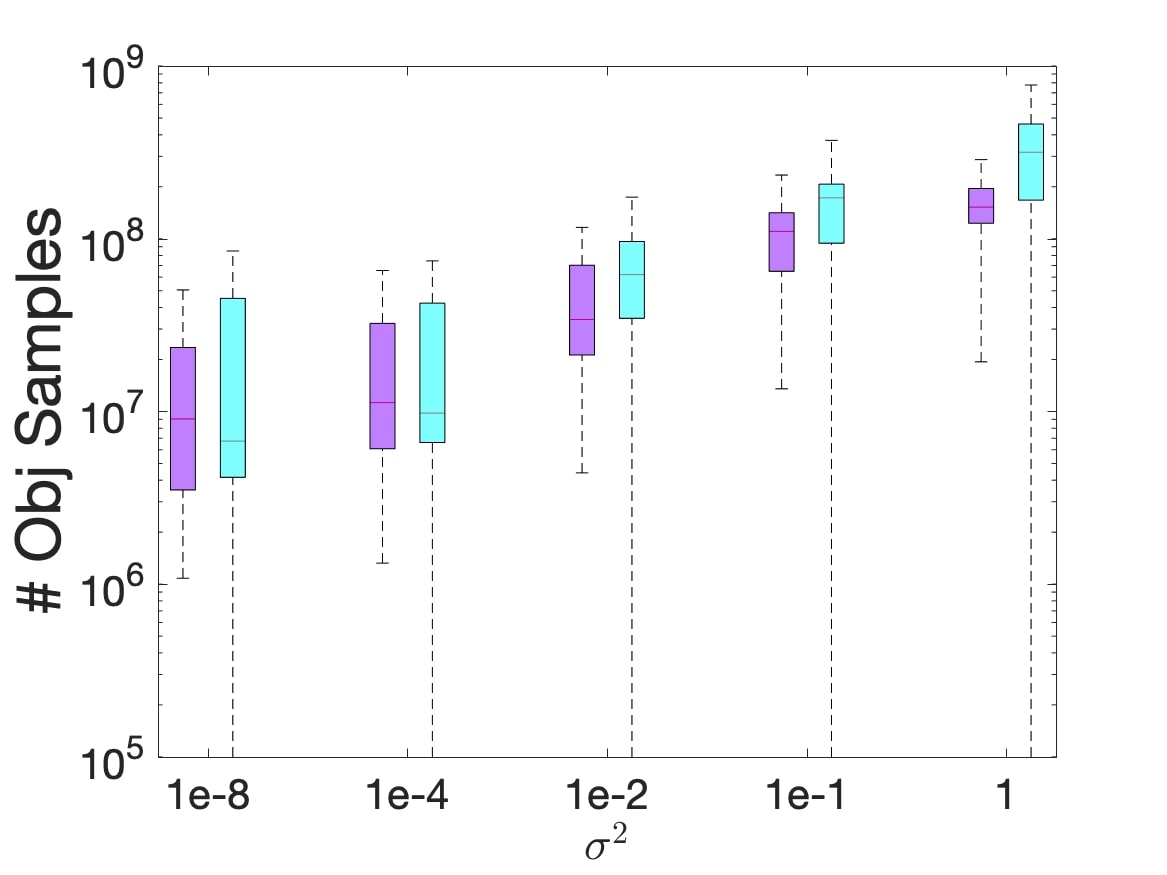}}
	
	\includegraphics[width=0.25\textwidth]{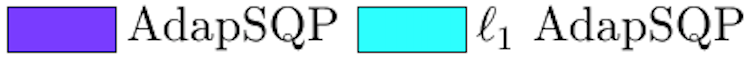}
	\caption{Objective evaluation boxplot. Each panel corresponds to a setup of constants.}\label{fig:4}
\end{figure}

\begin{figure}[!thp]
	\centering     
	\subfigure[$C_{grad} = C_f = 1$]{\label{K1}\includegraphics[width=55mm]{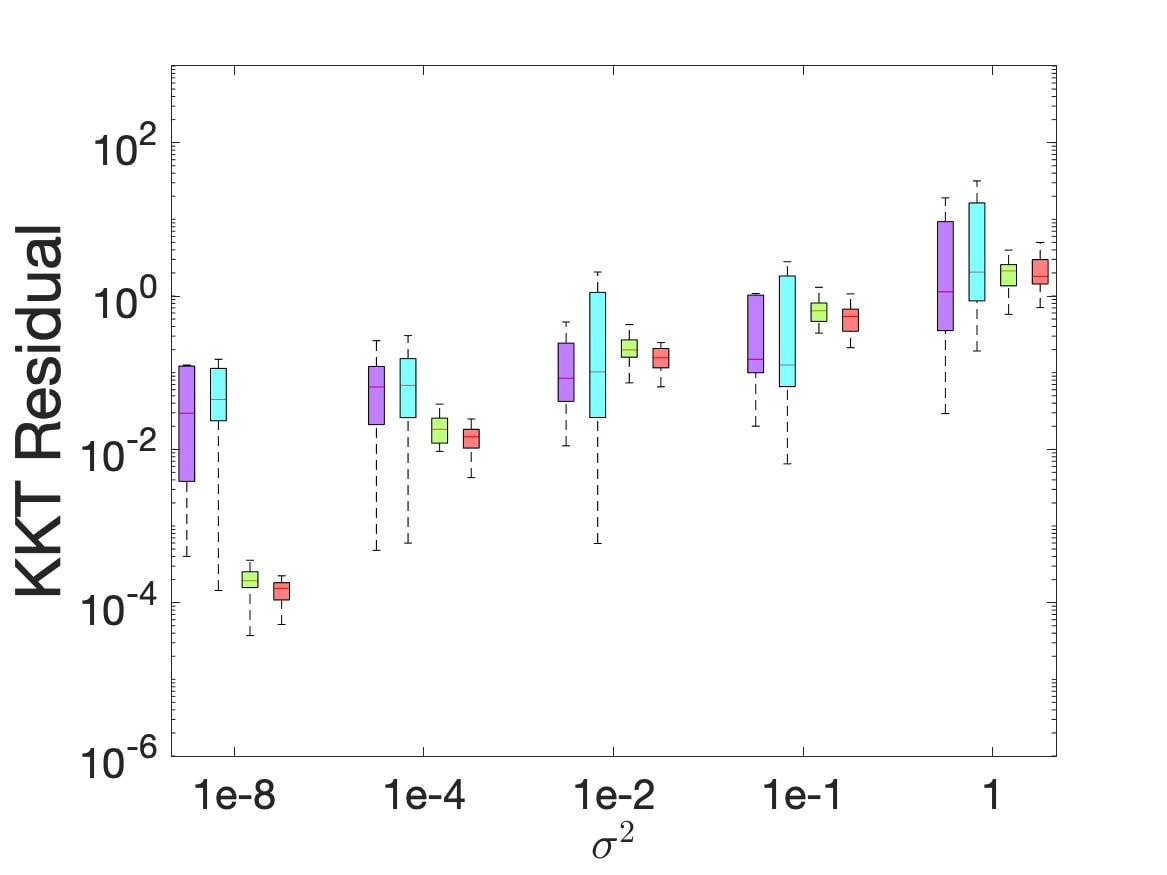}}
	\subfigure[$C_{grad} = C_f = 5$]{\label{F2}\includegraphics[width=55mm]{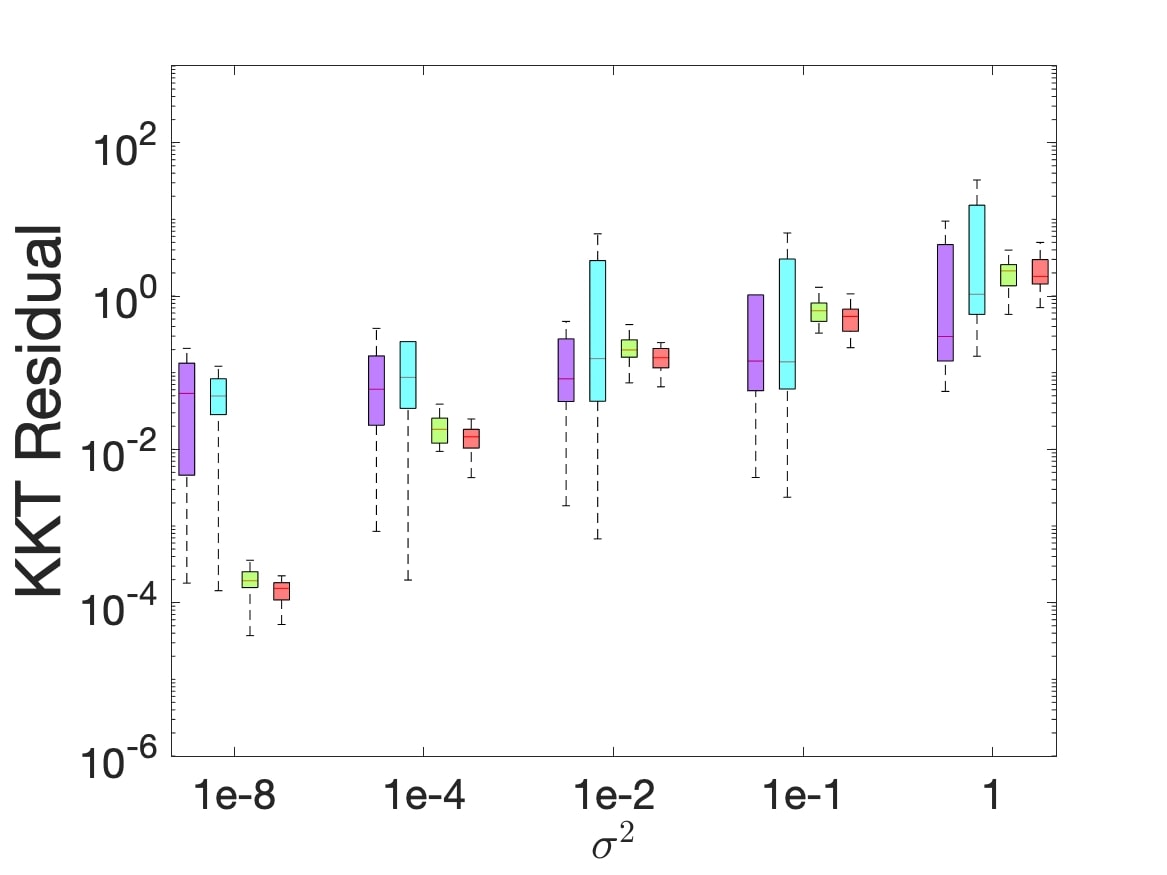}}
	\subfigure[$C_{grad} = C_f = 10$]{\label{F3}\includegraphics[width=55mm]{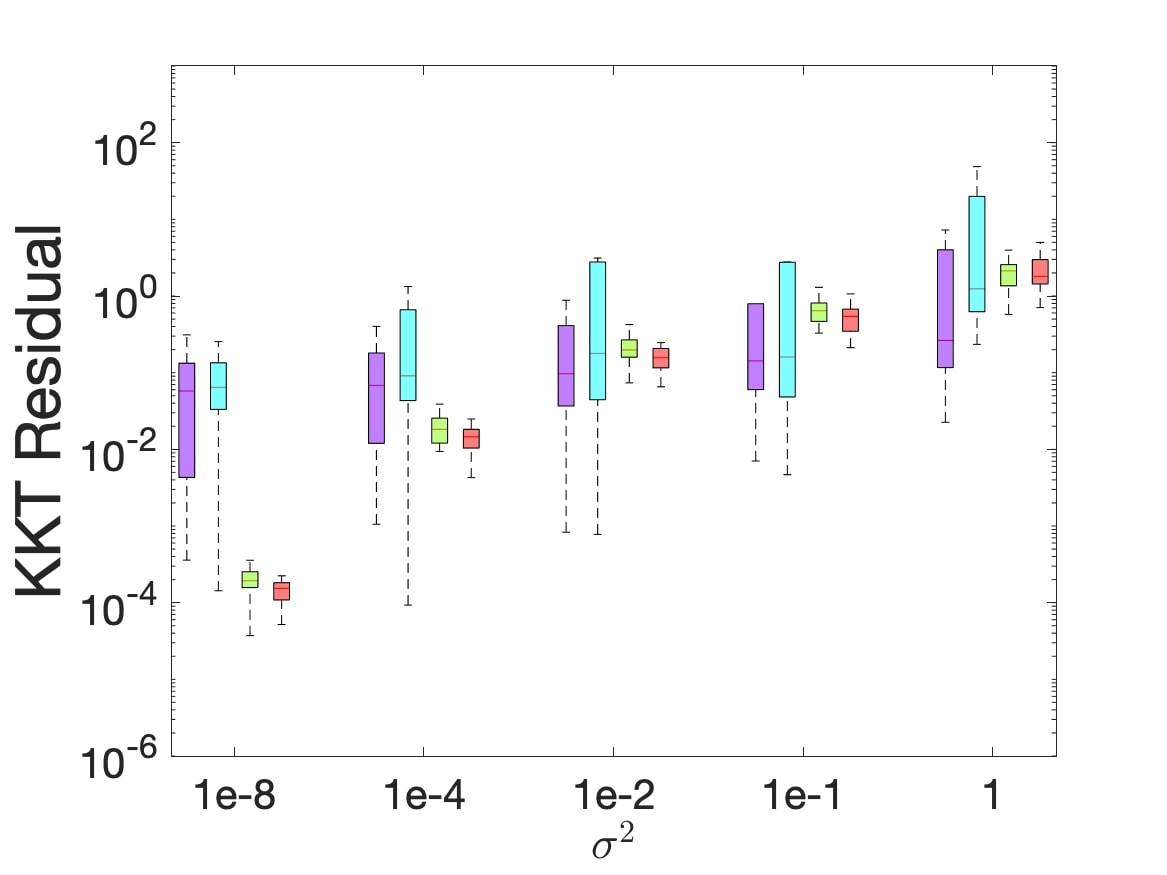}}
	\subfigure[$C_{grad} = C_f = 50$]{\label{F4}\includegraphics[width=55mm]{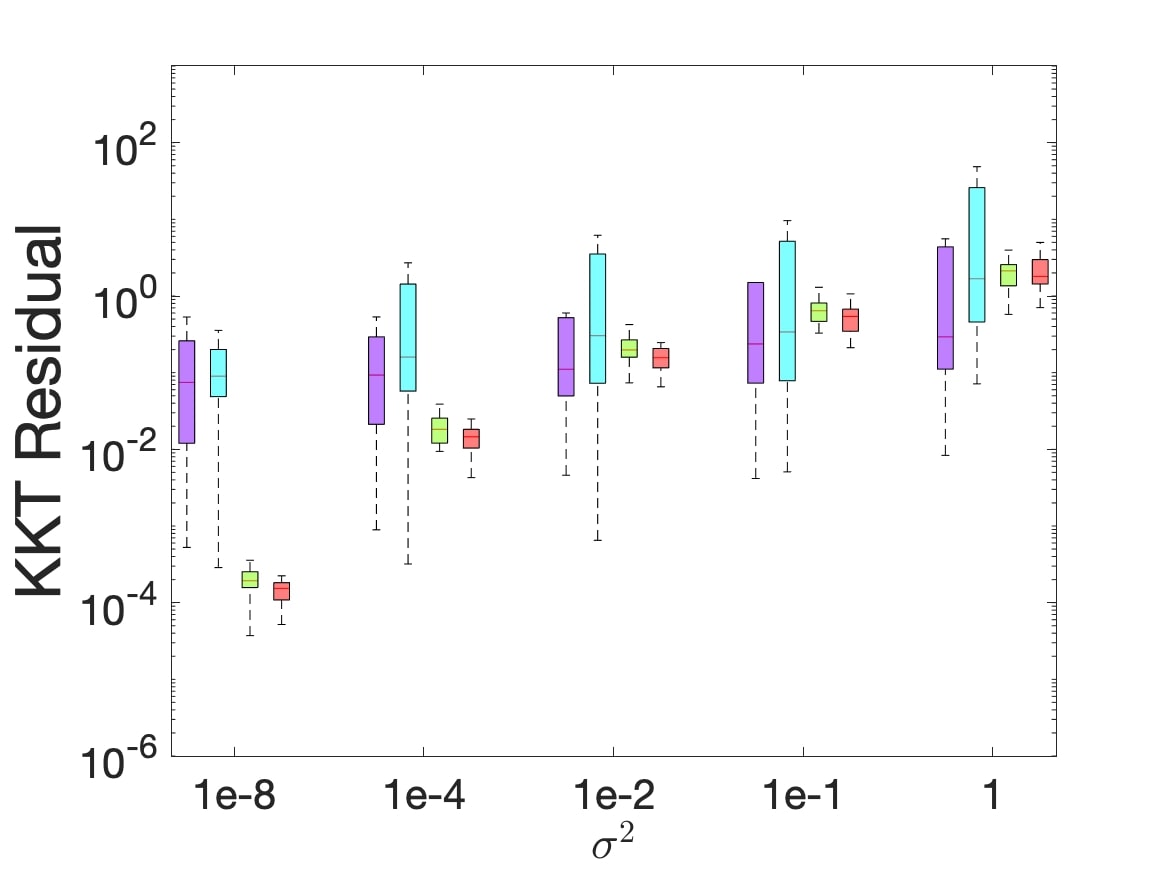}}
	
	\includegraphics[width=0.5\textwidth]{Figure/legend}
	\caption{KKT residual boxplot. Each panel corresponds to a setup of constants. The results of NonAdapSQP and $\ell_1$ SQP on all panels are the same, which correspond to $\alpha_k, \beta_k =  0.01$.}\label{fig:6}
\end{figure}

\begin{figure}[!tp]
	\centering     
	\subfigure[AdapSQP]{\label{A}\includegraphics[width=58mm]{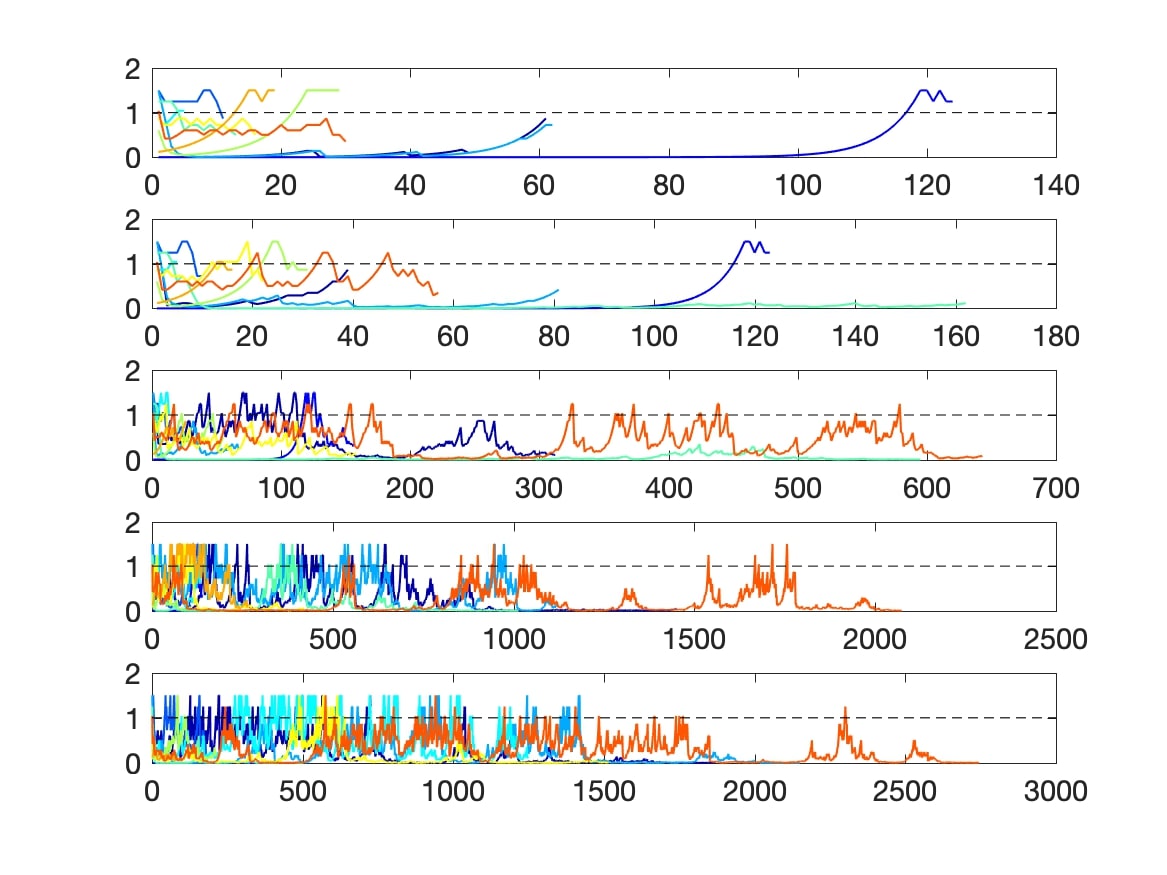}}
	\subfigure[$\ell_1$ AdapSQP]{\label{B}\includegraphics[width=58mm]{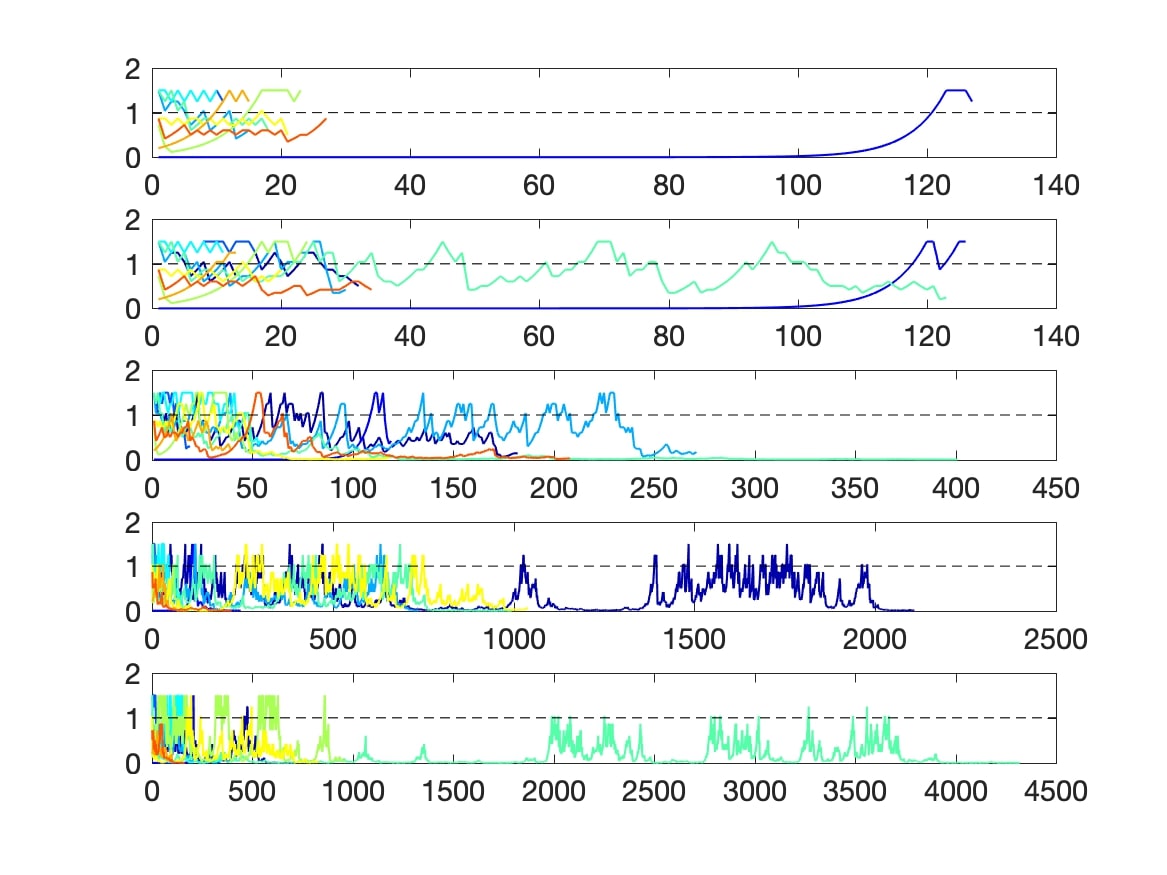}}
	\caption{Stepsize process. Each figure has five rows, from top to bottom, corresponding~to $\sigma^2 = 10^{-8}, 10^{-4}, 10^{-2}, 10^{-1}, 1$. Each plot has 10 lines corresponding to the stepsizes sequences of 10 convergent problems. The dash line corresponds to the unit stepsize.}\label{fig:5}
\end{figure}

Finally, an interesting comparison metric is to fix the total number of~samples (for which different methods have different usages)~and then compare the attained KKT residuals. To investigate this metric, we again take $\alpha_k, \beta_k=0.01$ for NonAdapSQP and $\ell_1$ SQP, and vary $C_{grad}, C_f$ for AdapSQP and $\ell_1$ AdapSQP. For all methods, we adjust the stopping criteria \eqref{equ:criteria} by only checking if the generated sample size exceeds $10^6$ or not. We see from Figures \ref{fig:3} and \ref{fig:4} that AdapSQP and $\ell_1$ AdapSQP require more than $10^6$ samples; thus, their performance is worsened under the above setup. However, by \eqref{equ:criteria}, the performance of $\ell_1$ SQP and NonAdapSQP is enhanced due to more iterations are performed. Figure  \ref{fig:6} shows the KKT residual boxplot. From Figure \ref{fig:6}, we see that two fully stochastic methods---$\ell_1$ SQP and NonAdapSQP---achieve smaller KKT residuals than two line search methods---AdapSQP and $\ell_1$ AdapSQP---when $\sigma^2$ is small. This suggests that fully stochastic methods are preferable if each sample can accurately estimate the stochastic quantity.~However, from the figure, we also see that two line search methods are comparable to fully stochastic methods when $\sigma^2$ is large. Due to the adaptivity and much fewer iterations that are conducted, line search methods are preferable in this scenario.

\vskip4pt
\noindent\textbf{Stochastic stepsize.} We now investigate the random process of stepsizes~selected by stochastic line search. We consider the setup of $C_{grad} = C_f = 1$,~and use the first 10 convergent problems to visualize the process. In particular, for each problem, we randomly pick one out of five runs, and show a sequence of stepsizes that are associated to successful steps (i.e. the Armijo condition is satisfied; cf. Line 9 of Algorithm \ref{alg:ASto:SQP}).

From Figure \ref{fig:5}, we see that the stochastic stepsize can frequently exceed $1$ even if it has been diminished to a small value. Allowing stepsizes to exceed $1$ is a special property in stochastic optimization, which is also shared by \cite{Berahas2021Sequential}~but not common in deterministic optimization. Different from NonAdapSQP and~\cite{Berahas2021Sequential}, line search can significantly increase the stepsize back as desired even if the current stepsize is small. This property reveals the exclusive benefit of utilizing line search, and is crucial to enable a fast convergence in practice. In other~words, line search can suggest a stepsize that is larger than the prespecified, controlled sequence, especially for large $k$. We note that the stepsize selection scheme in $\ell_1$ SQP enjoys certain  adaptivity. However, if $\alpha_k$ is small, then all the following stepsizes cannot be large because of the relation between $\alpha_k$ and prespecified $\beta_k$ in \cite[Lemma 3.6]{Berahas2021Sequential}.

Although the stepsizes for some problems  are finally small in Figure~\ref{fig:5}, it is not a negative evidence of the local benefits of the augmented Lagrangian. This is simply because studying local behavior requires a diminishing Hessian modification, i.e. $B_k\approx \nabla_{\bx}^2\mL_k$, while we fix $B_k$ to be $B_k = I$ in our simulation. It is unclear how to quantify local benefits based on a stochastic stepsize process. We leave both theoretical and empirical investigations to future work.

\section{Conclusions}\label{sec:6}

We proposed an adaptive StoSQP for solving constrained stochastic optimization problems. Our StoSQP scheme adopts a differentiable exact augmented Lagrangian as the merit function and incorporates a line search procedure to select the stepsize. To this end, we first revisited a classical SQP method~in \cite{Lucidi1990Recursive} for deterministic objectives. We simplified that algorithm and then designed a non-adaptive StoSQP, where the stepsize is given by a deterministic prespecified sequence. Finally, using a stochastic line search method, we developed an adaptive StoSQP that chooses the stepsize adaptively in each iteration. By adjusting the condition on selecting the penalty parameter, the algorithm automatically avoids converging to stationary points of the merit function where KKT residuals do not vanish. For both non-adaptive StoSQP and adaptive StoSQP, we established an ``almost sure" convergence result, which differs from the result of convergence in expectation in \cite{Berahas2021Sequential}.

One future work is to design StoSQP algorithms for solving inequality-constrained stochastic optimization problems. Moreover, establishing a local convergence result is very promising, which helps to understand whether differentiable merit functions enjoy local benefits or not in stochastic optimization. In addition, the iteration complexity analysis of our StoSQP remains an open question, that requires a deeper understanding on the random walk of the~selected penalty parameter. Finally, applying some acceleration techniques such as momentum and heavy ball to StoSQP is also an interesting future research direction.

\section*{Acknowledgments}

We would like to thank Associated Editor and two anonymous reviewers for helpful and instructive comments. Sen Na would like to thank Nick Gould for helping with the implementation of CUTEst in Julia. This material was completed in part with resources provided by the University of Chicago Research Computing Center. This material was based upon work supported by the U.S. Department of Energy, Office of Science, Office of Advanced Scientific Computing Research (ASCR) under Contract DE-AC02-06CH11347 and by NSF through award CNS-1545046.


\appendix
\numberwithin{equation}{section}
\numberwithin{theorem}{section}

\section{Proofs of Section \ref{sec:3}}\label{sec:appen:3}

\subsection{Proof of Lemma \ref{lem:6}}\label{pf:lem:6}

The expectation is taken over randomness of $\xi_g^k$ and $\xi_H^k$, which means that the $k$-th iterate $(\bx_k, \blambda_k)$ is supposed to be fixed. By the unbiasedness condition in Assumption \ref{ass:ran:2}, we have
\begin{align*}
&\mE_{\xi^k_g}[\barDelta \bx_k] = \mE[\barDelta\bx_k \mid \bx_k, \blambda_k] \stackrel{\eqref{equ:ran:Newton}}{=} -\mE\sbr{\begin{pmatrix}
	I & \0
	\end{pmatrix}\left(\begin{smallmatrix}
	B_k & G_k^T\\
	G_k & \0
	\end{smallmatrix}\right)^{-1}\left(\begin{smallmatrix}
	\bnabla_{\bx}\mL_k\\
	c_k
	\end{smallmatrix}\right) \Mid \bx_k, \blambda_k}\\
& =  -\begin{pmatrix}
I & \0
\end{pmatrix}\left(\begin{smallmatrix}
B_k & G_k^T\\
G_k & \0
\end{smallmatrix}\right)^{-1}\mE\sbr{\left(\begin{smallmatrix}
	\bnabla_{\bx}\mL_k\\
	c_k
	\end{smallmatrix}\right)\Mid \bx_k, \blambda_k } \stackrel{\eqref{Bound:var}}{=}   -\begin{pmatrix}
I & \0
\end{pmatrix}\left(\begin{smallmatrix}
B_k & G_k^T\\
G_k & \0
\end{smallmatrix}\right)^{-1}\left(\begin{smallmatrix}
\nabla_{\bx}\mL_k\\
c_k
\end{smallmatrix}\right)
\stackrel{\mathclap{\eqref{equ:Newton:AL}}}{=}  \Delta\bx_k,
\end{align*}
and
\begin{align*}
\mE_{\xi^k_g,\xi^k_H}[\barDelta \blambda_k]
&= \mE[\barDelta\blambda_k\mid \bx_k, \blambda_k] \stackrel{\mathclap{\eqref{equ:ran:Newton}}}{=} -\mE\sbr{(G_kG_k^T)^{-1}G_k\bnabla_{\bx}\mL_k +(G_kG_k^T)^{-1}\barM_k^T\barDelta\bx_k \mid \bx_k, \blambda_k}\\
&= -(G_kG_k^T)^{-1}G_k\mE[\bnabla_{\bx}\mL_k \mid \bx_k, \blambda_k] - (G_kG_k^T)^{-1}\mE[\barM_k^T\barDelta\bx_k\mid \bx_k, \blambda_k]\\
& = -(G_kG_k^T)^{-1}G_k\nabla_{\bx}\mL_k - (G_kG_k^T)^{-1}\mE[\barM_k^T \mid \bx_k, \blambda_k]\mE[\barDelta\bx_k\mid \bx_k, \blambda_k]\\
& \stackrel{\mathclap{\eqref{Bound:var}}}{=} -(G_kG_k^T)^{-1}G_k\nabla_{\bx}\mL_k - (G_kG_k^T)^{-1}M_k^T\Delta\bx_k\\
& \stackrel{\mathclap{\eqref{equ:Newton:AL}}}{=} \Delta\blambda_k,
\end{align*}
where the fourth equality is due to the independence between $\xi_g^k$ and $\xi_H^k$. Moreover, we have
\begin{align}\label{NN:1}
\mE_{\xi^k_g}\sbr{\|\barDelta \bx_k\|^2} & = \|\Delta\bx_k\|^2 + \mE_{\xi^k_g}\sbr{\|\barDelta\bx_k - \Delta\bx_k\|^2} \nonumber\\
& \stackrel{\mathclap{\eqref{equ:ran:Newton}}}{=}\|\Delta\bx_k\|^2+  \mE\bigg[\bigg\| \begin{pmatrix}
I & \0
\end{pmatrix}\underbrace{\left(\begin{smallmatrix}
	B_k & G_k^T\\
	G_k & \0
	\end{smallmatrix}\right)^{-1}}_{\eqqcolon \mK}\left(\begin{smallmatrix}
\barg_k - \nabla f_k\\
\0
\end{smallmatrix}\right) \bigg\|^2\Mid \bx_k, \blambda_k \bigg] \nonumber\\
& \stackrel{\mathclap{\eqref{Bound:var}}}{\leq} \|\Delta\bx_k\|^2 + \|\mK\|^2\psi.
\end{align}
We now bound $\|\mK\|$. Suppose $G_k^T = Y_kE_k$ where $Y_k$ has orthonormal columns spanning $\text{Image}(G_k^T)$ and $E_k$ is a square matrix, and suppose $Z_k$ has orthonormal columns spanning $\text{Ker}(G_k)$. By Assumption \ref{ass:ran:1}, we know $E_k$ is nonsingular and $Z_k^TB_kZ_k \succeq \gamma_{RH} I$.~By directly verifying $\left(\begin{smallmatrix}
B_k & G_k^T\\
G_k & \0
\end{smallmatrix}\right)\mK = I$, one can show
\begin{equation*}
\mK= 
\left(\begin{smallmatrix}
Z_k(Z_k^TB_kZ_k)^{-1}Z_k^T & (I - Z_k(Z_k^TB_kZ_k)^{-1}Z_k^TB_k)Y_kE_k^{-1}\\
E_k^{-1}Y_k^T(I - B_kZ_k(Z_k^TB_kZ_k)^{-1}Z_k^T) & E_k^{-1}Y_k^T(B_kZ_k(Z_k^TB_kZ_k)^{-1}Z_k^TB_k - B_k)Y_kE_k^{-1}
\end{smallmatrix}\right).
\end{equation*}
Noting that $\|E_k^{-1}\|^2 = \|E_k^{-1}(E_k^{-1})^T\| = \|(G_kG_k^T)^{-1}\| \stackrel{\eqref{Prop:1}}{\leq} 1/\kappa_{1,G}$, and $\|(Z_k^TB_kZ_k)^{-1}\| \leq 1/\gamma_{RH}$, we can bound $\|\mK\|$ by summing the spectrum norm of each block, and get
\begin{equation*}
\|\mK\|\leq \frac{1}{\gamma_{RH}} + \frac{2}{\sqrt{\kappa_{1,G}}}\rbr{1 + \frac{\kappa_B}{\gamma_{RH}}} + \frac{1}{\kappa_{1,G}}\rbr{\frac{\kappa_B^2}{\gamma_{RH}} + \kappa_B}.
\end{equation*}
Since $\kappa_B \geq 1\geq \gamma_{RH}\vee \kappa_{1,G}$, we can simplify the bound by
\begin{equation*}
\|\mK\|\leq \frac{5\kappa_B}{\sqrt{\kappa_{1,G}}\gamma_{RH}} + \frac{2\kappa_B^2}{\kappa_{1,G}\gamma_{RH}} \leq \frac{7\kappa_B^2}{\kappa_{1,G}\gamma_{RH}}.
\end{equation*}
Plugging the above inequality into \eqref{NN:1}, we have
\begin{equation}\label{NN:2}
\mE_{\xi^k_g}\sbr{\|\barDelta \bx_k\|^2} \leq \|\Delta\bx_k\|^2 + \frac{49\kappa_B^4}{\kappa_{1,G}^2\gamma_{RH}^2}\psi.
\end{equation}
Similarly, we can bound $\mE_{\xi^k_g,\xi^k_H}[\|\barDelta\blambda_k\|^2]$ as follows.
\begin{align}\label{NN:3}
&\mE_{\xi^k_g,\xi^k_H}\sbr{\|\barDelta\blambda_k\|^2} = \|\Delta\blambda_k\|^2 + \mE_{\xi^k_g,\xi^k_H}\sbr{\|\barDelta\blambda_k - \Delta\blambda_k\|^2} \nonumber\\
& \stackrel{\mathclap{\eqref{equ:ran:Newton}}}{=}  \|\Delta\blambda_k\|^2 + \mE_{\xi^k_g,\xi^k_H}\sbr{\nbr{(G_kG_k^T)^{-1}\rbr{G_k\bnabla_{\bx}\mL_k + \barM_k^T\barDelta\bx_k -G_k\nabla_{\bx}\mL_k - M_k\Delta\bx_k} }^2 } \nonumber\\
&\stackrel{\mathclap{\eqref{Prop:1}}}{\leq} \|\Delta\blambda_k\|^2 + \frac{2}{\kappa_{1,G}^2}\cbr{\mE_{\xi^k_g}\sbr{\nbr{G_k\bnabla_{\bx}\mL_k  - G_k\nabla_{\bx}\mL_k}^2} + \mE_{\xi^k_g,\xi^k_H}\sbr{\|\barM_k^T\barDelta\bx_k-M_k^T\Delta\bx_k\|^2 }} \nonumber\\
& \stackrel{\mathclap{\eqref{Prop:1}}}{\leq}\|\Delta\blambda_k\|^2 + \frac{2}{\kappa_{1,G}^2}\cbr{\kappa_{2,G} \mE_{\xi^k_g}\sbr{\|\barg_k - \nabla f_k\|^2} + \mE_{\xi^k_g,\xi^k_H}\sbr{\|\barM_k^T\barDelta\bx_k-M_k^T\Delta\bx_k\|^2 } } \nonumber\\
&\stackrel{\mathclap{\eqref{Bound:var}}}{\leq} \|\Delta\blambda_k\|^2 + \frac{2}{\kappa_{1,G}^2}\cbr{\kappa_{2,G} \psi + \mE_{\xi^k_g,\xi^k_H}\sbr{\|\barM_k^T\barDelta\bx_k-M_k^T\Delta\bx_k\|^2 } }.
\end{align}
For the last term, we have
\begin{align}\label{NN:4}
&\mE_{\xi^k_g,\xi^k_H}\sbr{\|\barM_k^T\barDelta\bx_k-M_k^T\Delta\bx_k\|^2 } \nonumber\\
& = \mE_{\xi^k_g,\xi^k_H}\sbr{ \|(\barM_k - M_k)^T(\barDelta\bx_k - \Delta\bx_k) + M_k^T(\barDelta\bx_k - \Delta\bx_k) + (\barM_k - M_k)^T\Delta\bx_k\|^2  } \nonumber\\
& = \mE_{\xi^k_g,\xi^k_H}\sbr{ \|(\barM_k - M_k)^T(\barDelta\bx_k - \Delta\bx_k)\|^2 + \|M_k^T(\barDelta\bx_k - \Delta\bx_k)\|^2 + \|(\barM_k - M_k)^T\Delta\bx_k\|^2 } \nonumber\\
& \stackrel{\mathclap{\eqref{Prop:1}}}{\leq} \mE_{\xi^k_H}\sbr{\|\barM_k - M_k\|^2}\mE_{\xi^k_g}\sbr{\|(\barDelta\bx_k - \Delta\bx_k)\|^2} + \kappa_M^2\mE_{\xi^k_g}\sbr{\|(\barDelta\bx_k - \Delta\bx_k)\|^2} \nonumber\\
&\quad + \mE_{\xi^k_H}\sbr{\|\barM_k - M_k\|^2}\|\Delta\bx_k\|^2,
\end{align}
where the second equality is due to the independence between $\xi_g^k$ and $\xi_H^k$. We note that
\begin{align*}
\|\barM_k - M_k\|^2 &\stackrel{\eqref{def:MT}}{\leq} 2\|G_k\|^2\|\barH_k - \nabla^2f_k\|^2 +  2\|\barT_k - T_k\|^2 \nonumber\\
& \stackrel{\eqref{def:MT}}{\leq}  2\|G_k\|^2\|\barH_k - \nabla^2f_k\|^2 +  2\|\nabla f(\bx_k; \xi_H^k) - \nabla f_k\|^2\sum_{i=1}^{m}\|\nabla^2c_i(\bx_k)\|^2 \nonumber\\
&\stackrel{\eqref{Prop:1}}{\leq} 2\kappa_{2,G}\|\barH_k - \nabla^2f_k\|^2  + 2\kappa_{\nabla_{\bx}^2c}^2\|\nabla f(\bx^k; \xi_H^k) - \nabla f_k\|^2.
\end{align*}
Taking expectation over $\xi_H^k$ on both sides and using \eqref{Bound:var}, we have
\begin{equation}\label{NN:5}
\mE_{\xi^k_H}[\|\barM_k - M_k\|^2] \leq 2(\kappa_{2,G}+\kappa_{\nabla_{\bx}^2c}^2)\psi.
\end{equation}
Combining \eqref{NN:5} with \eqref{NN:4} and \eqref{NN:1}, we obtain
\begin{multline*}
\mE_{\xi^k_g,\xi^k_H}\sbr{\|\barM_k^T\barDelta\bx_k-M_k^T\Delta\bx_k\|^2 }  \\\leq \cbr{2(\kappa_{2,G} + \kappa_{\nabla_{\bx}^2c}^2)\psi+\kappa_M^2 }\frac{49\kappa_B^4}{\kappa_{1,G}^2\gamma_{RH}^2}\psi + 2(\kappa_{2,G} + \kappa_{\nabla_{\bx}^2c}^2)\psi\|\Delta\bx_k\|^2.
\end{multline*}
Plugging the above inequality into \eqref{NN:3},
\begin{multline}\label{NN:6}
\mE_{\xi^k_g,\xi^k_H}\sbr{\|\barDelta\blambda_k\|^2} \leq \|\Delta\blambda_k\|^2 + \frac{4(\kappa_{2,G} + \kappa_{\nabla_{\bx}^2c}^2)\psi}{\kappa_{1,G}^2}\|\Delta\bx_k\|^2 \\
+ \cbr{\frac{2\kappa_{2,G}}{\kappa_{1,G}^2} +\frac{98\kappa_B^4}{\kappa_{1,G}^4\gamma_{RH}^2}\rbr{2(\kappa_{2,G} + \kappa_{\nabla_{\bx}^2c}^2)\psi+\kappa_M^2}}\psi.
\end{multline}
Combining \eqref{NN:6} with \eqref{NN:2}, we can define
\begin{equation}\label{equ:U_0}
\Upsilon_0 \coloneqq 1 + \frac{4(\kappa_{2,G} + \kappa_{\nabla_{\bx}^2c}^2)\psi}{\kappa_{1,G}^2} \vee \frac{2\kappa_{2,G}}{\kappa_{1,G}^2} + \frac{49\kappa_B^4}{\kappa_{1,G}^2\gamma_{RH}^2} \cdot\cbr{1+ \frac{4(\kappa_{2,G} + \kappa_{\nabla_{\bx}^2c}^2)\psi+2\kappa_M^2}{\kappa_{1,G}^2}} 
\end{equation}
and have
\begin{align*}
\mE_{\xi^k_g,\xi^k_H}\sbr{\nbr{\begin{pmatrix}
\barDelta\bx_k\\
\barDelta\blambda_k
\end{pmatrix}}^2} \leq \Upsilon_0\rbr{ \nbr{\begin{pmatrix}
\Delta\bx_k\\
\Delta\blambda_k
\end{pmatrix}}^2+\psi}.
\end{align*}
This completes the proof.

\subsection{Proof of Theorem \ref{thm:4}}\label{pf:thm:4}

We require the following lemmas.

\begin{lemma}\label{lem:1}
Let $(\Delta\bx_k, \Delta\blambda_k)$ be solved by \eqref{equ:Newton:AL}. Then for any $\mu, \nu$,
\begin{equation*}
\begin{pmatrix}
\nabla_{\bx}\mL_{\mu, \nu}^k\\
\nabla_{\blambda}\mL_{\mu, \nu}^k
\end{pmatrix}^T \begin{pmatrix}
\Delta\bx_k\\
\Delta\blambda_k
\end{pmatrix}=(\Delta\bx_k)^T\nabla_{\bx}\mL_k - \mu\|c_k\|^2 + c_k^T\Delta\blambda_k - \nu\|G_k\nabla_{\bx}\mL_k\|^2.
\end{equation*}	
\end{lemma}

\begin{proof}
We have
\begin{align*}
\begin{pmatrix}
\nabla_{\bx}\mL_{\mu, \nu}^k\\
\nabla_{\blambda}\mL_{\mu, \nu}^k
\end{pmatrix}^T \begin{pmatrix}
\Delta\bx_k\\
\Delta\blambda_k
\end{pmatrix}
& \stackrel{\eqref{equ:derivative:AL}}{=} (\Delta\bx_k)^T\rbr{I + \nu M_kG_k}\nabla_{\bx}\mL_k + \mu (\Delta\bx_k)^TG_k^Tc_k + (\Delta\blambda_k)^Tc_k \\
& \quad + \nu (\Delta\blambda_k)^TG_kG_k^TG_k\nabla_{\bx}\mL_k\\
& \stackrel{\eqref{equ:Newton:AL}}{=} (\Delta\bx_k)^T\nabla_{\bx}\mL_k - \mu\|c_k\|^2 + c_k^T\Delta\blambda_k - \nu\|G_k\nabla_{\bx}\mL_k\|^2.
\end{align*}
This completes the proof.\hfill\qed
\end{proof}

\begin{lemma}
Under Assumption \ref{ass:ran:1}, we have
\begin{equation*}
\nbr{\begin{pmatrix}
\Delta\bx_k\\
\Delta\blambda_k
\end{pmatrix}}^2 \leq \frac{3\kappa_M^2}{\kappa_{1,G}^2}\nbr{\begin{pmatrix}
\Delta\bx_k\\
G_k\nabla_{\bx}\mL_k
\end{pmatrix}}^2.
\end{equation*}

\end{lemma}

\begin{proof}

By \eqref{equ:Newton:AL}, we know
\begin{align*}
\|\Delta\blambda_k\|^2  = & \|(G_kG_k^T)^{-1}(G_k\nabla_{\bx}\mL_k+M_k^T\Delta\bx_k)\|^2\stackrel{\eqref{Prop:1}}{\leq}\frac{2}{\kappa_{1,G}^2}(\|G_k\nabla_{\bx}\mL_k\|^2 + \kappa_M^2\|\Delta\bx_k\|^2)\\
\leq & \frac{2\kappa_M^2}{\kappa_{1,G}^2}(\|G_k\nabla_{\bx}\mL_k\|^2 + \|\Delta\bx_k\|^2).\hskip2cm \text{(since $\kappa_M\geq 1\geq \kappa_{1,G}$)}
\end{align*}
This completes the proof.\hfill\qed
\end{proof}

Now, we are ready for proving Theorem \ref{thm:4}. We note that
\begin{align*}
\|\nabla_{\bx}\mL_k\|^2 & =  \|(I - G_k^T(G_kG_k^T)^{-1}G_k)\nabla_{\bx}\mL_k\|^2 + \|G_k^T(G_kG_k^T)^{-1}G_k\nabla_{\bx}\mL_k\|^2\\
& \stackrel{\mathclap{\eqref{equ:Newton:AL}}}{=}  \|(I - G_k^T(G_kG_k^T)^{-1}G_k)B_k\Delta\bx_k\|^2 + \|G_k^T(G_kG_k^T)^{-1}G_k\nabla_{\bx}\mL_k\|^2\\
&\stackrel{\mathclap{\eqref{Prop:1}}}{\leq} \kappa_B^2\|\Delta\bx_k\|^2 + \frac{1}{\kappa_{1,G}}\|G_k\nabla_{\bx}\mL_k\|^2,\\
\|c_k\|^2 & \stackrel{\mathclap{\eqref{equ:Newton:AL}}}{=}\|G_k\Delta\bx_k\|^2 \stackrel{\eqref{Prop:1}}{\leq}\kappa_{2,G}\|\Delta\bx_k\|^2.
\end{align*}
Since $\kappa_{2,G}\wedge\kappa_B \geq 1\geq \kappa_{1,G}$,
\begin{equation*}
\|\nabla\mL_k\|^2\leq \frac{\kappa_B^2+\kappa_{2,G}}{\kappa_{1,G}}\nbr{\begin{pmatrix}
\Delta\bx_k\\
G_k\nabla_{\bx}\mL_k
\end{pmatrix}}^2.
\end{equation*}
Plugging the above inequality into \eqref{NN:7}, we have
\begin{equation}\label{NN:8}
\mE_{\xi^k_g,\xi^k_H}[\mL_{\mu, \nu}^{k+1}]\leq \mL_{\mu, \nu}^k - \frac{\alpha_k\tilde{\delta}\kappa_{1,G}}{2(\kappa_B^2+\kappa_{2,G})}\|\nabla\mL_k\|^2 + \frac{\Upsilon_0\kappa_{\mL_{\mu, \nu}}\psi}{2}\alpha_k^2.
\end{equation}
\noindent\textbf{Case 1: constant stepsize.} If $\alpha_k = \alpha$, we take full expectation of \eqref{NN:8}, sum over~$k = 0,\ldots,K$, and have
\begin{equation*}
\min_{\mX\times\Lambda}\mL_{\mu, \nu} - \mL_{\mu, \nu}^0 \leq - \frac{\tilde{\delta}\kappa_{1,G}}{2(\kappa_B^2+\kappa_{2,G})}\cdot \alpha \sum_{k=0}^{K}\mE[\|\nabla\mL_k\|^2] + \frac{\Upsilon_0\kappa_{\mL_{\mu, \nu}}\psi(K+1)}{2}\alpha^2.
\end{equation*}
Rearranging the above inequality leads to
\begin{equation*}
\frac{1}{K+1}\sum_{k=0}^{K}\mE[\|\nabla\mL_k\|^2] \leq \frac{2(\kappa_B^2+\kappa_{2,G})}{\tilde{\delta} \kappa_{1,G}}\cdot\frac{\mL_{\mu, \nu}^0 - \min_{\mX\times\Lambda}\mL_{\mu, \nu}}{(K+1)\alpha}+ \frac{\Upsilon_0(\kappa_B^2+\kappa_{2,G})\kappa_{\mL_{\mu, \nu}}\psi}{\tilde{\delta}\kappa_{1,G}}\alpha.
\end{equation*}
\noindent\textbf{Case 2: decaying stepsize.} Let us define two sequences
\begin{equation*}
s_k = \mL_{\mu, \nu}^k + \frac{\Upsilon_0\kappa_{\mL_{\mu, \nu}}\psi}{2}\sum_{i=k}^{\infty}\alpha_i^2,\quad\quad
e_k = \frac{\alpha_k\tdelta\kappa_{1,G}}{2(\kappa_B^2 +\kappa_{2,G})}\|\nabla\mL_k\|^2.
\end{equation*}
From \eqref{NN:8}, we know
\begin{align*}
\mE_{\xi^k_g,\xi^k_H}[s_{k+1}] & =  \mE_{\xi^k_g,\xi^k_H}[\mL_{\mu, \nu}^{k+1}] + \frac{\Upsilon_0\kappa_{\mL_{\mu, \nu}}\psi}{2}\sum_{i=k+1}^{\infty}\alpha_i^2\\
& \leq \mL_{\mu, \nu}^k - \frac{\alpha_k\tilde{\delta}\kappa_{1,G}}{2(\kappa_B^2+\kappa_{2,G})}\|\nabla\mL_k\|^2 + \frac{\Upsilon_0\kappa_{\mL_{\mu, \nu}}\psi}{2}\sum_{i=k}^{\infty}\alpha_i^2 = s_k - e_k.
\end{align*}
Since $e_k \geq 0$, $\{s_k -\min_{\mX\times\Lambda}\mL_{\mu, \nu}\}_k$ is a positive supermartingale. By \cite[Theorem 4.2.12]{Durrett2019Probability}, we have $s_k\rightarrow s$ almost surely for some random variable $s$ and $\mE[s]\leq s_0<\infty$. Therefore,
\begin{equation*}
\mE[\sum_{k=0}^{\infty}e_k] = \sum_{k=0}^{\infty}\mE\sbr{e_k} \leq \sum_{k=0}^{\infty}\mE[s_k] - \mE[s_{k+1}]<\infty,
\end{equation*}
where the first equality is due to Fubini-Tonelli theorem \cite[Theorem 1.7.2]{Durrett2019Probability}. The above display implies that $\sum_{k=0}^{\infty}e_k<\infty$ almost surely. Moreover, since $\sum_{k=0}^{\infty}\alpha_k = \infty$, we obtain $\sum_{k=0}^{\infty}e_k/\sum_{k=0}^{\infty}\alpha_k = 0$ and $\liminf_{k\rightarrow 0}\|\nabla\mL_k\| = 0$ almost surely. This completes the proof.

\section{Proofs of Section \ref{sec:4}}

\subsection{Proof of Lemma \ref{lem:cond:2}}\label{pf:lem:cond:2}

Let $P_{\xi_f^k}(\cdot) = P(\cdot \mid \bx_k, \blambda_k,\barDelta\bx_k, \barDelta\blambda_k)$ be the conditional probability over randomness in $\xi_f^k$. We have that
\begin{align}\label{pequ:decom}
\big|\barL_{\barmu_k, \nu}^k - &\mL_{\barmu_k, \nu}^k \big| \nonumber\\
& \stackrel{\mathclap{\eqref{N:12}}}{\leq}  \abr{\barf_k - f_k} + \frac{\nu}{2}\abr{\|G_k(\bnabla f_k + G_k^T\blambda_k) \|^2 - \|G_k(\nabla f_k + G_k^T\blambda_k) \|^2} \nonumber\\
& \leq  \abr{\barf_k - f_k}  + \frac{\nu}{2}\|G_k\|^2\nbr{\bnabla f_k - \nabla f_k}^2  + \nu \|G_k(\nabla f_k + G_k^T\blambda_k)\|\|G_k\|\|\bnabla f_k - \nabla f_k\| \nonumber\\
& \stackrel{\mathclap{\eqref{Prop:1}}}{\leq}  \abr{\barf_k - f_k} + \nu\kappa_{2,G}\rbr{\nbr{\bnabla f_k - \nabla f_k} + \kappa_{\nabla_{\bx}\mL}}\|\bnabla f_k - \nabla f_k\|.
\end{align}
Let us denote
\begin{equation*}
m_k = - \kappa_f\baralpha_k^2\begin{pmatrix}
\bnabla_{\bx}\mL_{\barmu_k, \nu}^k\\
\bnabla_{\blambda}\mL_{\barmu_k, \nu}^k
\end{pmatrix}^T\begin{pmatrix}
\barDelta\bx_k\\
\barDelta\blambda_k
\end{pmatrix},
\end{equation*}
then the above display implies that $|\barL_{\barmu_k, \nu}^k - \mL_{\barmu_k, \nu}^k |\leq m_k$ by requiring
\begin{equation*}
\abr{\barf_k - f_k} \leq \frac{m_k}{2} \quad \text{ and}\quad  \nbr{\bnabla f_k - \nabla f_k}\leq \frac{m_k}{4\nu\kappa_{2,G}\kappa_{\nabla_{\bx}\mL}}\wedge 1.
\end{equation*}
The above condition is further implied by requiring
\begin{equation*}
\abr{\barf_k - f_k} \vee \nbr{\bnabla f_k - \nabla f_k}\leq \frac{m_k\wedge 1}{4\nu\kappa_{2,G}\kappa_{\nabla_{\bx}\mL} \vee 2} .
\end{equation*}
Thus, we apply Bernstein inequality \cite[Theorem 6.1.1]{Tropp2015Introduction} and have that if
\begin{equation}\label{pequ:sample:f:1}
|\xi_f^k| \geq \frac{16\max\{\Omega_0, \Omega_1\}^2 (4\nu^2\kappa_{2,G}^2\kappa_{\nabla_{\bx}\mL}^2\vee 1) }{m_k^2\wedge 1 }\log\rbr{\frac{8d}{p_f}},
\end{equation}
then
\begin{equation*}
P_{\xi_f^k}(|\barL_{\barmu_k, \nu}^k - \mL_{\barmu_k, \nu}^k|> m_k)\leq p_f/2.
\end{equation*}
Under \eqref{pequ:sample:f:1}, we also have $P_{\xi_f^k}(|\barL_{\barmu_k, \nu}^{s_k} - \mL_{\barmu_k, \nu}^{s_k}|> m_k )\leq p_f/2$, which implies the condition \eqref{equ:ran:cond:3} holds. As for the condition \eqref{equ:ran:cond:4}, we apply \eqref{pequ:decom} to obtain
\begin{align*}
\mE_{\xi_f^k}\big[\big|\barL_{\barmu_k, \nu}^k - \mL_{\barmu_k, \nu}^k \big|^2\big]
& \leq  2\mE_{\xi_f^k}[|\barf_k - f_k|^2] + 2\nu^2\kappa_{2,G}^2\rbr{\Omega_1+\kappa_{\nabla_{\bx}\mL}}^2 \mE_{\xi_f^k}[\|\bnabla f_k - \nabla f_k\|^2]\\
& \leq  \frac{2\Omega_0^2 + 2\nu^2\kappa_{2,G}^2\Omega_1^2\rbr{\Omega_1+\kappa_{\nabla_{\bx}\mL}}^2}{|\xi_f^k|},
\end{align*}
where the last inequality applies the variance bound for the sample average, that is $\mE[|\barf_k - f_k|^2] = \mE[|f(\bx_k; \xi) - f_k|^2]/|\xi_f^k| \leq \Omega_0^2/|\xi_f^k|$ (and similar for $\mE[\|\bnabla f_k - \nabla f_k\|^2]$). Following the same calculation, we can show the same bound for $|\barL_{\barmu_k, \nu}^{s_k} - \mL_{\barmu_k, \nu}^{s_k} |^2$. Thus, the condition~\eqref{equ:ran:cond:4} holds if
\begin{equation}\label{pequ:sample:f:2}
|\xi_f^k| \geq \frac{2\Omega_0^2 + 2\nu^2\kappa_{2,G}^2\Omega_1^2\rbr{\Omega_1+\kappa_{\nabla_{\bx}\mL}}^2}{\barepsilon_k^2}.
\end{equation}
Combining \eqref{pequ:sample:f:1} and \eqref{pequ:sample:f:2} and taking maximum on the right hand side, the conditions \eqref{equ:ran:cond:3} and \eqref{equ:ran:cond:4} are satisfied simultaneously when
\begin{equation}\label{equ:N:2}
|\xi_f^k|  \geq \frac{C_f\log\rbr{\frac{8d}{p_f}}}{m_k^2\wedge \barepsilon_k^2\wedge 1},
\end{equation}
with $C_f = 16\max\{\Omega_0,\Omega_1\}^2\{1\vee 4\nu^2\kappa_{2,G}^2(\Omega_1\vee \kappa_{\nabla_{\bx}\mL})^2\}$. Finally, we note that $m_k\neq 0$. Otherwise, by \eqref{equ:ran:cond:2} we know $\barDelta\bx_k = \0$ and $G_k\bnabla_{\bx}\mL_k = \0$. Then, by \eqref{equ:ran:Newton} we have $\barDelta\blambda_k = \0$, which contradicts $\|(\barDelta\bx_k, \barDelta\blambda_k)\|>0$ in Assumption \ref{ass:A:1}. This completes the proof.

\subsection{Proof of Proposition \ref{prop:2}}\label{pf:N:14}

Note that
\begin{align*}
\bnabla_{\bx}\mL_k +& \nu\barM_k G_k\bnabla_{\bx}\mL_k + G_k^Tc_k \\
= & (I - G_k^T(G_kG_k^T)^{-1}G_k)\bnabla_{\bx}\mL_k + G_k^T(G_kG_k^T)^{-1}G_k\bnabla_{\bx}\mL_k + \nu\barM_k G_k\bnabla_{\bx}\mL^k + G_k^Tc_k\\
\stackrel{\eqref{equ:ran:Newton}}{=} & - (I - G_k^T(G_kG_k^T)^{-1}G_k)B_k\barDelta\bx_k + \rbr{G_k^T(G_kG_k^T)^{-1} +  \nu\barM_k}G_k\bnabla_{\bx}\mL_k - G_k^TG_k\barDelta\bx_k.
\end{align*}
Thus, we know
\begin{multline}\label{N:15}
\begin{pmatrix}
\bnabla_{\bx}\mL_k + \nu\barM_k G_k\bnabla_{\bx}\mL_k + G_k^Tc_k\\
\nu G_kG_k^TG_k\bnabla_{\bx}\mL_k
\end{pmatrix} \\
= \begin{pmatrix}
-(I - G_k^T(G_kG_k^T)^{-1}G_k)B_k - G_k^TG_k & G_k^T(G_kG_k^T)^{-1} +  \nu\barM_k\\
\0 & \nu G_kG_k^T
\end{pmatrix}\begin{pmatrix}
\barDelta\bx_k\\
G_k\bnabla_{\bx}\mL_k
\end{pmatrix}.
\end{multline}
Using \eqref{Prop:1} and \eqref{equ:bound:M}, we upper bound the spectrum norm of each block of the right hand side matrix, and define
\begin{equation*}
\Upsilon_1 = \kappa_B+\kappa_{2,G} + \frac{1}{\sqrt{\kappa_{1,G}}} + \nu(\kappa_{\barM} + \kappa_{2,G})
\end{equation*}
with $\kappa_{\barM}$ given in \eqref{equ:bound:M}, then the first inequality in \eqref{N:14} is satisfied. Furthermore, we know
\begin{equation*}
\begin{pmatrix}
\barDelta\bx_k\\
G_k\bnabla_{\bx}\mL_k
\end{pmatrix} \stackrel{\eqref{equ:ran:Newton}}{=} \begin{pmatrix}
\barDelta\bx_k\\
-\barM_k^T\barDelta\bx_k - G_kG_k^T\barDelta\blambda_k
\end{pmatrix} = \begin{pmatrix}
I & \0 \\
-\barM_k^T & -G_kG_k^T
\end{pmatrix}\begin{pmatrix}
\barDelta\bx_k\\
\barDelta\blambda_k
\end{pmatrix}.
\end{equation*}
Thus, we let 
\begin{equation*}
\Upsilon_2 = \frac{1}{1 + \kappa_{\barM} + \kappa_{2,G}}
\end{equation*}
and the second inequality in \eqref{N:14} is satisfied. Moreover,
\begin{multline*}
\begin{pmatrix}
\barDelta\bx_k\\
\barDelta\blambda_k
\end{pmatrix} \stackrel{\eqref{equ:ran:Newton}}{=} \begin{pmatrix}
\barDelta\bx_k\\
-(G_kG_k^T)^{-1}\barM_k^T\barDelta\bx_k - (G_kG_k^T)^{-1}G_k\bnabla_{\bx}\mL_k
\end{pmatrix} \\
= \begin{pmatrix}
I & \0 \\
-(G_kG_k^T)^{-1}\barM_k^T & - (G_kG_k^T)^{-1}
\end{pmatrix}\begin{pmatrix}
\barDelta\bx_k\\
G_k\bnabla_{\bx}\mL_k
\end{pmatrix}.
\end{multline*}
Thus, we let 
\begin{equation*}
\Upsilon_3 = 1 + \frac{\kappa_{\barM}+1}{\kappa_{1,G}}
\end{equation*}
and the third inequality in \eqref{N:14} is satisfied. Finally, 
\begin{align*}
\begin{pmatrix}
\bnabla_{\bx}\mL_{\barmu_{\barK}, \nu}^k\\
\bnabla_{\blambda}\mL_{\barmu_{\barK}, \nu}^k
\end{pmatrix} &\stackrel{\eqref{N:13}}{=} \begin{pmatrix}
\rbr{I+\nu\barM_kG_k}\bnabla_{\bx}\mL_k + \barmu_{\barK}G_k^Tc_k\\
c_k + \nu G_kG_k^TG_k\bnabla_{\bx}\mL_k
\end{pmatrix}\\
&\hskip -0.8cm\stackrel{\eqref{N:15}}{=}\begin{pmatrix}
-(I - G_k^T(G_kG_k^T)^{-1}G_k)B_k - \barmu_{\barK}G_k^TG_k & G_k^T(G_kG_k^T)^{-1} +  \nu\barM_k\\
-G_k & \nu G_kG_k^T
\end{pmatrix}\begin{pmatrix}
\barDelta\bx_k\\
G_k\bnabla_{\bx}\mL_k
\end{pmatrix}.
\end{align*}
By Lemma \ref{lem:9}, we know $\barmu_{\barK} \leq \tmu$ with deterministic $\tmu$. Thus, we let
\begin{equation*}
\Upsilon_4 = \kappa_B + (\tmu+1)\kappa_{2,G} +\frac{1}{\sqrt{\kappa_{1,G}}} + \nu(\kappa_{\barM}+\kappa_{2,G})
\end{equation*}
and the fourth inequality in \eqref{N:14} is satisfied. This completes the proof.

\subsection{Proof of Lemma \ref{lem:5}}\label{pf:lem:5}

We consider the following three cases.

\noindent\textbf{Case 1: reliable step.} For the reliable step, we apply Lemma \ref{lem:A:2} and \eqref{pequ:8} is changed to
\begin{align}\label{NNN}
\mL_{\barmu_{\barK}, \nu}^{k+1} - \mL_{\barmu_{\barK}, \nu}^k
& \leq  \frac{\baralpha_k\beta}{2}\begin{pmatrix}
\bnabla_{\bx}\mL_{\barmu_{\barK}, \nu}^k\\
\bnabla_{\blambda}\mL_{\barmu_{\barK}, \nu}^k
\end{pmatrix}^T\begin{pmatrix}
\barDelta\bx_k\\
\barDelta\blambda_k
\end{pmatrix}\stackrel{\eqref{equ:sufficient:dec}}{\leq} \frac{\baralpha_k\beta}{4}\begin{pmatrix}
\bnabla_{\bx}\mL_{\barmu_{\barK}, \nu}^k\\
\bnabla_{\blambda}\mL_{\barmu_{\barK}, \nu}^k
\end{pmatrix}^T\begin{pmatrix}
\barDelta\bx_k\\
\barDelta\blambda_k
\end{pmatrix}  - \frac{\barepsilon_k}{4} \nonumber\\
& \stackrel{\eqref{equ:ran:cond:2}}{\leq}  -\frac{\baralpha_k\beta(\gamma_{RH}\wedge\nu)}{8}\nbr{\begin{pmatrix}
\barDelta\bx_k\\
G_k\bnabla_{\bx}\mL_{k}
\end{pmatrix}}^2  - \frac{\barepsilon_k}{4}.
\end{align}
By Line 13 of Algorithm \ref{alg:ASto:SQP}, $\barepsilon_{k+1} - \barepsilon_k = (\rho-1)\barepsilon_k$, while \eqref{pequ:9} is the same. Thus, by the condition on $\omega$ in \eqref{equ:cond:omega:1} we obtain
\begin{equation}\label{pequ:14}
\Phi_{\barmu_{\barK}, \nu, \omega}^{k+1} - \Phi_{\barmu_{\barK}, \nu, \omega}^k \leq -\frac{\beta(\gamma_{RH}\wedge\nu)\omega}{32}\baralpha_k\nbr{\begin{pmatrix}
\barDelta\bx_k\\
G_k\bnabla_{\bx}\mL_k
\end{pmatrix}}^2 - \frac{\omega\barepsilon_k}{8} + \rho(1-\omega)\baralpha_k\nbr{\begin{pmatrix}
\nabla_{\bx}\mL_{\barmu_{\barK}, \nu}^k\\
\nabla_{\blambda}\mL_{\barmu_{\barK}, \nu}^k
\end{pmatrix}}^2.
\end{equation}
\noindent\textbf{Case 2: unreliable step.} For the unreliable step, \eqref{NNN} is changed to
\begin{equation*}
\mL_{\barmu_{\barK}, \nu}^{k+1} - \mL_{\barmu_{\barK}, \nu}^k  \leq  \frac{\baralpha_k\beta}{2}\begin{pmatrix}
\bnabla_{\bx}\mL_{\barmu_{\barK}, \nu}^k\\
\bnabla_{\blambda}\mL_{\barmu_{\barK}, \nu}^k
\end{pmatrix}^T\begin{pmatrix}
\barDelta\bx_k\\
\barDelta\blambda_k
\end{pmatrix} \stackrel{\eqref{equ:ran:cond:2}}{\leq}  -\frac{\beta(\gamma_{RH}\wedge\nu)}{4}\baralpha_k\nbr{\begin{pmatrix}
\barDelta\bx_k\\
G_k\bnabla_{\bx}\mL_k
\end{pmatrix}}^2.
\end{equation*}
By Line 15 of Algorithm \ref{alg:ASto:SQP}, $\barepsilon_{k+1} - \barepsilon_k = -(1-1/\rho)\barepsilon_k$ and \eqref{pequ:9} is the same. Thus, under the condition on $\omega$ in \eqref{equ:cond:omega:1},
\begin{multline}\label{pequ:15}
\Phi_{\barmu_{\barK}, \nu, \omega}^{k+1} - \Phi_{\barmu_{\barK}, \nu, \omega}^k 
\leq -\frac{\beta(\gamma_{RH}\wedge\nu)\omega}{32}\baralpha_k\nbr{\begin{pmatrix}
	\barDelta\bx_k\\
	G_k\bnabla_{\bx}\mL_{k}
	\end{pmatrix}}^2\\ -\frac{1}{2}(1-\omega)\rbr{1-\frac{1}{\rho}}\barepsilon_k + \rho(1-\omega)\baralpha_k\nbr{\begin{pmatrix}
	\nabla_{\bx}\mL_{\barmu_{\barK}, \nu}^k\\
	\nabla_{\blambda}\mL_{\barmu_{\barK}, \nu}^k
	\end{pmatrix}}^2.
\end{multline}
\noindent\textbf{Case 3: unsuccessful step.} The result in \eqref{pequ:13} holds.

Combining \eqref{pequ:14}, \eqref{pequ:15}, \eqref{pequ:13}, we have
\begin{equation*}
\Phi_{\barmu_{\barK}, \nu, \omega}^{k+1} - \Phi_{\barmu_{\barK}, \nu, \omega}^k \leq \rho(1-\omega)\baralpha_k\nbr{\begin{pmatrix}
	\nabla_{\bx}\mL_{\barmu_{\barK}, \nu}^k\\
	\nabla_{\blambda}\mL_{\barmu_{\barK}, \nu}^k
	\end{pmatrix}}^2,
\end{equation*}
which completes the proof.

\subsection{Proof of Lemma \ref{lem:7}}\label{pf:lem:7}

We consider the following three cases.

\noindent\textbf{Case 1: reliable step.} We have
\begin{align}\label{pequ:16}
\mL_{\barmu_{\barK}, \nu}^{k+1} - \mL_{\barmu_{\barK}, \nu}^k & \leq  |\mL_{\barmu_{\barK}, \nu}^{k+1} - \barL_{\barmu_{\barK}, \nu}^{s_k}| + \barL_{\barmu_{\barK}, \nu}^{s_k} - \barL_{\barmu_{\barK}, \nu}^k + |\barL_{\barmu_{\barK}, \nu}^k - \mL_{\barmu_{\barK}, \nu}^k| \nonumber\\
&\hskip-2.3cm \stackrel{\eqref{equ:ran:Armijo}}{\leq}  \baralpha_k\beta\begin{pmatrix}
\bnabla_{\bx}\mL_{\barmu_{\barK}, \nu}^k\\
\bnabla_{\blambda}\mL_{\barmu_{\barK}, \nu}^k
\end{pmatrix}^T\begin{pmatrix}
\barDelta\bx_k\\
\barDelta\blambda_k
\end{pmatrix} + |\barL_{\barmu_{\barK}, \nu}^{s_k} - \mL_{\barmu_{\barK}, \nu}^{s_k}| + |\barL_{\barmu_{\barK}, \nu}^k - \mL_{\barmu_{\barK}, \nu}^k| \nonumber\\
&\hskip-2.3cm \stackrel{\eqref{equ:sufficient:dec}}{\leq} \frac{\baralpha_k\beta}{2}\begin{pmatrix}
\bnabla_{\bx}\mL_{\barmu_{\barK}, \nu}^k\\
\bnabla_{\blambda}\mL_{\barmu_{\barK}, \nu}^k
\end{pmatrix}^T\begin{pmatrix}
\barDelta\bx_k\\
\barDelta\blambda_k
\end{pmatrix} -\frac{\barepsilon_k}{2} + |\barL_{\barmu_{\barK}, \nu}^{s_k} - \mL_{\barmu_{\barK}, \nu}^{s_k}| + |\barL_{\barmu_{\barK}, \nu}^k - \mL_{\barmu_{\barK}, \nu}^k| \nonumber\\
&\hskip -2.3cm \stackrel{\eqref{equ:ran:cond:2}}{\leq}  -\frac{\beta\baralpha_k(\gamma_{RH}\wedge\nu)}{4}\nbr{\begin{pmatrix}
\barDelta\bx_k\\
G_k\bnabla_{\bx}\mL_k
\end{pmatrix}}^2 - \frac{\barepsilon_k}{2} + |\barL_{\barmu_{\barK}, \nu}^{s_k} - \mL_{\barmu_{\barK}, \nu}^{s_k}| + |\barL_{\barmu_{\barK}, \nu}^k - \mL_{\barmu_{\barK}, \nu}^k|.
\end{align}
By Line 13 of Algorithm \ref{alg:ASto:SQP}, $\barepsilon_{k+1} - \barepsilon_k = (\rho-1)\barepsilon_k$, and \eqref{pequ:9} holds as well. By \eqref{equ:cond:omega:1}, we have
\begin{align}\label{pequ:17}
\Phi_{\barmu_{\barK}, \nu, \omega}^{k+1} - \Phi_{\barmu_{\barK}, \nu, \omega}^k & \leq  -\frac{\beta(\gamma_{RH}\wedge\nu)\omega}{32}\baralpha_k\nbr{\begin{pmatrix}
\barDelta\bx_k\\
G_k\bnabla_{\bx}\mL_k
\end{pmatrix}}^2 - \frac{\omega\barepsilon_k}{8} + \rho(1-\omega)\baralpha_k\nbr{\begin{pmatrix}
\nabla_{\bx}\mL_{\barmu_{\barK}, \nu}^k\\
\nabla_{\blambda}\mL_{\barmu_{\barK}, \nu}^k
\end{pmatrix}}^2 \nonumber\\
&\quad +\omega( |\barL_{\barmu_{\barK}, \nu}^{s_k} - \mL_{\barmu_{\barK}, \nu}^{s_k}| + |\barL_{\barmu_{\barK}, \nu}^k - \mL_{\barmu_{\barK}, \nu}^k| ).
\end{align}
\noindent\textbf{Case 2: unreliable step.} The inequality in \eqref{pequ:16} is changed to
\begin{align*}
\mL_{\barmu_{\barK}, \nu}^{k+1} - \mL_{\barmu_{\barK}, \nu}^k & \leq  |\mL_{\barmu_{\barK}, \nu}^{k+1} - \barL_{\barmu_{\barK}, \nu}^{s_k}| + \barL_{\barmu_{\barK}, \nu}^{s_k} - \barL_{\barmu_{\barK}, \nu}^k + |\barL_{\barmu_{\barK}, \nu}^k - \mL_{\barmu_{\barK}, \nu}^k| \nonumber\\
&\hskip-4pt \stackrel{\eqref{equ:ran:Armijo}}{\leq}  \baralpha_k\beta\begin{pmatrix}
\bnabla_{\bx}\mL_{\barmu_{\barK}, \nu}^k\\
\bnabla_{\blambda}\mL_{\barmu_{\barK}, \nu}^k
\end{pmatrix}^T\begin{pmatrix}
\barDelta\bx_k\\
\barDelta\blambda_k
\end{pmatrix} + |\barL_{\barmu_{\barK}, \nu}^{s_k} - \mL_{\barmu_{\barK}, \nu}^{s_k}| + |\barL_{\barmu_{\barK}, \nu}^k - \mL_{\barmu_{\barK}, \nu}^k| \nonumber\\
&\hskip-4pt \stackrel{\eqref{equ:ran:cond:2}}{\leq}  -\frac{\beta\baralpha_k(\gamma_{RH}\wedge\nu)}{2}\nbr{\begin{pmatrix}
	\barDelta\bx_k\\
	G_k\bnabla_{\bx}\mL_k
	\end{pmatrix}}^2 + |\barL_{\barmu_{\barK}, \nu}^{s_k}-\mL_{\barmu_{\barK}, \nu}^{s_k}| + |\barL_{\barmu_{\barK}, \nu}^k - \mL_{\barmu_{\barK}, \nu}^k|.
\end{align*}
We further have $\barepsilon_{k+1} - \barepsilon_k = -(1-1/\rho)\barepsilon_k$ and \eqref{pequ:9} holds as well. Using \eqref{equ:cond:omega:1}, we get
\begin{multline}\label{pequ:18}
\Phi_{\barmu_{\barK}, \nu, \omega}^{k+1} - \Phi_{\barmu_{\barK}, \nu, \omega}^k  \leq  -\frac{\beta(\gamma_{RH}\wedge\nu)\omega}{32}\baralpha_k\nbr{\begin{pmatrix}
\barDelta\bx_k\\
G_k\bnabla_{\bx}\mL_k
\end{pmatrix}}^2 - \frac{1}{2}(1-\omega)\rbr{1-\frac{1}{\rho}}\barepsilon_k\\ + \rho(1-\omega)\baralpha_k\nbr{\begin{pmatrix}
\nabla_{\bx}\mL_{\barmu_{\barK}, \nu}^k\\
\nabla_{\blambda}\mL_{\barmu_{\barK}, \nu}^k
\end{pmatrix}}^2 
+\omega( |\barL_{\barmu_{\barK}, \nu}^{s_k} - \mL_{\barmu_{\barK}, \nu}^{s_k}| + |\barL_{\barmu_{\barK}, \nu}^k - \mL_{\barmu_{\barK}, \nu}^k| ).
\end{multline}
\noindent\textbf{Case 3: unsuccessful step.} The result in \eqref{pequ:13} holds.

Combining \eqref{pequ:17}, \eqref{pequ:18}, \eqref{pequ:13}, we obtain
\begin{equation*}
\Phi_{\barmu_{\barK}, \nu, \omega}^{k+1} - \Phi_{\barmu_{\barK}, \nu, \omega}^k \leq  \rho(1-\omega)\baralpha_k\nbr{\begin{pmatrix}
	\nabla_{\bx}\mL_{\barmu_{\barK}, \nu}^k\\
	\nabla_{\blambda}\mL_{\barmu_{\barK}, \nu}^k
	\end{pmatrix}}^2 + \omega( |\barL_{\barmu_{\barK}, \nu}^{s_k} - \mL_{\barmu_{\barK}, \nu}^{s_k}| + |\barL_{\barmu_{\barK}, \nu}^k - \mL_{\barmu_{\barK}, \nu}^k| ),
\end{equation*}
which completes the proof.

\section{Additional Simulation Results}\label{sec:appen:5}

We list the detailed results for each selected CUTEst problem. For each problem and~each method, we have five levels of noise. For all four methods, we report the smallest result among different setups. In particular, for NonAdapSQP and $\ell_1$ SQP, we take the minimum over all six setups of stepsize sequences (four constant sequences and two decaying sequences). For AdapSQP and $\ell_1$ AdapSQP, we take the minimum over four setups of constants $C_{grad}, C_f = \{1,5,10,50\}$. The result here is the KKT residual of the last iterate. We average over convergent runs across five independent runs. The convergence results are summarized in Tables \ref{tab:1}, \ref{tab:2}, and \ref{tab:3}.

In tables, the column of logR shows $\log(\|\nabla \mL_k\|)$ and the column of logStd shows the log of standard deviation. For each problem, the first line is the result of AdapSQP; the second line is the result of $\ell_1$ AdapSQP; the third line is the result of $\ell_1$ SQP; and the fourth line is the result of NonAdapSQP. The entry `$\diagup$' means that the algorithm does not converge (within the given budget).

\begin{table}[!htp]
\centering
\caption{KKT results summary.}\label{tab:1}
\scalebox{0.8}{
\begin{tabular}{ c|c|c|c|c|c|c|c|c|c|c } 
\hline
\multirow{2}{*}{Problem}  & \multicolumn{2}{c|}{$\sigma^2 = 10^{-8}$} &  \multicolumn{2}{c|}{$\sigma^2 = 10^{-4}$} & \multicolumn{2}{c|}{$\sigma^2 = 10^{-2}$} & \multicolumn{2}{c|}{$\sigma^2 = 10^{-1}$} & \multicolumn{2}{c}{$\sigma^2 = 1$} \\	
 & \multicolumn{1}{c}{logR} & logStd & \multicolumn{1}{c}{logR} & logStd & \multicolumn{1}{c}{logR} & logStd & \multicolumn{1}{c}{logR} & logStd & \multicolumn{1}{c}{logR} & logStd \\
\hline
\multirow{4}{*}{BT9} 
& -9.94 & -10.37 & -9.47 & -10.66 & -9.74 & -10.47 
& -9.55 & -10.26 & -8.36 & -9.85\\
& -9.40 & -11.08 & -9.10 & -10.91 & -8.39 & -9.84
& -7.91 & -9.17 & -6.85 & -7.97\\
& -9.37 & -11.21 & -9.35 & -12.37 & -0.31 & -3.09
& 0.21 & -2.24 & 0.83 &  -1.18\\
& -9.47 & -11.21 & -9.36 & -11.17 & -5.96 & -7.28
& $\diagup$ & $\diagup$ &  $\diagup$ & $\diagup$\\
\hline

\multirow{4}{*}{BYRDSPHR}
& -10.12 & -11.73 & -9.64 & -10.90 & -9.65 & -10.36
& 1.26 & 2.06 & 2.64 & 2.06\\
& -9.36  & -15.49  & -8.80 & -10.77 & -8.62 & -10.14
& -7.92 & -9.13 & -7.50 & -8.51\\
& -9.23  & -14.53  & -4.07  &  -4.92 & -1.63 & -2.52 
& -0.61 & -1.50 & 0.56 & -2.07\\
& -9.77  &  -10.40 & -9.50 & -11.06 & -6.96 & -6.50
& $\diagup$ & $\diagup$ & $\diagup$ & $\diagup$\\
\hline

\multirow{4}{*}{BT10}
& -10.02 & -10.38 & -9.91   & -10.69 & -9.86 & -10.31
& -9.81 & -10.83 & -9.52 & -11.73\\
& -9.30   & -14.55 & -9.18   & -11.11 & -7.97 & -9.29
& -6.26 & -7.57 & -5.06 & -6.00\\
& -9.49   & -10.98 & -9.44   &  -Inf & -0.95  & -4.48
& -0.19 & -1.97 & -0.12 & -0.94\\
& -9.55   & -10.48 & -10.06 & -10.62 & -10.62 & -Inf
& -9.65 & -10.41 & -10.03 & -10.86\\
\hline

\multirow{4}{*}{HS39} 
& -9.92 & -10.22 & -9.50 & -11.56 & -9.58 & -10.41
& -9.49 & -10.82 & -8.23 & -8.69\\
& -9.82 & -16.00 & -8.67 & -10.30 & -9.03 & -11.71
& -8.37 & -9.83 & -7.07 & -8.46\\
& -9.58 & -10.97 & -9.43 & -11.96 & -0.30 & -2.79
& 0.04 & -2.43 & 0.80 & -1.80\\
& -9.47 & -11.19 & -9.32 & -11.53 & -6.23 & -6.64
& -4.67 & -Inf & $\diagup$ & $\diagup$\\
\hline

\multirow{4}{*}{MSS1} 
& $\diagup$ & $\diagup$ & $\diagup$ & $\diagup$ & $\diagup$ & $\diagup$  & $\diagup$ & $\diagup$ & $\diagup$ & $\diagup$\\
& $\diagup$ & $\diagup$ & $\diagup$ & $\diagup$ & $\diagup$ & $\diagup$  & $\diagup$ & $\diagup$ & $\diagup$ & $\diagup$\\
& $\diagup$ & $\diagup$ & $\diagup$ & $\diagup$ & $\diagup$ & $\diagup$  & $\diagup$ & $\diagup$ & $\diagup$ & $\diagup$\\
& $\diagup$ & $\diagup$ & $\diagup$ & $\diagup$ & $\diagup$ & $\diagup$  & $\diagup$ & $\diagup$ & $\diagup$ & $\diagup$\\
\hline

\multirow{4}{*}{MARATOS} 
& -10.51 & -10.95 & -9.80 & -10.47 & -9.70 & -10.27
& -9.64 & -10.26 & -9.73 & -10.18\\
& -9.78  & -14.27 & -8.88 & -10.74 & -8.50 & -10.21
& -8.27 & -9.81 & -7.17 & -8.12\\
& -9.77   & -10.68 & -9.68 & -10.77 & -9.56 & -10.62
& -0.21 & -0.79 & -0.01 & -4.10\\
& -9.60  & -10.31 & -10.11 & -11.17 & -7.32 & -6.66
& -5.64 & -5.60 & -3.60 & -3.63\\
\hline

\multirow{4}{*}{ORTHREGB} 
& -9.71 & -10.49 & -9.04 & -9.98 & -7.19 & -6.76 
& -7.72 & -8.59 & -3.86 & -3.13\\
& -9.69 & -10.64 & -9.17 & -11.75 & -8.34 & -9.95 
& -6.92 & -8.21 & -6.14 & -7.92\\
& -9.38 & -11.57 & -4.04 &  -Inf & 1.06 & -0.97 
& 2.83 & 1.73 & 4.67 & 1.31\\
& -9.33 & -13.29 & -7.95 & -9.65 & $\diagup$ & $\diagup$ & $\diagup$ & $\diagup$ & $\diagup$ & $\diagup$\\
\hline

\multirow{4}{*}{HS6}
& -9.77 & -12.02 & -9.81 & -10.45 & -9.84 & -10.89
& -9.72 & -10.36 & -8.97 & -10.06\\
& -9.30 & -11.85 & -8.85 & -10.78 & -8.17 & -9.26
& -7.39 & -8.73 & -6.75 & -8.00\\
& -8.60 & -9.65 & -9.70 &  -10.35 & -9.42 & -Inf
&  -3.07 & -3.41 & -2.38 & -Inf\\
& -9.51 & -10.98 & -9.85 & -10.92 & -5.16 & -5.90
& -4.54 & -5.36 & -4.13 & -4.94\\
\hline

\multirow{4}{*}{BT8}
& -10.25 & -10.42 & -9.57 & -11.39 & -9.42 & -10.80
& -9.30 & -11.55 & -8.88 & -9.51\\
& -9.76 & -10.51 & -8.46 & -9.88 & -9.00 & -11.43
& -8.38 & -9.94 & -7.09 & -8.20\\
& -9.65 & -10.52 & -9.66 & -Inf & $\diagup$ & $\diagup$ & 0.15 & -3.07 & 0.64 & -2.10\\
& -9.46 & -10.52 & -9.57 & -10.57 & -6.81 & -7.02 
& $\diagup$ & $\diagup$ & $\diagup$ & $\diagup$\\
\hline

\multirow{4}{*}{MSS2} 
& $\diagup$ & $\diagup$ & $\diagup$ & $\diagup$ & $\diagup$ & $\diagup$& $\diagup$ & $\diagup$ & $\diagup$ & $\diagup$\\
& $\diagup$ & $\diagup$ & $\diagup$ & $\diagup$ & $\diagup$ & $\diagup$ & 0 & $\diagup$ & $\diagup$ & $\diagup$\\
& $\diagup$ & $\diagup$ & $\diagup$ & $\diagup$ & $\diagup$ & $\diagup$ & $\diagup$ & $\diagup$ & $\diagup$ & $\diagup$\\
& $\diagup$ & $\diagup$ & $\diagup$ & $\diagup$ & $\diagup$ & $\diagup$ & $\diagup$ & $\diagup$ & $\diagup$ & $\diagup$\\
\hline

\multirow{4}{*}{BT1}
& -10.18 & -10.62 & -9.82 & -10.52 & -10.13 & -11.03
& -9.51 & -10.61 & -9.95 & -10.56\\
& -10.23 & -15.53 & -8.87 & -10.36 & -8.44 & -10.11
& -8.45 & -9.97 & -7.60 & -8.63\\
& -9.52  &  -10.49 & -9.83 &  -10.63 & -9.78 & -10.47
& -9.26 & -Inf & $\diagup$ & $\diagup$\\
& -10.11 & -10.89 & -10.07 & -10.51 & -6.88 & -6.15
& -4.98 & -7.96 & -3.75 & -4.02\\
\hline

\multirow{4}{*}{GENHS28}
& -9.76 & -10.33 & -9.40 & -10.78 & -9.20 & -11.83 & -8.44 & -9.97 & -7.32 & -9.95\\
& -9.31 & -11.53 & -8.76 & -10.41 & -8.42 & -9.82 & -7.66 & -9.29 & -6.20 & -7.83\\
& -9.37 & -11.26 & $\diagup$ & $\diagup$ & $\diagup$ & $\diagup$ & 0.87 & -1.04 & 1.61 & -0.06\\
& -9.34 & -11.61 &  -7.68 & -9.15 & $\diagup$ & $\diagup$ & $\diagup$ & $\diagup$ & $\diagup$ & $\diagup$\\
\hline

\multirow{4}{*}{BT5}
& -10.22 & -11.08 & -9.84 & -10.72 & -9.63 & -11.17 
& -9.10 & -10.40 & -8.07 & -9.44\\
& -9.97 & -15.38 & -8.53 & -9.95 & -8.97 & -11.38
& -7.76 & -8.99 & -7.27 & -8.76\\
&  -9.43 & -11.17 & -9.56 & -10.60 & $\diagup$ & $\diagup$ & $\diagup$ & $\diagup$ & 1.69 & 0.03\\
& -9.51 & -10.76 & -9.45 & -10.70 & -9.39 & -11.00
& -4.95 & -5.06 & -4.38 & -5.11\\
\hline

\multirow{4}{*}{HS61} 
& $\diagup$ & $\diagup$ & $\diagup$ & $\diagup$ & $\diagup$ & $\diagup$ & $\diagup$ & $\diagup$& 0 & $\diagup$\\
& $\diagup$ & $\diagup$ & $\diagup$ & $\diagup$ & $\diagup$ & $\diagup$ & $\diagup$ & $\diagup$ & $\diagup$ & $\diagup$\\
& $\diagup$ & $\diagup$ & $\diagup$ & $\diagup$ & $\diagup$ & $\diagup$ & $\diagup$ & $\diagup$ & $\diagup$ & $\diagup$\\
& $\diagup$ & $\diagup$ & $\diagup$ & $\diagup$ & $\diagup$ & $\diagup$ & $\diagup$ & $\diagup$ & $\diagup$ & $\diagup$\\
\hline

\multirow{4}{*}{BT4}
& -9.62 & -10.29 & -9.39 & -10.05 & -9.51 & -10.61
& -9.46 & -10.79 & -8.32 & -9.01\\
& $\diagup$ & $\diagup$ & $\diagup$ & $\diagup$
& $\diagup$ & $\diagup$ & $\diagup$ & $\diagup$ & $\diagup$ & $\diagup$\\
& -9.51 &  -10.55 & -9.47 & -10.82 & $\diagup$ & $\diagup$ & $\diagup$ & $\diagup$ & $\diagup$ & $\diagup$\\
& -9.41 &  -11.23 & -9.48 & -10.63 & $\diagup$ & $\diagup$ & $\diagup$ & $\diagup$ & $\diagup$ & $\diagup$\\
\hline

\multirow{4}{*}{BT2}
& -9.65 & -10.23 & -8.39 & -8.77 & -9.50 & -10.34
& -9.12 & -10.23 & -8.49 & -9.29\\
& -9.75 & -10.88 & -9.07 & -11.60 & -8.49 & -10.20
& -7.80 & -9.12 & -6.73 & -8.65\\
& -9.41 & -11.98 & -9.57 & -10.78 & -1.75 & -3.46 
& 8.61 & 3.36 & 8.58 & 5.04\\
& -9.34 & -11.62 & -9.57 & -Inf & $\diagup$ & $\diagup$ & $\diagup$ & $\diagup$ & $\diagup$ & $\diagup$\\
\hline

\multicolumn{11}{l}{\tiny $\text{logR} = \log(\|\nabla\mL_k\|)$, $\text{logStd}$ is the standard deviation. The first line: AdapSQP; the second line: $\ell_1$ AdapSQP;}\\
\multicolumn{11}{l}{\tiny the third line: $\ell_1$ SQP; the fourth line: NonAdapSQP.}

\end{tabular}}

\end{table}

\begin{table}[!htp]
\centering
\caption{KKT results summary.}\label{tab:2}
\scalebox{0.8}{
\begin{tabular}{ c|c|c|c|c|c|c|c|c|c|c } 
\hline
\multirow{2}{*}{Problem}  & \multicolumn{2}{c|}{$\sigma^2 = 10^{-8}$} &  \multicolumn{2}{c|}{$\sigma^2 = 10^{-4}$} & \multicolumn{2}{c|}{$\sigma^2 = 10^{-2}$} & \multicolumn{2}{c|}{$\sigma^2 = 10^{-1}$} & \multicolumn{2}{c}{$\sigma^2 = 1$} \\	
& \multicolumn{1}{c}{logR} & logStd & \multicolumn{1}{c}{logR} & logStd & \multicolumn{1}{c}{logR} & logStd & \multicolumn{1}{c}{logR} & logStd & \multicolumn{1}{c}{logR} & logStd \\
\hline
\multirow{4}{*}{S316-322} 
& $\diagup$ & $\diagup$ & $\diagup$ & $\diagup$ & $\diagup$ & $\diagup$ & $\diagup$ & $\diagup$ & $\diagup$ & $\diagup$ \\
& $\diagup$ & $\diagup$ & $\diagup$ & $\diagup$ & $\diagup$ & $\diagup$ & $\diagup$ & $\diagup$  & $\diagup$ & $\diagup$ \\
& $\diagup$ & $\diagup$ & $\diagup$ & $\diagup$ & $\diagup$ & $\diagup$ & $\diagup$ & $\diagup$ & $\diagup$ & $\diagup$ \\
& $\diagup$ & $\diagup$ & $\diagup$ & $\diagup$ & $\diagup$ & $\diagup$ & $\diagup$ & $\diagup$ & $\diagup$ & $\diagup$ \\
\hline

\multirow{4}{*}{FLT} 
& $\diagup$ & $\diagup$ & $\diagup$ & $\diagup$ & $\diagup$ & $\diagup$ & $\diagup$ & $\diagup$  & $\diagup$ & $\diagup$ \\
& $\diagup$ & $\diagup$ & $\diagup$ & $\diagup$ & $\diagup$ & $\diagup$ & $\diagup$ & $\diagup$  & $\diagup$ & $\diagup$  \\
& $\diagup$ & $\diagup$ & $\diagup$ & $\diagup$ & $\diagup$ & $\diagup$ & $\diagup$ & $\diagup$  & $\diagup$ & $\diagup$ \\
& $\diagup$ & $\diagup$ & $\diagup$ & $\diagup$ & $\diagup$ & $\diagup$ & $\diagup$ & $\diagup$  & $\diagup$ & $\diagup$ \\
\hline

\multirow{4}{*}{HS51}
& -9.59 & -10.96 & -9.83 &  -10.58 & -9.46 & -11.15
& -8.98 & -10.57 & -7.66 &  -8.08\\
& -9.32 & -10.94 & -8.20 & -9.46 & -8.74 & -10.44
& -7.34 & -8.44 & -6.39 & -7.79\\
& -6.37&  -Inf & -4.49 & -Inf & $\diagup$ & $\diagup$ & -0.32 & -Inf & $\diagup$ & $\diagup$\\
& -9.37 & -12.03 & -9.40 & -11.19 & -6.22 & -6.93
& $\diagup$ & $\diagup$ & $\diagup$ & $\diagup$\\
\hline

\multirow{4}{*}{BT12}
& -10.02 & -10.82 & -9.62 &  -10.57 & -9.37 & -11.37
& -8.97 & -10.49 & -7.80 & -9.72\\
& -9.38 & -11.10 & -8.96 & -10.98 & -8.83 &  -10.54
& -7.73 & -8.78 & -6.77 & -8.21\\
& -9.34 & -11.50 & $\diagup$ & $\diagup$ & $\diagup$ & $\diagup$ & 0.13 &  -Inf & $\diagup$ & $\diagup$\\
& -9.35 & -11.53 & -7.68 & -7.85 & -6.29 & -7.78
& -5.27 & -5.86 & $\diagup$ & $\diagup$\\
\hline

\multirow{4}{*}{HS52}
& -9.42 &  -11.31 & -8.71 &  -9.06 & -9.22 &  -10.94
& -8.64 & -9.20 & -7.89 & -8.75\\
& -10.03 & -12.54 & -8.95 & -11.26 & -8.26 & -9.57
& -7.48 & -8.94 & -7.01 & -8.18\\
& -9.67 & -10.57 & $\diagup$ & $\diagup$ &  $\diagup$ & $\diagup$ & $\diagup$ &  $\diagup$ & $\diagup$ &  $\diagup$\\
& -8.54 & -9.00 & -9.32 & -12.21 & $\diagup$ & $\diagup$ & $\diagup$ &  $\diagup$ & $\diagup$ &  $\diagup$\\
\hline

\multirow{4}{*}{HS48}
& -9.74 & -10.45 & -9.50 &  -10.67 & -9.39 &  -11.29
& -8.86 & -10.26 & -7.86 & -8.93\\
& -9.41 & -10.59 & -9.09 &  -11.83 & -8.72 & -10.52
& -7.71 & -8.83 & -6.65 & -8.05\\
& -9.34 & -Inf & -1.25 & -8.06 & -0.74 & -Inf
& -0.63 & -Inf &  -0.75 &  -1.53\\
& -9.43 & -11.08 & -9.37 & -10.68 & -6.80 & -7.59
&  $\diagup$ & $\diagup$ &  $\diagup$ & $\diagup$\\
\hline

\multirow{4}{*}{S308NE} 
& $\diagup$ & $\diagup$ & $\diagup$ & $\diagup$ & $\diagup$ & $\diagup$  & $\diagup$ & $\diagup$  & $\diagup$ & $\diagup$\\
& $\diagup$ & $\diagup$ & $\diagup$ & $\diagup$ & $\diagup$ & $\diagup$ & $\diagup$ & $\diagup$  & $\diagup$ & $\diagup$\\
& $\diagup$ & $\diagup$ & $\diagup$ & $\diagup$ & $\diagup$ & $\diagup$ & $\diagup$ & $\diagup$  & $\diagup$ & $\diagup$\\
& $\diagup$ & $\diagup$ & $\diagup$ & $\diagup$ & $\diagup$ & $\diagup$ & $\diagup$ & $\diagup$  & $\diagup$ & $\diagup$\\
\hline

\multirow{4}{*}{HS42} 
& -9.78 & -10.63 & -9.47 & -11.14 & -9.50 &  -11.16
& -9.25 & -10.02 & -7.78 & -8.32\\
& -10.32 & -11.32 & -8.05 & -9.30 & -8.58 &  -10.34
& -8.07 & -9.85 & -7.21 & -8.89\\
& -9.46 & -11.31 & -10.08 & -11.18 & $\diagup$ & $\diagup$ & $\diagup$ & $\diagup$ & 1.29 & -0.43\\
& -9.33 & -12.35 & -9.58 & -11.19 & -5.98 & -6.17
& -5.21 & -Inf & $\diagup$ & $\diagup$ \\
\hline

\multirow{4}{*}{HS27} 
& -10.39 & -10.27 & -9.94 & -10.60 & -9.73 & -10.30
& -9.25 & -9.80 & -8.75 & -9.33\\
& -9.46 & -10.27 & -8.34 & -9.69 & -8.41 & -10.09
& -7.34 & -8.49 & -6.53 & -7.42\\
& -9.53 & -10.49 & -9.38 & -11.02 & -9.38 & -Inf
& -1.88 & -2.86 & -1.16 &  -3.44\\
& -9.35 & -10.78 & -9.53 & -10.60 & $\diagup$ & $\diagup$ & $\diagup$ & $\diagup$ & $\diagup$ & $\diagup$\\
\hline

\multirow{4}{*}{DIXCHLNG} 
& 9.95 & -9.54 & 9.95 & -5.63 & 9.95 &  -2.27 & 9.95 & -1.72 & 9.95 & -0.69\\
& $\diagup$ & $\diagup$ & $\diagup$ & $\diagup$ & $\diagup$ & $\diagup$ & $\diagup$ & $\diagup$ & $\diagup$ & $\diagup$\\
& -9.36 & -11.56 & $\diagup$ & $\diagup$ & $\diagup$ & $\diagup$ & $\diagup$ & $\diagup$ & $\diagup$ & $\diagup$\\
& $\diagup$ & $\diagup$ & $\diagup$ & $\diagup$ & $\diagup$ & $\diagup$ & $\diagup$ & $\diagup$ & $\diagup$ & $\diagup$\\
\hline

\multirow{4}{*}{COATINGNE} 
& $\diagup$ & $\diagup$ & $\diagup$ & $\diagup$ & $\diagup$ & $\diagup$ & $\diagup$ & $\diagup$ & $\diagup$ & $\diagup$\\
& $\diagup$ & $\diagup$ & $\diagup$ & $\diagup$ & $\diagup$ & $\diagup$ & $\diagup$ & $\diagup$ & $\diagup$ & $\diagup$\\
& $\diagup$ & $\diagup$ & $\diagup$ & $\diagup$ & $\diagup$ & $\diagup$ & $\diagup$ & $\diagup$ & $\diagup$ & $\diagup$\\
& $\diagup$ & $\diagup$ & $\diagup$ & $\diagup$ & $\diagup$ & $\diagup$ & $\diagup$ & $\diagup$ & $\diagup$ & $\diagup$\\
\hline

\multirow{4}{*}{HS28} 
& -9.41 & -10.72 & -9.65 & -10.25 & -9.69 & -10.51
& -9.07 & -9.63 & -8.82 & -9.91\\
& -9.30 & -11.32 & -8.94 &  -10.42 & -8.31 & -9.86
& -8.11 & -9.47 & -7.39 & -8.72\\
& $\diagup$ & $\diagup$ & $\diagup$ & $\diagup$ & $\diagup$ & $\diagup$ & $\diagup$ & $\diagup$ & $\diagup$ & $\diagup$\\
& -9.34 & -11.66 & -9.51 & -10.82 & -9.36 &  -11.61
& -5.35 & -7.07 & -4.69 & -Inf\\
\hline

\multirow{4}{*}{DEVGLA2NE} 
& $\diagup$ & $\diagup$ & $\diagup$ & $\diagup$ & $\diagup$ & $\diagup$ & $\diagup$ & $\diagup$ & $\diagup$ & $\diagup$\\
& $\diagup$ & $\diagup$ & $\diagup$ & $\diagup$ & $\diagup$ & $\diagup$ & $\diagup$ & $\diagup$ & $\diagup$ & $\diagup$\\
& $\diagup$ & $\diagup$ & $\diagup$ & $\diagup$ & $\diagup$ & $\diagup$ & $\diagup$ & $\diagup$ & $\diagup$ & $\diagup$\\
& $\diagup$ & $\diagup$ & $\diagup$ & $\diagup$ & $\diagup$ & $\diagup$ & $\diagup$ & $\diagup$ & $\diagup$ & $\diagup$\\
\hline

\multirow{4}{*}{BT3} 
& -9.86 & -10.40 & -9.48 & -10.77 & -9.44 & -10.84
& -8.67 & -9.40 & -8.01 &  -8.91\\
& -10.14 & -11.04 & -8.67 & -10.25 & -9.06 & -12.04
& -8.02 &  -9.18 & -6.66 & -7.91\\
& -9.54 & -11.05 & $\diagup$ & $\diagup$ & $\diagup$ & $\diagup$ & $\diagup$ & $\diagup$ & $\diagup$ & $\diagup$\\
& -9.51 & -11.02 & -9.33 & -12.09 & -6.39 & -7.87
& -6.80 & -Inf & $\diagup$ & $\diagup$\\
\hline

\multirow{4}{*}{HS79} 
& -9.66 & -10.32 & -9.62 & -10.95 & -9.38 & -10.39
& -9.14 & -10.41 & -7.90 &  -9.13\\
& -9.40 & -10.92 & -8.99 & -11.15 & -9.14 & -12.37
& -7.57 & -8.49 &  -6.62 & -7.94\\
& -9.49 & -11.02 & -2.88 & -5.61 & -1.59 & -2.76
& -0.80 & -2.51 & 0.52 & -1.16\\
& -9.55 & -10.85 & $\diagup$ & $\diagup$ & $\diagup$ & $\diagup$ & $\diagup$ & $\diagup$  & $\diagup$ & $\diagup$\\
\hline

\multirow{4}{*}{HS7} 
& -9.58 & -11.02 & -9.66 & -10.79 & -9.87 & -10.86
& -9.75 & -10.51 & -9.16 & -9.97 \\
& -9.44 & -10.30 & -9.10 & -10.97 & -7.70 & -8.94
& -8.34 & -9.70 &-6.94 & -7.75\\
& -9.46 & -10.49 & -9.65 & -11.31 & -9.70 & -10.31
& -9.39 & -Inf & 0.01 & -3.56\\
& -9.68 & -10.53 & -9.89 & -10.52 & -5.61 & -5.03
& -5.23 & -5.59& -3.90 & -4.63\\
\hline
			
\multicolumn{11}{l}{\tiny $\text{logR} = \log(\|\nabla\mL_k\|)$, $\text{logStd}$ is the standard deviation. The first line: AdapSQP; the second line: $\ell_1$ AdapSQP;}\\
\multicolumn{11}{l}{\tiny the third line: $\ell_1$ SQP; the fourth line: NonAdapSQP.}
			
\end{tabular}}
	
\end{table}

\begin{table}[!htp]
\centering
\caption{KKT results summary.}\label{tab:3}
\scalebox{0.8}{
\begin{tabular}{ c|c|c|c|c|c|c|c|c|c|c } 
\hline
\multirow{2}{*}{Problem}  & \multicolumn{2}{c|}{$\sigma^2 = 10^{-8}$} &  \multicolumn{2}{c|}{$\sigma^2 = 10^{-4}$} & \multicolumn{2}{c|}{$\sigma^2 = 10^{-2}$} & \multicolumn{2}{c|}{$\sigma^2 = 10^{-1}$} & \multicolumn{2}{c}{$\sigma^2 = 1$} \\	
& \multicolumn{1}{c}{logR} & logStd & \multicolumn{1}{c}{logR} & logStd & \multicolumn{1}{c}{logR} & logStd & \multicolumn{1}{c}{logR} & logStd & \multicolumn{1}{c}{logR} & logStd \\
\hline
\multirow{4}{*}{BT11}
& -9.45 & -10.81 & -9.34 &  -12.18 & -9.26 & -10.61 
& -8.97 & -10.11 & -7.43 & -8.21\\
& -9.55 &  -11.17 & -8.69 & -10.34 & -8.90 & -10.93
& -7.38 & -8.43 & -6.56 &  -7.73\\
& -9.57 & -11.63 & $\diagup$ & $\diagup$ & $\diagup$ & $\diagup$ & $\diagup$ & $\diagup$ & 1.63 & -1.80\\
& -9.50 & -10.66 & $\diagup$ & $\diagup$ & $\diagup$ & $\diagup$ & $\diagup$ & $\diagup$ & $\diagup$ & $\diagup$\\
\hline

\multirow{4}{*}{BT6} 
& -10.28 & -10.67 & -9.67 & -10.52 & -9.43 & -10.53
& -8.78 & -10.37 & -7.90 & -9.08\\
& -9.32 & -10.67 & -8.53 & -10.33 & -8.84 &  -10.62
& -7.58 & -9.51 & -6.29 &  -7.45\\
& -9.25 & -12.89 &  -2.52 & -4.61 & -0.87 & -3.07
& -0.43 & -2.36 & 3.83 & -1.24\\
& -9.30 & -12.03 & -9.37 & -Inf & $\diagup$ & $\diagup$ & $\diagup$ & $\diagup$ & $\diagup$ & $\diagup$\\
\hline

\multirow{4}{*}{HS40} 
& -10.04 & -10.84 & -9.86 &  -11.28 & -9.56 & -10.90
& -9.03 & -9.92 & -8.35 &  -8.89\\
& -10.08 & -11.07 & -8.76 & -10.25 & -8.66 & -10.37
& -8.36 & -9.65 & -7.15 & -8.66\\
& -9.80 & -11.23 & $\diagup$ & $\diagup$ & -0.32 & -2.58 & -0.31 & -3.31 & -0.17 & -1.59\\
& -9.57 & -10.79 & -9.45 & -11.06 & -6.38 & -6.89
& -5.70 & -6.42 & $\diagup$ & $\diagup$\\
\hline

\multirow{4}{*}{HS50} 
& -9.66 & -10.19 & -9.47 & -10.50 & -9.35 & -10.81
& -8.90 & -9.40 & -7.77 & -8.99\\
& -9.23 & -12.03 & -8.66 & -10.30 & -8.87 & -11.03
& -7.71 & -8.73 & -6.63 &  -7.97\\
& -6.19 &  -Inf & -5.07 &  -8.89 & -2.70 & -3.45
& $\diagup$ & $\diagup$ & $\diagup$ & $\diagup$\\
& -8.27 & -10.53 & $\diagup$ & $\diagup$ & $\diagup$ & $\diagup$ & $\diagup$ & $\diagup$ & $\diagup$ & $\diagup$\\
\hline

\multirow{4}{*}{HS26} 
& -9.14 & -10.97 & -8.92 & -9.45 & -8.57 & -9.14
& -7.67 & -7.87 & -8.04 & -7.68\\
& -8.64 & -10.23 & -7.39 & -8.51 & -6.69 &  -7.71
& -7.45 & -8.56 & -7.00 & -8.81\\
& 0.90 & 1.71 & 1.60 & 1.91 & 2.00 &  1.91 & 1.04 & 1.67 & 1.96 & 1.08\\
& -9.15 & -11.13 & -9.46 &  -11.01 & $\diagup$ & $\diagup$ & $\diagup$ & $\diagup$ & $\diagup$ & $\diagup$\\
\hline

\multirow{4}{*}{HS9} 
& -9.24 & -12.83 & -9.25 & -13.21 & -8.97 & -9.23
& -9.28 & -9.50 & -8.97 & -9.31\\
& -9.22 & -14.15 & -9.19 &   -13.61 & -7.39 & -8.47
& -7.43 & -8.47 & -6.25 &  -7.24\\
& $\diagup$ & $\diagup$ & $\diagup$ & $\diagup$ & $\diagup$ & $\diagup$ & $\diagup$ & $\diagup$ & $\diagup$ & $\diagup$\\
& -9.32 &  -12.08 & -9.60 & -10.61 & -9.70 & -10.39
& -9.58 & -10.47 & -10.10 & -10.50\\
\hline

\multirow{4}{*}{HS100LNP}
& -8.23 & -8.51 & -7.83 & -7.89 & -7.45 &  -8.08 & -6.98 & -7.63 & -5.10 &  -4.96\\
& -9.40 & -10.54 & -8.38 & -10.00 & -7.99 &  -9.42 & -5.98 &  -7.03 & -6.29 &  -8.07\\
& -9.44 &  -11.07 & 2.68 &  -7.11 & 2.68 & -4.91 & 2.69 & -3.63 & 2.71 & -2.29\\
& -6.30 &   -6.42 & $\diagup$ & $\diagup$ & $\diagup$ & $\diagup$ & $\diagup$ & $\diagup$ & $\diagup$ & $\diagup$\\
\hline

\multirow{4}{*}{HS77} 
& -9.59 & -10.55 & -9.67 & -11.12 & -9.37 & -10.98
& -8.77 & -10.64 & -7.65 & -8.60\\
& -9.97 &  -11.32 & -8.57 & -10.12 & -9.10 & -11.97
& -7.20 & -8.41 & -6.34 &  -7.60\\
& -9.23 & -13.27 & -2.57 &  -4.53 & -0.63 & -1.78
& -0.02 &  -1.33 & 3.78 & -0.86\\
& -9.31 & -12.54  & $\diagup$ & $\diagup$ & $\diagup$ & $\diagup$ & $\diagup$ & $\diagup$ & $\diagup$ & $\diagup$\\
\hline

\multirow{4}{*}{MWRIGHT}
& -7.71 & -8.11 & -6.24 & -5.98 & -6.09 & -5.80
& -6.80 & -6.54 & -6.41 & -6.76\\
& -9.89 & -11.04 & -8.82 & -10.77 & -5.97 &  -6.90
& -6.20 &  -7.26 & -5.74 & -7.01\\
& -9.47 & -10.77 & $\diagup$ & $\diagup$ & $\diagup$ & $\diagup$ & $\diagup$ & $\diagup$ & 2.76 &  -Inf\\
& -8.14 &  -8.92 & $\diagup$ & $\diagup$ & $\diagup$ & $\diagup$ & $\diagup$ & $\diagup$ & $\diagup$ & $\diagup$\\
\hline

\multirow{4}{*}{HS46} 
& -9.22 & -13.26 & -9.21 & -12.64 & -9.23 &  -11.95
& -8.86 &  -9.74 & -7.68 &  -8.15\\
& -9.21 & -16.47 & -9.11 &   -12.35 & -8.00 & -9.39
& -7.48 & -8.63 & -6.91 & -8.28\\
& -9.23 & -13.49 & -2.51 &  -6.25 & -2.34 & -Inf
& -1.64 & -2.98 & 0.98 &  -1.30\\
& -9.22 & -14.39 & -9.38 &    -Inf & $\diagup$ & $\diagup$  & $\diagup$ & $\diagup$ & $\diagup$ & $\diagup$\\
\hline

\multirow{4}{*}{HS49} 
& -8.55 & -10.23 & -7.53 & -7.99 & -7.30 & -7.67
& -6.79 & -6.96 & -7.15 & -7.70\\
& -8.09 &  -9.81 & -6.56 & -7.48 & -6.46 & -7.64
& -6.96 & -8.32 & -6.33 &  -8.06\\
& $\diagup$ & $\diagup$ & $\diagup$ & $\diagup$ & $\diagup$ & $\diagup$ & $\diagup$ & $\diagup$ & $\diagup$ & $\diagup$\\
&  -9.23 & -13.84 & $\diagup$ & $\diagup$ & $\diagup$ & $\diagup$ & $\diagup$ & $\diagup$ & $\diagup$ & $\diagup$\\
\hline

\multirow{4}{*}{HS78} 
& -9.85 & -10.34 & -9.49 &  -11.26 & -9.35 & -11.36
& -9.23 & -10.53 & -7.69 &  -8.26\\
& -9.58 &  -14.76 & -8.32 &  -9.63 & -8.88 &  -10.90
& -7.76 & -9.02 & -6.54 &  -8.28\\
& -9.46 & -10.96 & -3.00 &   -4.62 & -0.63 & -1.88
& 0.62 & -0.72 & 1.40 &   -1.34\\
& -9.58 & -10.72 & -7.09 &  -7.90 & -5.91 & -6.88
& -6.14 & -Inf & $\diagup$ & $\diagup$ \\
\hline

\multirow{4}{*}{HS56} 
& -9.30 &  -11.46 & -9.08 & -10.47 & -8.23 &  -8.12
& -7.10 & -7.44 &  -6.71 & -7.11\\
& -9.47 & -10.63 & -8.94 & -11.15 & $\diagup$ & $\diagup$ & $\diagup$ & $\diagup$ & $\diagup$ & $\diagup$\\
& -9.41 & -11.00 & -0.37 & -6.81 & 0.14 & -2.69
& 0.77 & -1.83 & 1.11 &  -0.86\\
& -8.48 & -9.53 & $\diagup$ & $\diagup$ & $\diagup$ & $\diagup$ & $\diagup$ & $\diagup$ & $\diagup$ & $\diagup$\\
\hline

\multirow{4}{*}{HS111LNP} 
&-8.58 & -9.93 & -7.22 & -8.30 & -7.03 &  -7.75
& -6.80 & -7.81 & -6.19 & -6.38\\
& -8.42 & -10.09 & -6.53 &  -7.66 & -5.79 & -6.78
& -5.63 &  -6.79 & -5.91 &  -7.33\\
& $\diagup$ & $\diagup$ & $\diagup$ & $\diagup$ & $\diagup$ & $\diagup$ & $\diagup$ & $\diagup$ & $\diagup$ & $\diagup$\\
& -9.24 & -12.42 & -2.82 &  -6.98 & $\diagup$ & $\diagup$ & $\diagup$ & $\diagup$ & $\diagup$ & $\diagup$\\
\hline

\multirow{4}{*}{HS47} 
& -9.22 & -13.80 & -9.08 & -10.13 & -9.29 & -11.33
& -8.36 & -9.19 & -7.88 &  -9.21\\
& -9.21 &  -11.86 & -7.82 & -8.95 & -8.35 &  -9.71
& -7.13 & -8.75 & -6.24 &  -7.55\\
& -7.87 &  -10.99 & -5.62 &  -Inf & -3.00 & -Inf
& -1.01 &  -3.94 & 1.56 &   -1.48\\
& -7.97 &  -9.43  & -9.23 & -Inf & $\diagup$ & $\diagup$ & $\diagup$ & $\diagup$& $\diagup$ & $\diagup$  \\
\hline

\multicolumn{11}{l}{\tiny $\text{logR} = \log(\|\nabla\mL_k\|)$, $\text{logStd}$ is the standard deviation. The first line: AdapSQP; the second line: $\ell_1$ AdapSQP;}\\
\multicolumn{11}{l}{\tiny the third line: $\ell_1$ SQP; the fourth line: NonAdapSQP.}

\end{tabular}}
	
\end{table}

\bibliographystyle{spbasic}
\bibliography{ref}

\begin{flushright}
\scriptsize \framebox{\parbox{4.5in}{Government License: The submitted manuscript has been created by UChicago Argonne, LLC, Operator of Argonne National Laboratory (``Argonne"). Argonne, a U.S. Department of Energy Office of Science laboratory, is operated under Contract No. DE-AC02-06CH11357.  The U.S. Government retains for itself, and others acting on its behalf, a paid-up nonexclusive, irrevocable worldwide license in said article to reproduce, prepare derivative works, distribute copies to the public, and perform publicly and display publicly, by or on behalf of the Government. The Department of Energy will provide public access to these results of federally sponsored research in accordance with the DOE Public Access Plan. http://energy.gov/downloads/doe-public-access-plan. }}
\normalsize
\end{flushright}

\end{document}